\documentclass[a4paper, 12pts]{article}
\usepackage{amssymb, url, amsthm, amsmath, undertilde, scalerel, stmaryrd, enumitem, lipsum, tikz , tikz-cd, sectsty, filecontents, enumitem}

\usepackage[nottoc,notlof,notlot]{tocbibind} 

\newcommand{\lusim}[1]{\smash{\underset{\raisebox{1.2pt}[0cm][0cm]{$\sim$}}
		{{#1}}}}
\usepackage{geometry}[in=1]

\chaptertitlefont{\LARGE}

\theoremstyle{plain}
\newtheorem{thm}{Theorem}[section]
\newtheorem{claim}[thm]{Claim}
\newtheorem{lemma}[thm]{Lemma}
\newtheorem{corollary}[thm]{Corollary}
\newtheorem{remark}[thm]{Remark}

\newtheorem{defn}[thm]{Definition}

\newtheorem*{remark*}{Remark}
\newtheorem*{lemma*}{Lemma}
\newtheorem*{claim*}{Claim}
\newtheorem*{thm*}{Theorem}

\title{The Magidor Iteration and Restrictions of Ultrapowers to the Ground Model}

\author{Eyal Kaplan}

\begin{document}
\maketitle
  
\begin{abstract}
We study the Magidor iteration of Prikry forcings below a measurable limit of measurables $ \kappa $. We first characterize all the normal measures $ \kappa $ carries in the generic extension, building on and extending the main result of \cite{ben2014forcing}. Then, for every such normal measure, we prove that the restriction of its ultrapower, from the generic extension to the ground model, is an iterated ultrapower of $ V $ by normal measures. This is done without core model theoretic assumptions; $ \mbox{GCH}_{\leq \kappa} $ in the ground model suffices.
\end{abstract}  
  
\section*{Introduction}

In this paper we revisit the Magidor iteration of Prikry forcings, which was first introduced in  \cite{MAGIDOR197633}. Let $ \kappa $ be a measurable limit of measurables. The Magidor iteration can be used to destroy the measurability of every measurable cardinal $ \alpha<\kappa $, while preserving cardinals and the measurability of $ \kappa $ itself. Given such an iteration $ P $ and a generic set $ G\subseteq P $ over the ground model $ V $, we consider the following questions:

\begin{enumerate}
	\item What are the normal measures on $ \kappa $ in $ V\left[G\right] $?
 	\item Given a normal measure $ W\in V\left[G\right] $ on $ \kappa $, let $ j_W\colon V\left[G\right]\to M\left[H\right] $ be its ultrapower embedding. Is $ j_W\restriction_{V} $ an iteration of $ V $ (by its measures or extenders)? 
 	\item Given a normal measure $ W\in V\left[G\right] $, is $ j_W\restriction_{V} $ a definable class of $ V $?
\end{enumerate}

The first question was extensively studied by O. Ben-Neria in \cite{ben2014forcing}. For every normal measure $ U\in V $ on $ \kappa $, he assigned a corresponding measure $ U^{\times}\in V\left[G\right] $ on $ \kappa $, and showed that the mapping $ U\mapsto U^{\times} $ is a bijection between the set of normal measures on $ \kappa $ in $ V $, and the set of normal measures on $ \kappa $ in $ V\left[G\right] $. This was done under the assumption that the ground model $ V $ is the core model and $ 0^{\P} $ does not hold. In this paper, we extend this result, weakening the assumption on the ground model $ V $:

\begin{thm} \label{Thm: Characterization of normal measures}
Assume $ \mbox{GCH}_{\leq \kappa} $ holds in $ V $. Let $ W\in V\left[G\right] $ be a normal measure on $ \kappa $. Then $ W = U^{\times} $ for some normal measure $ U\in V $ on $ \kappa $. Moreover, the measures $ \langle U^{\times} \colon U\in V \mbox{ is a normal measure on } \kappa \rangle $  are pairwise distinct.
\end{thm}

The proof relies on the same methods of Ben-Neria in \cite{ben2014forcing}; the core-model theoretic aspects of the argument are replaced by the tools developed in \cite{RestElm}. 

The second question is answered affirmatively - for every forcing notion $ P $, not only Magidor iterations - under the assumption that there is no inner model with a Woodin cardinal, and the ground model is the core model (see  \cite{schindler2006iterates}). We prove that this remains true for the Magidor iteration in much more general settings:

\begin{thm} \label{Thm: Structure of j_Wrestriction V}
Assume $ \mbox{GCH}_{\leq \kappa} $ holds in $ V $. Let $ W\in V\left[G\right] $ be a normal measure on $ \kappa $. Then $ j_W\restriction_{V} $ is an iterated ultrapower of $ V $ by normal measures. 
\end{thm}
Moreover, a concrete description of $ j_W\restriction_{V} $ as an iterated ultrapower is given. This uses and extends ideas appearing in \cite{NonStatRestElm}, where iterations of Prikry forcings were considered under the simpler nonstationary support.

The answer to the third question depends on the choice of the normal measures used along the iteration to singularize the measurables of $ V $ below $ \kappa $. In general, $ j_W\restriction_{V} $ may not be definable in $ V $ (see, e.g., section 5.2 in \cite{RestElm}). In lemma \ref{Lemma: FS, sufficient condition for definability} we provide a sufficient condition for definability of $ j_W\restriction_{V} $ as a class of $ V $.

This paper is organized as follows: In the first section we present the forcing and its basic properties. In section $ 2 $ we prove theorem \ref{Thm: Characterization of normal measures}. In section $ 3 $ we prove theorem \ref{Thm: Structure of j_Wrestriction V}, provide a sufficient condition for definability of $ j_W\restriction_{V} $ as a class of $ V $, and completely  describe the Prikry sequences added to measurables of $ M $ above $ \kappa $ in $ H $.

\section{The Forcing}
\begin{defn}
An iteration $ \langle P_\alpha, \lusim{Q}_\beta \colon \alpha\leq\kappa\ ,\ \beta <\kappa \rangle $ is called a full support (Magidor) iteration of Prikry-type forcings if and only if, for every $ \alpha\leq\kappa $ and $p\in P_\alpha$,
\begin{enumerate}
\item  $ p $ is a function with domain $ \alpha $ such that for every $\beta <\alpha$, $p\restriction \beta  \in P_\beta$, $p\restriction \beta \Vdash p(\beta) \in \lusim{Q}_\beta \mbox{ and }  \langle \lusim{Q}_{\beta}, \lusim{\leq}_{\lusim{Q}_{\beta} }, \lusim{\leq}^*_{\lusim{Q}_{\beta} } \rangle \ \mbox{ is a Prikry-type forcing.}$ . 
\item There exists a finite set $ b\subseteq \kappa $ such that for every $ \beta\notin b $, $ p\restriction_{ \beta}\Vdash  p(\beta) \lusim{\geq}^{*}_{\beta} \lusim{0}_{Q_{\beta}} $, where $ \lusim{\geq}^{*}_{\beta} $ is the direct extension order of $ \lusim{Q}_{\beta} $ (after we define the order $ \leq_{P} $ on  $ P $, this will be abbreviated to $ p\geq_{P} 0_{P} $).
\end{enumerate}

Suppose that $p,q \in P_\alpha$. Then $p\geq q$, which means that $p$ extends $q$, holds if and only if:
\begin{enumerate}
\item For every $\beta \leq \alpha$, $p\restriction \beta \Vdash p(\beta)\geq_\beta q(\beta)$ (where $ \geq_\beta$ is the order of $ Q_\beta $).
\item There is a finite subset $b\subseteq \alpha$, such that for every $\beta \in \alpha\setminus b$, $p\restriction \beta \Vdash p(\beta)\geq^{*}_\beta q(\beta)$ (where $ \geq^{*}_{\beta} $ is the direct extension order of $ Q_\beta $).
\end{enumerate} 
If $b=\emptyset$, we say that $p$ is a direct extension of $q$, and denote it by $p\geq^* q$. 
\end{defn}

Let $ \langle P_\alpha, \lusim{Q}_\beta \colon \alpha\leq\kappa\ ,\ \beta <\kappa \rangle $ be a full support iteration of Prikry forcings, such that, for every $V$-measurable cardinal, $ \alpha $, $ \lusim{Q}_{\alpha} $ is non-trivial, and is forced to be Prikry forcing with a given $ P_{\alpha} $-name for a normal measure on $ \alpha $. If $ \alpha  $ is not measurable in $ V $, $ \lusim{Q}_{\alpha} $ is the trivial forcing.

\textbf{Notations.} Denote $ \Delta = \{ \alpha<\kappa \colon \alpha \mbox{ is measurable in }V \} $. For every $ \alpha\in \Delta $, let $ \lusim{W}_{\alpha} $ be the $ P_{\alpha} $-name for a normal measure on $ \alpha $, which is forced by $ P_{\alpha} $ to be the measure used in the Prikry forcing  $ \lusim{Q}_{\alpha} $. Assume that $ p\in P_{\kappa} $ is a given condition and $ \alpha\in \Delta $. We denote by $ \lusim{t}^{p}_{\alpha} $ and $ \lusim{A}^{p}_{\alpha} $ the $ P_{\alpha} $-names such that $ p\restriction_{\alpha} \Vdash p(\alpha) = \langle \lusim{t}^{p}_{\alpha}, \lusim{A}^{p}_{\alpha} \rangle $. In $ V\left[G\right] $, we denote by $ d\colon \Delta \to \kappa $ the function which maps each former measurable in $ \Delta $ to the first element in its Prikry sequence. Finally, we adopt the following useful notation, introduced by O. Ben-Neria in \cite{ben2014forcing}: Given a condition $ p\in P_{\kappa} $ and $ \alpha< \kappa $, let $ p^{-\alpha} \geq^* p $ be the condition $ p^* $ which satisfies, for every measurable $ \xi>\alpha $, 
	$$ p^* \restriction_{\xi} \Vdash \lusim{A}^{p^*}_{\xi} = \lusim{A}^{p}_{\xi} \setminus \left(\alpha+1\right) $$

The following lemma is standard (see \cite{gitik2010prikry} for example):

\begin{lemma}
$ P=P_{\kappa} $ satisfies the Prikry property.
\end{lemma}

\begin{lemma}[Fusion Lemma] \label{Lemma: FS, fusion lemma}
Let $\delta\leq \kappa$ be a limit ordinal and $ p\in P_{\delta} $. For every $ \alpha<\delta $, let $ e(\alpha) $ be a $ P_{\alpha} $-name such that--
\begin{align*}
p\restriction_{\alpha}\Vdash & \mbox{"} e(\alpha) \mbox{ is a dense open subset of  } P\setminus \alpha \mbox{ above } p\setminus\alpha \mbox{, } \\
& \mbox{ with respect to the direct extension order."}
\end{align*}
Then there exist $p^* \geq^{*} p$ such that for every $ \alpha \in \left[\nu,  \delta\right)$,
$$ p^* \restriction_{\alpha} \Vdash \left(p^*\setminus \alpha\right)^{-\alpha} \in e(\alpha)$$
\end{lemma}
\begin{proof}
Define a sequence $ \langle p_{\xi} \colon \xi >\delta \rangle $ of direct extensions of $ p $, such that for every $ \eta<\xi<\delta $,
\begin{enumerate}
\item $ p\leq^* p_{\eta}\leq^* \left(p_{\xi}\right)^{-\eta} $.
\item $ p_{\eta}\restriction_{\eta} = p_{\xi}\restriction_{\eta} $.
\end{enumerate}
Take $ p_0 = p $. Assume that $ \xi = \xi'+1 $ is successor, and let us define $ p_{\xi} $. Take $ p_{\xi}\restriction_{\xi} = p_{\xi'}\restriction_{\xi} $. Take $ p_{\xi}\setminus {\xi} \geq^* p_{\xi'} \setminus {\xi}$ be a $ P_{\xi} $--name for a direct extension which belongs to $ e(\xi) $. 

Assume that $ \xi <\delta$ is limit, and let us define $ p_{\xi} $. First, set--
$$ p_{\xi}\restriction_{\xi}  = \bigcup_{\xi'<\xi} p_{\xi'}\restriction_{\xi'} $$
We now define a $ P_{\xi} $-name for a condition $ r\in P\setminus \xi $. If $ \xi $ is non-measurable, $ r(\xi) $ is trivial. If it is: Let $ \lusim{t} $ be a $ P_{\xi} $-name, and, for every $ \xi'<\xi $, take $ P_{\xi} $-names $ \lusim{A}_{\xi'} $ such that $ p_{\xi}\restriction_{\xi} \Vdash p_{\xi'}(\xi) = \langle \lusim{t}, \lusim{A}_{\xi'} \rangle $. Set $ r(\xi) = \langle \lusim{t}, \triangle_{\xi'<\xi} \lusim{A}_{\xi'} \rangle $. Finally, let $ r\setminus  \left(\xi+1\right) $ be a direct extension of all the conditions $ \langle p_{\xi'} \setminus \left(\xi+1\right) \colon \xi'<\xi \rangle $ (the direct extension order above $ \xi $ is more than $ \xi $-closed). This defines $ r $. Let $ p_{\xi}\setminus \xi \geq^* r $ be a direct extensions which belongs to $ e(\xi) $. Note that $ \lusim{A}^{p_{\xi}}_{\xi} \subseteq \triangle_{\xi'<\xi} \lusim{A}^{p_{\xi'}}_{\xi} $, and thus, for every $ \eta<\xi $, $ \lusim{A}^{p_{\xi}}_{\xi} \setminus \left(\eta+1\right) \subseteq \lusim{A}^{p_{\eta}}_{\xi} $. Thus $ p_{\eta} \leq^* \left(p_{\xi}\right)^{-\eta} $.

This finishes the construction. Define $ p^* = \bigcup_{\xi<\delta} p_{\xi}\restriction_{\xi}$. We claim that $ p^* $ is as desired. Let $ \alpha<\delta $. Then $ p^*\restriction_{\alpha} = p_{\alpha}\restriction_{\alpha} $. Thus, this condition forces that $ \left(p_{\alpha}\setminus \alpha\right)^{-\alpha} \in e(\alpha) $. It also forces that $ \left(p^*\setminus \alpha\right)^{-\alpha} $ direct extends $ \left(p_{\alpha}\setminus \alpha\right)^{-\alpha} $, and thus it belongs to $ e\left( \alpha \right) $, as desired.
\end{proof}

%

\begin{lemma} \label{Lemma: FS, P preserves cardinals}
$ P = P_{\kappa} $ preserves cardinals.
\end{lemma}

\begin{proof}
We prove by induction that for every $ \delta\leq \kappa $, $ P_{\delta} $ preserves cardinals. This is clear for successor values of $ \delta $. By $ \mbox{GCH}_{\leq \kappa} $, this is clear as well if $ \delta $ is not a limit of measurables. Thus, let us assume that $ \delta \leq \kappa$ is a limit of measurables and $ \mu $ is a cardinal. If $ \mu<\delta $, factor $ P_{\delta} = P_{< \mu}*\lusim{Q}_{\mu}*P_{>\mu} $. Since the direct extension order of $ P_{>\mu} $ is more than $ \mu $-closed, it preserves $ \mu $; $ Q_{\mu} $ preserves $ \mu $  because it is either trivial or a Prikry forcing; finally, by induction, $ P_{<\mu} $ preserves $ \mu $. If $ \mu = \delta $, then $ \mu $ is a limit of measurables, each of them is preserved by induction. If $ \mu> \delta^{+} $, $ \mu $ is preserved since $ \left|P_{\delta}\right| = \delta^{+} $ by $ \mbox{GCH}_{\leq \kappa} $. Thus, it suffices to prove that $ P_{\delta} $ preserves $ \delta^{+} $ for every limit of measurables $ \delta $. It suffices to prove that $ P_{\delta} $ has the $ \delta^{+}-c.c. $: For any antichain $ A\subseteq P_{\delta} $ of cardinality $ \delta^{+} $, there exists a subset $ A'\subseteq A $ of cardinality $ \delta^{+} $, such that the following holds: There exists a finite set $ b\subseteq \delta $, and, for every $ \alpha\in b $, a $ P_{\alpha} $-name for a finite increasing sequence $ \lusim{t}_{\alpha}\in  \left[ \alpha \right]^{<\omega} $, such that--
$$\forall p\in A' \ \forall \beta\in \delta\setminus b \ , p\restriction_{ \beta}\Vdash p(\beta) \geq^* 0_{Q} $$
and--
$$\forall p\in A' \ \forall \alpha\in b \ \exists \lusim{A} , p\restriction_{ \alpha} \Vdash p(\alpha) = \langle \lusim{t}_{\alpha},\lusim{A} \rangle$$
Given these properties, every pair of conditions in $ A' $ are compatible, which is a contradiction.
\end{proof}

\begin{lemma} \label{Lemma: FS, no fresh subsets of kappa and kappa+}
$ P= P_{\kappa} $ doesn't add fresh subsets of $\kappa, \kappa^{+} $.
\end{lemma}

The above lemma is proved, for example, in \cite{RestElm}. We remark that  this proof uses the fact that some normal measure on $ \kappa $ in $ V $ extends to a normal measure in $ V\left[G\right] $, and this is indeed the case (this is well known, and in any case, will be proved in the next section in lemma \ref{Lemma: FS, extending U to U*}. The proof will not rely on the current lemma or its consequences).

In \cite{RestElm} it is proved that, if a forcing notion $ P $ preserves cardinals and does not add fresh subsets to cardinals in the interval $ \left[\kappa, \left( 2^{\kappa} \right)^V\right] $, then every $ \kappa $-complete ultrafilter in the generic extension extends a $ \kappa $-complete ultrafilter of $ V $. Since we assume $ \mbox{GCH}_{\leq \kappa} $, the following follows:

\begin{corollary} \label{Corollaery: Every measure in the generic extension extends a measure of V}
Let $ G\subseteq P_{\kappa} $ be generic over $ V $, and let $ W\in V\left[G\right] $ be a $ \kappa $-complete ultrafilter on $ \kappa $. Then $ W\cap V \in V $.
\end{corollary}

\begin{lemma} \label{Lemma: every name for an ordinael can be decided up to boudedly many values by a direct extension}
Let $ \delta \leq \kappa $ be an inaccessible cardinal. Let $ p\in P_{\delta} $ and assume that $ \lusim{\alpha} $ is a $ P_{\delta} $-name for an ordinal. Then there exists $ p^*\geq^* p $ and a set $ A\in V $ with $ \left|A\right|<\delta $ such that $ p^*\Vdash \lusim{\alpha}\in A $.
\end{lemma}

\begin{proof}
Denote by $ D $ the dense open subset of $ P_\delta $ which consists of conditions which decide the value of $ \lusim{\alpha }$. We will apply on $ D $ the following claim:
\begin{claim}
	Let $ \delta \leq \kappa $ be a limit ordinal and let $ D\subseteq P_{\delta} $ be a dense open subset of $ P_{\delta} $. Assume that $ p\in P_{\delta} $. Then there exists $p^* \geq^{*} p$ such that for every $p^* \leq q\in D $, 
	$${q\restriction_{\gamma+1} }^{\frown} \left(p^* \setminus \left( \gamma+1 \right)\right)^{-\left(\gamma+1\right)} \in D$$
	where $\gamma $ is the maximal coordinate which satisfies-- $$ q\restriction_{\gamma} \Vdash \mbox{"} q(\gamma) \mbox{ is not a direct extension of } p^*(\gamma) \mbox{"}$$
	(and, if such $ \gamma $ does not exist, then $ \gamma = 0 $).
\end{claim}

\begin{proof}
	Fix a non-measurable $ \alpha<\delta $ and $ G_{\alpha}\subseteq P_{\alpha} $ generic over $ V $ such that $ p\restriction_{\alpha} \in G_{\alpha} $. Given $p\restriction_{\alpha}\leq  q\in G_{\alpha} $, we define a subset of $ P\setminus \alpha $ which is $ \leq^* $-dense open above $ p\setminus \alpha$:
	$$ e_q(\alpha) = \{ r\in P\setminus \alpha  \colon   q^{\frown} r\in D \mbox{ or }  \left( \forall r'\geq^* r,  \  q^{\frown} r' \notin D  \right)  \} $$
	
	Since $ \alpha $ is non-measurable, the direct extension order of $ P\setminus \alpha $ is more than $ \left| G_{\alpha}\right|^{+} $-distributive. Let $ e(\alpha) $ be a $ P_{\alpha} $-name for the set--
	$$ e(\alpha) = \bigcap_{q\in G_{\alpha}}  e_{q}(\alpha) $$
	then $ p\restriction_{\alpha} $ forces that $ e(\alpha) $ is $ \leq^* $-dense open above $ p\setminus \alpha $.

	Apply lemma \ref{Lemma: FS, fusion lemma}. Let $ p^* \geq^* p $ be such that, for every non-measurable $ \alpha<\delta $, $$ p^* \restriction_{ \alpha } \Vdash \left(p^* \setminus \alpha\right)^{-\alpha} \in e(\alpha) $$
	Assume now that $ p^* \leq q \in D $. Let $ \gamma $ be as in the formulation of the claim. Then $ \gamma+1 $ is not measurable, so--
	$$ p^* \restriction_{ \gamma+1 } \Vdash \left(p^*\setminus \left( \gamma+1 \right)\right)^{-\gamma+1} \in e(\gamma+1) $$
	In particular,
	$$ q \restriction_{ \gamma+1 } \Vdash \left(p^*\setminus \left( \gamma+1 \right)\right)^{-\gamma+1} \in e(\gamma+1) $$
	
	Finally, since there exists a direct extension $ r' = q\setminus \left( \gamma+1 \right) \geq^* p^*\setminus \left(\gamma+1 \right) $ such that $ {q\restriction_{\gamma+1}}^{\frown} r' \in D $, it follows that ${q\restriction_{\gamma+1}}^{\frown}  \left(p^*\setminus \left( \gamma+1\right)\right)^{-\gamma+1}\in D$, as desired.
\end{proof}

Pick a direct extension $ q\geq^* p $, by applying the claim on the set $ D $ of conditions deciding the value of $ \lusim{\alpha} $. We will construct below a direct extension $ q^* \geq^* q $; After this is done, we will prove that $ q^* $ has a direct extension $ p^* \geq^* q^* $ as desired in the lemma. Namely, $ p^* $ satisfies that for some set of ordinals $ A $ with $ \left|A\right| <\delta$, $ p^*\Vdash \lusim{\alpha}\in A $.

First, let us construct $ q^* \geq^* q $. Assume that $ \gamma<\delta $, and $ q^*\restriction_{\gamma} $ has been defined. To define $ q^*\left( \gamma \right) $, we shrink the set $ \lusim{A}^{q}_{\gamma} $. We shrink it to a set $ A\in W_{\gamma} $, such that, for every $ n<\omega $, exactly one of the following holds: Either for every $ s\in \left[ A \right]^n $, there exists a set of ordinals $ A_{s} $ with $ \left|A_s\right|<\delta $, such that--
$$ \langle {t^{q}_{\gamma}}^{\frown} s,  A\setminus \max\left(s\right) \rangle^{\frown} \left( q\setminus \left(\gamma+1\right) \right)^{-\gamma+1} \Vdash \lusim{\alpha}\in A_{s} $$
or, there is no such $ s $. 

This results in a direct extension $ q^*\geq^* q $. It suffices to prove that $ q^* $ has a direct extension $ p^{*} $ which belongs to $ D $. Assume otherwise. Let $ r\geq q^* $ be a condition in $ D $, which is chosen with the least number of non-direct extensions. Let $ \gamma  $ be the maximal coordinate in which a non-direct extension was taken in the extension $ r\geq q^* $. Clearly $ r\geq q $, and in this extension, as well, $ \gamma  $ is the maximal coordinate in which a non-direct extension is taken. Thus, by the choice of $ q $, 
$$  {r\restriction_{\gamma+1}}^{\frown} \left( q\setminus \left(\gamma+1\right) \right)^{-\gamma+1}\in D $$
Let $ n<\omega $ be such that $ r\restriction_{\gamma} $ forces that $ \mbox{lh}\left(  t^{r}_{\gamma} \right) = n+\mbox{lh}\left( t^{q}_{\gamma} \right) $. Then $ r\restriction_{\gamma } $ forces that for every $ s\in \left[\lusim{A}^{r}_{\gamma}\right]^{n} $, there exists a set $ A_{s} $ with $ \left|A_{s}\right|<\delta $, such that--
$$  {    \langle{ t^{q}_{\gamma} }^{\frown} s, \lusim{A}^{r}_{\gamma}\setminus \max\left( s \right) \rangle  }^{\frown}  \left(  q\setminus \left( \gamma+1 \right)  \right)^{-\gamma+1} \Vdash \lusim{\alpha}\in A_{s}  $$
By taking union on the possible values of the sets $ A_s $ as above, there exists a set $ A\in V $ with $ \left|A\right|<\delta $ such that--

$$  { r\restriction_{\gamma} }^{\frown} {\langle \lusim{t}^{q}_{\gamma}, \lusim{A}^{r}_{\gamma}  \rangle}^{\frown} \left(   q\setminus \left( \gamma+1 \right)   \right)^{-\gamma+1} \Vdash \lusim{\alpha}\in A  $$
and this contradicts the minimality of the number of non-direct extensions in the choice of $ r\geq q^* $.
\end{proof}

\begin{corollary} \label{Lemma: FS, evaluating new functions by old ones}
Assume that $ \delta \leq \kappa $ is inaccessible, $ p\in P_{\delta} $ and let $ \lusim{f} $ be a $ P_{\delta} $-name for a function from $ \delta $ to the ordinals. Then there exists $ p^*\geq^* p $ and a function $ F\colon \delta \to \left[ \mbox{Ord} \right]^{<\delta} $ in $ V $, such that for every $ \alpha<\delta $, 
$$  \left(p^*\right)^{-\alpha} \Vdash \lusim{f}(\alpha)\in F\left( \alpha \right) $$
\end{corollary}

\begin{proof}
For every $ \alpha<\delta $, set--
\begin{align*}
e\left( \alpha \right) = \{  r\in P\setminus \alpha \colon \mbox{there exists }  A\subseteq \mbox{Ord with } \left|A\right|<\delta \mbox{ such that } r\Vdash \lusim{f}(\alpha)\in A  \}
\end{align*}
by lemma \ref{Lemma: every name for an ordinael can be decided up to boudedly many values by a direct extension}, $ e(\alpha) $ is $ \leq^* $-dense open. Thus, by Fusion, there exists $ p^*\geq^* p $ such that for every $ \alpha<\delta $,
$$  {p^*}\restriction_{\alpha} \Vdash \mbox{there exists }  A_{\alpha}\subseteq \mbox{Ord with } \left|A_{\alpha}\right|<\delta \mbox{ such that } \left(  p^*\setminus \alpha \right)^{-\alpha} \Vdash \lusim{f}(\alpha)\in A_{\alpha} $$
Finally, for every $ \alpha<\delta $, let $ F(\alpha) = \{ \beta  \colon \exists q\geq p^*\restriction_{\alpha}, \ r\Vdash \beta \in \lusim{A}_{\alpha}  \} $. Then $ \left( p^* \right)^{-\alpha}\Vdash \lusim{f}(\alpha)\in F(\alpha) $ and $ \left|F(\alpha)\right|<\delta $, as desired.
\end{proof}

\section{Normal Measures in the Generic Extension}

This section is devoted to the proof of theorem \ref{Thm: Characterization of normal measures}. The same result was first observed by O. Ben-Neria in \cite{ben2014forcing}, assuming that $ V $ is the core model and there is no inner mode with overlapping extenders. We will reduce the assumptions on $ V $ to $ \mbox{GCH}_{\leq \kappa} $.  

Throughout this section, we will extensively use arguments and notations introduced in \cite{ben2014forcing}: For every normal measure on $ \kappa $, $ U\in V $, we will define a measure $ U^*\in V\left[G\right] $ which extends $ U $. It will  turn out that $ U^* $ is normal if and only if $ o(U)=0 $. Let $ U^{\times} $ be the normal measure below $ U^{*} $ in the Rudin-Keisler order. We will prove that every normal measure on $ \kappa $ in $ V\left[G\right] $ has the form $ U^{\times} $ for some $ U\in V $.  Moreover, we simultaneously prove the following:

\begin{lemma} \label{Lemma: Every normal measure concentrates on the complement of the set of unions}
Let $ W $ be a normal measure on $ \kappa  $ in $ V\left[G\right] $. Then--
$$ \kappa\setminus \bigcup_{\alpha\in \Delta \cap \kappa} \left(d(\alpha), \alpha \right] $$
\end{lemma}

Working by induction, we assume that theorem \ref{Thm: Characterization of normal measures} and lemma \ref{Lemma: Every normal measure concentrates on the complement of the set of unions} hold in every generic extension $ V^{P_{\alpha}} $ where $ \alpha< \kappa $ is measurable (where $ \alpha $ replaces $ \kappa $ in their formulation). We will then prove that they hold in the generic extension $ V\left[G\right] $, where $ G\subseteq P_{\kappa} $ is generic over $ V $.

\begin{remark}
In \cite{ben2014forcing}, as in other applications of the Magidor iteration, it was assumed that the measures $ \langle W_{\xi} \colon \xi \in \Delta \rangle $, which were used to singularize the measurables of $ \Delta $, are all derived from normal measures of order $ 0 $ (in the sense that, for every $ \xi\in \Delta $, there exists $ U_{\xi}\in V $ of order $ 0 $, such that $ W_{\xi} = U^{*}_{\xi} $). We do not assume this in the current paper. Each measure $ W_{\xi} $ has, by induction, the form $ U^{\times}_{\xi} $ for some normal measure $ U\in V $, but $ U $ does not necessarily has Mitchell order $ 0 $. 
\end{remark}

We start by extending every normal measure $ U\in V $ on $ \kappa $, to a measure $ U^*\in V\left[G\right] $. For every $ P_{\kappa} $-name $ \lusim{A} $ for a subset of $ \kappa $, $ \left( \lusim{A} \right)_G \in U^* $ if and only if, for some $ p\in G $, 
$$ \{ \xi< \kappa \colon p^{-\xi} \Vdash \check{\xi}\in \lusim{A} \}\in U $$ 
or simply $ \left(j_U(p)\right)^{-\kappa} \Vdash \check{\kappa } \in j_{U}\left( \lusim{A} \right) $ in $ M_U $.

\begin{lemma} \label{Lemma: FS, extending U to U*}
$ U^* $ is a measure on $ \kappa $ in $ V\left[G\right] $ which extends $ U $. Moreover, $ U^* $ is normal if and only if $ U $ has Mitchell order $ 0 $ in $ V $.
\end{lemma}

\begin{proof}
It's not hard to verify that $ U^* $ is a filter which extends $ U $. Let us prove that it is a $ \kappa $-complete ultrafilter. Assume that $ \langle \lusim{A}_{\xi} \colon \xi<\delta \rangle $ is forced by a condition $ p\in G $ to be a partition of $ \kappa $, for some $ \delta< \kappa $. Assume that $ q $ is an arbitrary condition above $ p $. For every $ \alpha \in \left( \delta, \kappa \right) $, consider the $ P_{\alpha} $-name for the following set $ e(\alpha) $, which is forced by $ q\restriction_{\alpha} $ to be $ \leq^* $-dense open above $ q\setminus \alpha $,
$$ e(\alpha) = \{  r\geq^* q\setminus \alpha \colon \exists \xi^*< \delta, \ r\Vdash \check{\alpha} \in \lusim{A}_{\xi^*} \} $$
by lemma \ref{Lemma: FS, fusion lemma}, there exists $ p^* \in G $ above $ p $, such that for every $ \alpha\in \left( \delta, \kappa \right) $,
$$ p^*\restriction_{\alpha} \Vdash \exists \xi^*< \delta, \ \left(p^*\setminus \alpha\right)^{-\alpha} \Vdash \check{\alpha} \in \lusim{A}_{\xi} $$
and thus-- 
$$p^* \Vdash \exists \xi^*<\delta, \  \left(j_U\left( p^* \right)\right)^{-\kappa} \setminus \kappa \Vdash \check{\kappa}\in j_U\left( \lusim{A}_{\xi^*} \right) $$
by extending $ p^* $ to a stronger condition in $ G $, we can assume that $ p^* $ decides the value of $ \xi^* $, and so, for some $ \xi^*<\kappa $, 
$$  \left(j_U\left( p^* \right)\right)^{-\kappa} \Vdash \check{\kappa}\in j_U\left( \lusim{A}_{\xi^*} \right) $$
as desired.

Let us assume that $ U $ has Mitchell order $ 0 $. Let $ \lusim{f} $ be a $ P_{\kappa}$-name for a regressive function, as forced by some $ p\in G $. We use a similar argument as before, but now $ e(\alpha) $ is defined for every non-measurable $ \alpha $, to be the name for the following set, which is forced by any extension of $ p\restriction_{\alpha} $ to be $ \leq^* $-dense open above $ p\setminus \alpha $:
$$ e(\alpha) = \{  r\in P\setminus \alpha \colon \exists \xi^*<\alpha, \ r\Vdash \lusim{f}\left( \alpha \right) = \xi^* \} $$ 
where we used the fact that $ \alpha $ is not measurable, and thus $\langle P\setminus \alpha , \leq^*\rangle $ is more than $ \alpha $-closed. Thus, there exists $ p^* \in G $ such that- $$ p^*\Vdash \exists \xi^*< \kappa, \ \left(j_U(p^*)\setminus \kappa\right)^{-\kappa} \Vdash j_U(\lusim{f})\left( \kappa \right) = \xi^* $$ 
By extending $ p^* $ to a condition in $ G $, we can assume that $ p^* $ decides the value of $ \xi^* $. Thus, $ \{  \xi< \kappa \colon f(\xi) = \xi^* \}\in U^* $, as desired.

Finally, assume that $ U^* $ is normal. Let $ j_{U^*} \colon V\left[G\right]\to M\left[H\right] $ be the ultrapower embedding. Note that $ \kappa $ is not measurable in $ M $, since, else, $ \kappa $ would have been singular in $ V\left[G\right] $. Thus, 
$$ \kappa \in j_{U^*}\left(  \{ \xi< \kappa \colon \xi \mbox{ is not measurable in } V \} \right) $$
and thus $ U = U^*\cap V $ concentrates on non-measurables. 
\end{proof}

Let us define the measure $ U^{\times}  \in V\left[G\right]$. If $ U $ has Mitchell order $ 0 $, we take $ U^{\times} = U^* $. Assume otherwise. Let $ d\colon \Delta \to \kappa $ be the function which maps every measurable cardinal of $ V $ to the first element in its Prikry sequence in $ V\left[G\right] $. We claim that whenever $ U^* $ is non-normal, namely, $ \Delta \in U^* $, $ d $ projects $ U^* $ to the normal measure below it in the Rudin-Keisler order; we denote this projected measure by  $ U^{\times} $. In other words,
$$ U^{\times} = d_{*}\left( U^* \right) = \{  A\subseteq \kappa \colon d^{-1}\left[A\right] \in U^* \} $$
Let us consider the previous equation as the definition of $ U^{\times} $ in the case where $ o(U)>0 $, and prove that it is a normal measure on $ \kappa $.

\begin{lemma} \label{Lemma: FS, U^times in normal}
Let $ U $ be a normal measure on $ \kappa $ in $ V $. Then $ U^{\times} $ is a normal measure on $ \kappa $ in $ V\left[G\right] $.
\end{lemma}

\begin{proof}
We can assume that $ U $ has Mitchell order $ >0 $. It suffices to prove that $ \left[d\right]_{U^*} = \kappa $. 

First, note that for every $ x<\kappa $, $ d^{-1}\{x\} $ is finite. Indeed, given an arbitrary condition $ p\in P_{\kappa} $, let $ b\subseteq \kappa $ be the finite set such for every $ \xi\in \kappa\setminus b $, $ p\restriction_{\xi} \geq^* 0_{Q_{\xi}} $. For every such $ \xi $, let $ p^*\geq^* p $ be such that $ x $ is removed from every measure one set. Then $ p^* $ forces that $ d^{-1}\{x\} $ is finite, and since $ p $ was arbitrary, this indeed holds in $ V\left[G\right] $.

This shows that $ \left[d\right]_{W^*} \geq \kappa$. Assume that $ f\in V\left[G\right] $ is a function in $ V\left[G\right] $ such that, for every $ \xi\in \Delta $, $ f(\xi) < d(\xi)$. Let $ p $ be a condition which forces this. Assume that $ q\geq p $ is arbitrary, and let $ \xi_0 $ be an ordinal which such that for every $ \xi> \xi_0 $, $ q\restriction_{\xi} \Vdash q(\xi) \geq^* 0_{\lusim{Q}_{\xi}} $. For every $ \xi\in \Delta $ above $ \xi_0 $, we describe a name for a subset of $ P\setminus \xi $ which is forced by $ q\restriction_{\xi} $ to be $ \leq^* $ dense open subset of $ P\setminus \xi $ above $ q\setminus \xi $,
\begin{align*}
	 e(\xi) = \{ r\geq q\setminus \xi \colon \mbox{there exists } \gamma< \xi \mbox{ such that } r\setminus \xi \Vdash \lusim{f}\left( \xi \right) = \gamma \}
\end{align*}
The density follows since every name for an ordinal below the first element for a Prikry sequence can be decided by a direct extension. 

By fusion, there exists $ p^*\in G $ above $ p $ such that--
$$ p^* \Vdash \exists \gamma <\kappa, \ \left(j_U\left( p \right)\setminus \kappa\right)^{-\kappa} \Vdash j_U\left(\lusim{f}\right)(\kappa)= \check{\gamma} $$
and by extending $ p^* $ to a condition in $ G $, we can assume that it decides the value of $ \gamma<\kappa $. So $ \left(j_U(p^*)\right)^{-\kappa} \Vdash j_U\left( \lusim{f} \right)(\kappa) = \check{\gamma} $, and thus, in $ V\left[G\right] $, $ \left[f\right]_{U^*} = \gamma < \kappa $, as desired.
\end{proof}

\begin{claim}
	Let $ U\in V $ be a normal measure on $ \kappa $. The following are equivalent:
	\begin{enumerate}
		\item $ U$ has Mitchell order $ 0 $ in $ V $.
		\item $ U^{\times} = U^{*} $.
		\item $ d''\Delta \notin U^{\times} $.
	\end{enumerate}
\end{claim}

\begin{proof}
Clearly $ 1 $ implies $ 2 $ by the definition of $ U^{\times} $. 

Assume $ 2 $. If $d''\Delta \in U^{\times} $ then $ d''\Delta\in U^* $, and thus, there exists $ p\in G $ such that--
$$ \left(j_U(p)\right)^{-\kappa} \Vdash \check{\kappa}\in j_U( \lusim{d}''\Delta ) $$
but this cannot happen, since $ \left(j_U(p)\right)^{-\kappa} $ forces that $ \kappa $ does not appear as an element in any of the Prikry sequences. 

Finally, if $ U $ has Mitchell order higher than $ 0 $ in $ V $, then $ \Delta\in U $ and thus $ \Delta\in U^* $.  Therefore, $ d''\Delta \in U^{\times} $.
\end{proof}

\begin{lemma} \label{Lemma: FS, Properties of U* with o(U)>0}
Let $ U $ be a normal measure on $ \kappa $ in $ V $ with $ o(U)>0 $. Let $ j_{U^{*}}\colon  V\left[G\right] \to M\left[H\right]$ be the ultrapower embedding of $ U^{*} $. Then $ \left[Id\right]_{U^*} $ is measurable in $ M $, and $ \kappa $ appears as a first element in its Prikry sequence in $ M\left[H\right] $. $ \left[Id\right]_{U*} $ is maximal with this property, namely, for every measurable above $ \left[Id\right]_{U^*} $, $ \kappa $ does not appear in its Prikry sequence. Furthermore, for every $ \mu> \left[Id\right]_{U^*} $ measurable in $ M $, $ d\left( \mu \right) > \left[Id\right]_{U^*} $.
\end{lemma}

\begin{proof}
Since $ \Delta\in U\subseteq U^* $, $ \left[Id\right]_{U^*} $ is measurable in $ M $. But-- $$\kappa =  \left[d\right]_{U^*} =j_{U^*}\left(d\right)\left( \left[Id\right]_{U^*} \right)$$
so $ \kappa $ appears first in the Prikry sequence of $ \left[Id\right]_{U^*} $ in $ M\left[H\right] $. 

Finally, fix any condition $ p\in G $. Then--
$$ \left(j_{U}(p)\right)^{-\kappa} \Vdash \mbox{for every }  \mu\in j_U\left( \Delta \right)\setminus \left(\kappa+1\right),  \  j_U(\lusim{d})(\mu) > \kappa $$
In particular, $ \{ \xi<\kappa \colon \mbox{for every } \mu\in \Delta\setminus \left(\xi+1\right), \ d(\mu) >\xi \}\in U^* $. Thus, for every measurable $ \mu> \left[Id\right]_{U^*} $, $ d(\mu) > \left[Id\right]_{U^*} $.
\end{proof}

Let us assume now that $ W $ is an arbitrary normal measure on $ \kappa $ in $ V\left[G\right] $. Our goal will be to prove that $ W = U^{\times} $ for some normal measure $ U\in V $. Denote by $ j_W \colon V\left[G\right] \to M\left[H\right] $ the ultrapower embedding of $ W $ over $V\left[G\right] $.

\begin{remark} \label{Remark: U^0 = Wcap V}
$ W\cap V $ is a normal measure in $ V $ of Mitchell order $ 0 $. Indeed, by corollay \ref{Corollaery: Every measure in the generic extension extends a measure of V}, $ W\cap V \in V $. Clearly $ W\cap V $ is normal in $ V $. Finally, note that $ \Delta\notin W\cap V $, namely $ \kappa \notin j_{W}\left( \Delta \right) $. Otherwise, $ \kappa $ was measurable in $ M $, and thus singular in $ M\left[H\right] \subseteq V\left[G\right] $. But $ \kappa $ is regular in $ V\left[G\right] $, a contradiction.
\end{remark}

Let us assume, by induction, that for every measurable $ \mu<\kappa $, the normal measures on $ \mu $ in $ V^{P_{\mu}} $ have the form $ U^{\times} $ for some normal measure $ U $ on $ \mu $ in $ V $. From the previous remark, we can assume also that every such $ U^{\times} $ concentrates on non-measurables of $ V $ below $ \mu $.

We now define a measure $ W^* \in V\left[G\right] $ on $ \kappa $. If $ d''\Delta \notin W $, take $ W^* = W $. Assume otherwise. For every $ \delta < \kappa $, the set $ d^{-1}\{\delta\} $ is finite (see the proof of lemma \ref{Lemma: FS, U^times in normal}). Define a set $ \Delta^* \subseteq \Delta $,
$$ \Delta^* = \{ \xi \in \Delta \colon \xi = \max d^{-1}\{ d \left(\xi\right)\} \} $$
$ \Delta^*  $ is an unbounded subset of $ \Delta $, on which $ d $ is injective. Let--
$$ W^* = \{ X\subseteq \kappa \colon d'' \left(X\cap \Delta^*\right) \in W \} $$
$ W^* $ is a non-trivial, $ \kappa $-complete ultrafilter on $ \kappa $.

Let us review some of the properties of $ W^* $ in the case where $ d''\Delta\in W $. Clearly $ \Delta,\Delta^* \in W^* $. $ d $ is a Rudin-Keisler projection of $ W^* $ onto $ W $, and is injective on $ \Delta^*\in W^* $. Therefore $ W \equiv_{RK} W^* $, and in particular $ j_W = j_{W^*} $, namely $ W,W^* $ have the same ultrapower embedding from $ V\left[G\right] $ to $ M\left[H\right] $. In $ M\left[H\right] $, $ \kappa = \left[d\right]_{W^*} = j_{W^*}(d)\left( \left[Id\right]_{W^*} \right)  $, namely $ \kappa $ is the first element in the Prikry sequence of $ \left[Id\right]_{W^*} $. Finally, $ \Delta^*\in W^* $, and thus--
$$ \left[Id\right]_{W^*} = \max j_{W^*}(d)^{-1} \{ \kappa \} $$
so $ \kappa $ does not appear as first element in the Prikry sequence of any  measurable above $ \left[Id\right]_{W^*} $.

\begin{lemma} \label{Lemma: FS, W*cap V is a normal measure}
$ W^*\cap V\in V $ is a normal measure on $ \kappa $ in $ V $. 
\end{lemma}

\begin{proof}
By corollary \ref{Corollaery: Every measure in the generic extension extends a measure of V}, $ W^*\cap V \in V $. If $ d''\Delta\notin W $, then $ W^* = W $ is normal, and so is $ W^*\cap V $. Let us assume that $ d''\Delta\in W $. Assume that $ f\in V $ and $ \{ \xi < \kappa \colon f(\xi) < \xi \}\in W^*\cap V $. Denote this set by $ A $ and assume that $ A\subseteq \Delta $ (else, intersect).

For every $ p \in P_{\kappa} $, there exists a direct extension $ p^*\geq^* p $ and a finite subset $ b\subseteq \kappa $ such that, for every $ \xi \in A \setminus b $, 
$$ p^*\restriction_{\xi} \Vdash \lusim{A}^{p^*}_{\xi} \subseteq \xi\setminus \left(f(\xi)+1\right) \mbox{ and } t^{p^*}_{\xi} = \langle \rangle$$ 
thus, there exists such $ b\subseteq \kappa $ and $p^*\in G$. Then $ p^* $ forces that for every $ \xi\in A\setminus b $, $f(\xi) < d(\xi) $. But $ A \in W^* $, and thus $ A \setminus b\in W^* $, so, in $ M\left[H\right] $, $ \left[f\right]_{W^*} < \left[ d \right]_{W^*} = d\left( \left[Id\right]_{W^*} \right)= \kappa $. Therefore, there exists $ \beta<\kappa $ such that--
$$ \{ \xi<\kappa \colon f(\xi) = \beta \} \in W^* $$
but this set belongs to $ V $ (since $ f\in V $), and thus--
$$ \{ \xi<\kappa \colon f(\xi) = \beta \} \in W^*\cap V $$
as desired.
\end{proof}

\begin{lemma} \label{Lemma: FS, jW of p minus idW* belongs to H}
Let $ p\in G $ be a condition. Then $ \left(j_{W}(p)\right)^{-\left[ Id\right]_{W^*}} \in H $. In particular, if $ d''\Delta \notin W $, Then $ j_W(p)^{-\kappa} \in H $.
\end{lemma}

\begin{proof}
$ j_W(p) \in H $ since $ p\in G $. In order to prove that $ \left(j_W(p)\right)^{-\left[Id\right]_{W^*}} \in H $, it suffices to prove that ordinals $ \leq \left[Id\right]_{W^*} $ do not appear in Prikry sequences of measurables above $ \left[Id\right]_{W^*} $ in $ M\left[H\right] $.

Clearly, for every $ \mu > \left[Id\right]_{W^*} $, $ d(\mu) \geq \kappa $. Otherwise, there exist $ \alpha< \kappa $ and $A\in W^* $ such that for every $ \xi\in A $, there is some $ \mu(\xi)> \xi $ with $ d\left( \mu(\xi) \right)= \alpha $. In particular, $ d^{-1}\{ \alpha \} $ is infinite, a contradiction.

Let us argue now that for every $ \mu> \left[Id\right]_{W^*} $, $ d(\mu) > \kappa $. If $ d''\Delta\notin W $ this is clear, since $ \kappa $ does not belong to the image of $ d $ in $ M\left[H\right] $. Thus, let us take care of the case where $ d''\Delta \in W $. In this case, recall that in $ M\left[H\right] $, $ \left[Id\right]_{W^*} = \max d^{-1} \{\kappa\} $. Thus, for every $ \mu> \left[Id\right]_{W^*} $, $ d(\mu) \neq \kappa $. 

Finally, let us argue that for every $ \mu> \left[Id\right]_{W^*} $, $ d(\mu) > \left[Id\right]_{W^*} $. It suffices to prove that for every such $ \mu $, $ d(\mu) \notin \left( \kappa , \left[Id\right]_{W^*} \right] $. If $ d''\Delta\notin W $ this is clear, since in this case $ W^* = W $ and $ \left[Id\right]_{W^*} = \kappa $. Let us assume that $ d''\Delta\in W $. We claim that in $ V\left[G\right] $, there exists a finite set $ b\subseteq \kappa $ such that for every measurable $ \mu>\sup(b) $, 
$$ d(\mu)\notin \bigcup_{ \xi\in \Delta\cap \mu }\left( d(\xi), \xi \right] $$
We prove this by a density argument. Fix a condition $ p\in P_{\kappa} $. Let $ b\subseteq \kappa $ be the set of coordinates such that for every $ \mu> \sup(b) $, $ p\restriction_{\mu}\Vdash p(\mu)\geq^* 0 $. We extend $ p $ to $ p^*\geq^* p $ such that, for every measurable $ \mu> \sup(b) $, 
$$p^*\restriction_{\mu} \Vdash \lusim{d}(\mu) \notin \bigcup_{\xi\in \Delta \cap \mu} \left( d(\xi), \xi  \right] $$
this is possible since, by the induction hypothesis, the weakest condition in $ P_{\mu} $ forces that--
$$ \bigcup_{\xi\in \Delta\cap \mu}  \left( d(\xi), \xi \right]$$
does not belong to any normal measure in $ V^{P_{\mu}} $. 
Pick such $ p^*\in G $. Then in $ V\left[G\right] $, for every $ \mu>\sup(b) $, 
$$ d(\mu)\notin \bigcup_{ \xi\in \Delta\cap \mu }\left( d(\xi), \xi \right] $$
This is true for every $ \mu\in \Delta \setminus \sup(b)\in W^* $. Thus, in $ M\left[H\right] $, for every $ \mu> \left[Id\right]_{W^*} $, $ d\left( \mu \right)\notin \left( \kappa, \left[Id\right]_{W^*} \right] $.
\end{proof}

We are now ready to prove theorem \ref{Thm: Characterization of normal measures}. Recall that the proof is inductive, so we assume that for every $ \alpha\in \Delta $, theorem \ref{Thm: Characterization of normal measures} and lemma \ref{Lemma: Every normal measure concentrates on the complement of the set of unions} are true in $ V^{P_{\alpha}} $; we proceed and prove each of  them in $ V\left[G\right] $.

\begin{proof} [Proof of theorem \ref{Thm: Characterization of normal measures}]
Let $ W\in V\left[G\right] $ be a normal measure on $ \kappa $. Let $ U = W^* \cap V $. Let $ k\colon M_U \to M $ be the embedding which satisfies, for every $ f\in V $, 
$$ k\left( \left[f\right]_U \right) = \left[f\right]_{W^*} $$
It's not hard to verify that $ k $ is elementary and $j_{W}\restriction_{V} =  j_{W^*}\restriction_{V} = k\circ j_U $. Moreover, $ \mbox{crit}\left(k \right) > \kappa $ if and only if $ d''\Delta\notin W $. Let us argue now that $ W = U^{\times} $.

Assume first that $ d''\Delta\notin W $. Then $ U $ has Mitchell order $ 0 $. We argue that $ W = U^{\times} = U^* $. It suffices to prove that $ U^* \subseteq W $. Let $ X\in U^* $, and assume that $ \lusim{X}\in V $  is a $ P_{\kappa} $-name such that $ \left( \lusim{X} \right)_G = X $. Then for some $ p\in G $, $$ \left(j_U(p)\right)^{-\kappa} \Vdash \check{\kappa}\in j_U\left( \lusim{X} \right) $$
By applying $ k\colon M_U \to M $,
$$ \left(j_W(p)\right)^{-\kappa} \Vdash \check{\kappa}\in j_W\left( \lusim{X} \right) $$
where we used that fact that $ k\left( \kappa \right) = \kappa $. Since $ p\in G $ and $ d''\Delta\notin W $, $ \left(j_W(p)\right)^{-\kappa}\in H $, and thus, in $ M\left[H\right] $, $ \kappa\in j_W\left( X \right) $, as desired.

Assume now that $ d''\Delta \in W $. Then $ o\left(U\right)>0 $. In this case, $ k\left( \kappa \right) = \left[Id\right]_{W^*} > \kappa $. 
Let us prove that $ W = U^{\times} $. Since both are ultrafilters in $ V\left[G\right] $, it suffices to prove that $ U^{\times} \subseteq W $. Assume that $ X\in U^{\times} $, and let $ \lusim{X}\in V $ be such that  $ \left( \lusim{X} \right)_{G} = X $. Let $ p\in G $ be a condition such that-- 
$$ \left(j_U(p)\right)^{-\kappa} \Vdash \check{\kappa} \in j_U\left( \lusim{d}^{-1} \lusim{X} \right) $$ By applying $ k\colon M_U \to M $,
$$ \left(j_W(p)\right)^{-\left[Id\right]_{W^*}} \Vdash \lusim{d}\left(\left[Id\right]_{W^*}\right) \in j_W\left( \lusim{X} \right) $$
but $ \left(j_W(p)\right)^{-\left[Id\right]_{W^*}} \in H $ by lemma \ref{Lemma: FS, jW of p minus idW* belongs to H}, and thus, in $ M\left[H\right] $,
$$\kappa = d\left( \left[Id\right]_{W^*}\right) \in \left(j_{W} \left( \lusim{X} \right) \right)_{H} = j_{W}\left(  X \right)$$
so $ X\in W $, as desired.

Finally, assume that $ U\neq U' $ are normal measures in $ V $. If both have Mitchell order $ 0 $, then $ U^*\neq {U'}^{*} $ and thus $ U^{\times} \neq {U'}^{\times} $. If exactly one of them, say $ U $, has Mitchell order $ 0 $, then $ d''\Delta \in {U'}^{\times}\setminus U^{\times} $. Thus, let us consider the case where both have Mitchell order higher than $ 0 $. Let $ A\in U $, $ B\in U' $ be disjoint sets. In $ V\left[G\right] $, let $ A^* = A\cap \Delta^*\in U^{*} $, $ B^* = B\cap \Delta^*\in {U'}^* $. Then $ d''A^*\in U^{\times}, \  d''B^*\in {U'}^{\times}$, and $ d''A^*\cap d''B^* =\emptyset$ since $ d $ is injective on $ \Delta^* $. Thus $ U^{\times}\neq {U'}^{\times} $. \end{proof}

The embedding  $ k\colon M_{U} \to M $ from the above proof will be used in the next sections to analyze the structure of $ j_{W}\restriction_{V} $. For now, let us note that $ \mbox{crit}(k) = \kappa $ if and only if $ d''\Delta\notin W $.

\begin{proof} [Proof of lemma \ref{Lemma: Every normal measure concentrates on the complement of the set of unions}]
	Let $ U $ be a normal measure in $ V $ such that $ W = U^{\times} $. If $ U $ has Mitchell order $ 0 $, then for every $ p\in G $,
	$$ \left(j_U(p)\right)^{-\kappa} \Vdash \check{\kappa}\notin \bigcup_{\alpha\in j_U(\Delta)} {   \left(  \lusim{d}(\alpha), \alpha \right]   }  $$
	and thus--
	$$ \kappa\setminus \left(  \bigcup_{\alpha\in \Delta} \left( d(\alpha), \alpha \right] \right) \in U^* = W $$
	Let us assume that $ U $ has Mitchell order higher than $ 0 $. Then it suffices to prove that--
	$$ \kappa\setminus \left(  \bigcup_{\alpha\in \Delta} d{^{-1}}\left( d(\alpha), \alpha \right] \right) \in U^* $$
	this holds if and only if there exists $ p\in G $ such that, in $ M_{U} $,
	$$ \left(j_U(p)\right)^{-\kappa}\Vdash \lusim{d}(\kappa)\notin \bigcup_{\alpha\in j_U(\Delta)} {  \left( d(\alpha), \alpha \right] } $$
	By induction hypothesis, for every measurable $ \mu<\kappa $, 
	$$ \bigcup_{\alpha\in \Delta\cap \mu}  \left( d(\alpha), \alpha \right] \notin U^{\times}_{\mu} $$
	and this is forced by the weakest condition in $ P_{\mu} $. Given an arbitrary condition $ p\in P_{\kappa} $, there exists $ p^*\geq^* p $ such that for every measurable $ \mu\in \Delta $,
	$$ p^*\restriction_{\mu} \Vdash \lusim{A}^{p^*}_{\mu}\subseteq \mu\setminus \bigcup_{\alpha\in \Delta\cap \mu} \left( d(\alpha) , \alpha \right]  $$
	and thus there exists such $ p^* \in G $. Then--
	$$ j_U(p^*) \Vdash \lusim{d}(\kappa)\notin \bigcup_{\alpha\in j_U\left( \Delta \right)\cap \kappa } \left( d(\alpha), \alpha \right] $$
	and thus--
	$$ \left(j_U(p^*)\right)^{-\kappa} \Vdash \lusim{d}(\kappa)\notin \bigcup_{\alpha\in j_U\left( \Delta \right) } \left( d(\alpha), \alpha \right] $$
	as desired.
\end{proof}

\section{The Structure of $ j_W\restriction_{V} $}

Given a normal measure $ W\in V\left[G\right] $ on $ \kappa $, let $ j_W\colon V\left[G\right] \to M\left[H\right]$ be the ultrapower embedding, and let $ U\in V $ be a normal measure on $ \kappa $ such that $ W = U^{\times} $. Out main goal in this section will be to factor $ j_W\restriction_{V} $ to an iterated ultrapower of $ V $. 

We divide this section to several subsections. In the first subsection, we isolate a natural number $ m<\omega $ and a sequence $  U^0 \vartriangleleft U^{1} \vartriangleleft\ldots \vartriangleleft U^{m} = U $ of measures on $ \kappa $ in $ V $, all of them participate in the iteration $ j_W\restriction_{V} $ (when taken in decreasing order with respect to the Mitchell order). In the second subsection, we describe in detail the structure of $ j_W\restriction_{V} $ and sketch the main steps in the proof. We will also demonstrate the structure of $ j_W\restriction_{V} $ in several simple cases. In the third subsection, we develop a generalization of the Fusion lemma. This generalization will be applied in the fourth subsection, where we complete the proof of theorem \ref{Thm: Structure of j_Wrestriction V}, provide a sufficient condition for the definability of $ j_W\restriction_{V} $ in $ V $, and describe the Prikry sequences added by $ H $ for measurables of $ M $ above $ \kappa $. For instance, we will prove that each measure $ U^{j} $, for $ 0\leq j < m $, is iterated in $ j_{W}\restriction_{V} $ $ \omega $-many times, producing Prikry sequences for each of the measurables in the finite set $ d^{-1}\{ \kappa \} $

The value of $ m<\omega $ and the exact measures participating in the sequence $  U^0 \vartriangleleft U^{1} \vartriangleleft\ldots \vartriangleleft U^{m} = U $ depend on $ W $ and on the measures in the sequence $\langle W_{\xi} \colon \xi\in \Delta \rangle\in V\left[G\right]$, namely the measures used in $ G $ to singularize the measurables of $ \Delta $. For every $ \xi\in \Delta $, denote by $ U_{\xi}\in V $ the measure on $ \xi $ such that $ W_{\xi} = U^{\times}_{\xi} $. By induction, for every $ \xi\in \Delta $ there exists a natural number $ m_{\xi} $ and a sequence $ U^{0}_{\xi}\vartriangleleft \ldots \vartriangleleft U^{m_{\xi}-1}_{\xi} \vartriangleleft  U^{m_{\xi}}_{\xi} = U_{\xi} $ of normal measures on $ \xi $ in $ V $. The identity of the measures $ \langle U^{i}_{\xi} \colon \xi\in \Delta, \ j\leq m_{\xi} \rangle $ determines what are the measures participating in the iteration of $ j_W\restriction_{V} $, and whether or not this iteration is definable in $ V $.

\subsection{The system $ U^0 \vartriangleleft U^{1} \vartriangleleft\ldots \vartriangleleft U^{m} $ associated with $ W $}

Denote $ m = m(W) = \left|  d^{-1}\{ \kappa \} \right| $ as computed in $ M\left[H\right] $. Namely, $ m(W) < \omega $ is the number of occurrences of $ \kappa $ as a first element in Prikry sequences added to measurables in $ M $. Possibly $ m(W) =0  $, in the case where $ d''\Delta\notin W $. Define, for every $  i\geq 1 $, the set $ \Delta_i \subseteq \Delta $:
\begin{align*}
	\Delta_{i} =& \{ \xi \in \Delta \colon  \left| \xi \cap d^{-1}\{d(\xi) \} \right| = i-1 \} = \\ &\{ \xi\in \Delta \colon \xi \mbox{ is the } i\mbox{-th element in } d^{-1}\{ d(\xi) \}  \} 
\end{align*}
For $ i=0 $, let $ \Delta_{0} = \kappa\setminus \Delta $, the set of non-measurables below $ \kappa $. We state some straightforward properties:
\begin{claim} \mbox{ }
	\begin{enumerate}
		\item $ \{  \xi<\kappa \colon \xi \mbox{ appears as first element in } m \mbox{ Prikry sequences below } \kappa \}\in W $.
		\item For every $ \xi\in \Delta $, if $m\left( W_{\xi} \right) = i-1 $ for some $ 1\leq i<\omega $, then $ \xi $ is the $ i$-th element in $ d^{-1}\{ d(\xi) \} $.
		\item $d''\Delta_1 \supseteq d''\Delta_2 \supseteq \ldots \supseteq d''\Delta_{n} \supseteq \ldots $ ($n<\omega$).
		\item $ m $ is the maximal index such that $ d''\Delta_{m}\in W $.
	\end{enumerate}
\end{claim}

Note that $ d $ is injective on each of the sets $ \Delta_i $.
Let us define, for every $ 1\leq i\leq m $, a measure $ W^i $ as follows:
$$ W^i = \{ X\subseteq \kappa \colon  d''\left( X\cap \Delta_{i}\right) \in W \} $$ 
In particular, $ W^m$ is the measure $ W^* $ defined in the previous section. Since $ d $ is injective on each set $ \Delta_i $, 
$$ W\equiv_{RK} W^1 \equiv_{RK} W^2 \equiv_{RK} \ldots \equiv_{RK} W^m = W^* $$
For every $1\leq  i\neq j \leq m $ , let $\pi_{i,j} \colon \Delta_i \to \Delta_j$ be the function which maps each $ \xi\in \Delta_i $ to the $ j $-th element in $ d^{-1}(d\left(\xi\right)) $ (which typically exists. if not, set $ \pi_{i,j}(\xi) = 0  $). Then $ \pi_{i,j} $, which  projects $ W^i $ onto $ W^j $, is injective on the set-- 
$$ \{ \xi \in \Delta_i \colon \left| d^{-1}\left( d(\xi) \right) \right| = m \}\in W^i $$
Finally, denote, for every $ 1\leq i\leq m $, $ U^{i} = W^i \cap V \in V$, and note that $  U = U^{m} $.\\
For sake of completeness, let us denote $ W^{0} = W $ and $ U^{0} = W\cap V $. By remark \ref{Remark: U^0 = Wcap V}, $ U^{0} $ concentrates on $\Delta_0 = \kappa\setminus \Delta$. We begin by studying the properties of $ U^{0} $.

\begin{lemma} \label{Lemma: FS, U0 below U}
$ U^{0}\trianglelefteq U $ is a normal measure of Mitchell order $ 0 $ in $ V $. $ U^{0} = U $ if and only if $ U $ already has Mitchell order $ 0 $ in $ V $. Finally, if $ U $ has Mitchell order above $ 0 $ in $ V $, then $ U^{0} = \{ A\subseteq \kappa \colon \kappa\in k\left( A \right) \}\cap M_U $.
\end{lemma}

We will need the following claim:

\begin{claim} \label{Lemma: FS, MUG and VG have the same subsets of kappa}
	Let $ U\in V $ be a measure on $ \kappa $. Then $ M_U\left[G\right] $ and $ V\left[G\right] $ have the same subsets of $ \kappa $.
\end{claim}

\begin{proof}
	Assume that $ \lusim{A}\in V $ is a $ P$-name for a subset of $ \kappa $. For every non-measurable $ \alpha<\kappa $, let $ e(\alpha) $ be the $ \leq^* $-dense open subset of $ P\setminus \alpha $ which decides the value of $ \lusim{A}\cap \alpha $ over $ V^{P_{\alpha}} $. By lemma \ref{Lemma: FS, fusion lemma}, there exists $ p\in G $ such that for every non-measurable $ \alpha<\kappa $,
	$$ p\restriction_{\alpha}\Vdash \left( p\setminus \alpha \right)^{-\alpha} \in e(\alpha) $$
	For every such $ \alpha $, let $ \lusim{A}_{\alpha}\in V_{\kappa} $ be a $ P_{\alpha} $-name such that-- 
	$$ p\restriction_{\alpha} \Vdash p\setminus \alpha \Vdash \lusim{A}\cap \alpha = \lusim{A}_{\alpha} $$
	The sequence $ \langle \lusim{A}_{\alpha} \colon \alpha<\kappa \rangle $ belongs to $ M_U $. Thus, $ A = \left( \lusim{A} \right)_G \in  M_{U}\left[G\right] $, since--
	$$ A = \bigcup_{\alpha<\kappa}  \left( \lusim{A}_{\alpha} \right)_{G_{\alpha}} $$
\end{proof}

\begin{proof}[Proof of Lemma \ref{Lemma: FS, U0 below U}]
	If $ U $ has Mitchell order $ 0 $ in $ V $, then $ W = U^{\times} = U^{*} $ and thus $ U^{0} = W\cap V = U $. Let us assume that $ U $ has Mitchell order higher than $ 0 $, namely $ \Delta\in U $.
	
	We provide a definition of $ U^{0} $ which is different from the definition $ U^{0} = W\cap V $ as in the statement of the lemma. From the definition we provide, it will be simple to see that $ U^{0}\in M_{U} $. After that,  we will prove that indeed $ U^{0} = W\cap V $.
	 
	In $ V\left[G\right] $, define for every $ \alpha\in \Delta $,  $ U^{0}_{\alpha} = W_{\alpha}\cap V\in V $.  In $ V $, let $ \lusim{U}^{0} = j_U\left( \langle \lusim{U}^{0}_{\alpha} \colon \alpha\in \Delta \rangle \right)\left( \kappa \right) $. This is a $ j_U(P)\restriction_{\kappa} = P $-name for a normal measure of Mitchell order $ 0 $ which belongs to $ M_U $. Let $ U^{0} = \left( \lusim{U}^{0} \right)_G \in M_U $. Then $ U^{0}\in V $ is a normal measure on $ \kappa $ of Mitchell order $ 0 $. Since $ U^{0}\vartriangleleft U $, it suffices to prove that $U^{0} = W\cap V$. Assume that $ A\in U^{0} $ holds, and consider this as a statement in $ M_U\left[G\right] $. For some $ p\in G $, 
	$$ p\Vdash \check{A} \in  j_U\left( \langle \lusim{U}^{0}_{\alpha} \colon \alpha\in \Delta \rangle \right)\left( \kappa \right)   $$
	Let $ \alpha \mapsto A(\alpha) $ be a function in $ V $ which represents $ A $ in $ M_U $. Then we can assume that for every $ \alpha\in \Delta $, 
	$$ p\restriction_{\alpha}\Vdash \check{A}(\alpha)\in \lusim{U}^{0}_{\alpha} \subseteq U^{\times}_{\alpha} $$
	By lemma \ref{Lemma: FS, fusion lemma}, there exists $ p^*\geq^* p $ such that, for all but finitely $ \alpha\in \Delta $, 
	$$ \left( p^*\right)^{-\alpha}\Vdash \lusim{d}(\alpha)\in \check{A}(\alpha) $$
	where $ d(\alpha) $ is the first element in the Prikry sequence of $ \alpha $. Thus,
	$$ \left(j_U(p^*)\right)^{-\kappa} \Vdash \check{\kappa}\in j_U\left(\lusim{d}^{-1}\left( \check{A} \right)\right) $$
	and thus $ d^{-1}\left( A \right) \in U^* $ in $ V\left[G\right] $. Therefore, 
	$$ d''\left(d^{-1}A\right)\in U^{\times} = W $$
	so $ A\in W $, as desired.
	
	Finally, let us assume that $ \Delta\in U $ and argue that $ U^{0} = \{ A\subseteq \kappa \colon \kappa\in k\left( A \right) \}\cap M_U $. Since both are ultrafilters in $ M_{U} $, it suffices to prove that $ U^{0}  \subseteq  \{ A\subseteq \kappa \colon \kappa \in k\left(A\right) \}\cap M_{U}$.
	
	Let $ A\in U^{0} $ be a set, and assume that $ \xi\mapsto A(\xi) $ is a function in $ V $ such that $ \left[\xi \mapsto A\left(\xi\right)\right]_U = A $. Assume that $ p\in G $ forces that $ A\in \lusim{U}^{0} $. We can assume that for every $ \xi<\kappa $, $ p\restriction_{\xi} \Vdash A\left(\xi\right)\in U^{0}_{\xi} $, and in particular, $ p\restriction_{\xi} \Vdash A\left( \xi \right)\in U^{\times}_{\xi} $.
	
	Given any extension $q\geq p $ in $ P_{\kappa} $, there exists $ p^*\geq^* p $ and a finite subset $ b\subseteq \kappa $ such that, for every $ \xi \in \Delta\setminus b $,
	$$ p^*\restriction_{\xi}\Vdash \lusim{A}^{p^*}_{\xi} \subseteq A(\xi) \mbox{ and } t^{p^*}_{\xi} = \langle \rangle$$
	and thus, there exists such $ p^* \in G $. Since $ \Delta\setminus b\in W^* $ and $ j_{W^*}(p) \in H $, it follows that, in $ M\left[H\right] $,
	$$ \kappa = d(\left[Id\right]_{W^*}) \in \left[\xi \mapsto A_{\xi}\right]_{W^*} = k\left(  A \right) $$
	as desired.
\end{proof}

\begin{lemma} \label{Lemma: FS, the U^i s are Mitchell-ordered}
	For every $ 1\leq i \leq m $, $ U^i $ is normal and has Mitchell order higher than $ 0 $. Furthermore, 
	$$ U^{0} = U^0\vartriangleleft U^1 \vartriangleleft U^2   \vartriangleleft \ldots  \vartriangleleft U^m=U $$
\end{lemma} 

\begin{proof}
	The proof that each $ U^{i} $ is normal is identical to \ref{Lemma: FS, W*cap V is a normal measure},  and essentially uses the fact that $ d $ projects each $ W_i $ onto $ W $. 
	
	For $ i\geq 1 $, each $ U^{i}$ has Mitchell order above $ 0 $: otherwise, $ \kappa\setminus \Delta\in U^{i}\subseteq W^{i} $, and this contradicts the fact that $ \Delta_i \in W_i $ is disjoint from $ \kappa\setminus \Delta $.
	
	Let us prove that for every $ 1\leq i<m $, $ U^i \vartriangleleft U^{i+1} $.  Work in $ V\left[G\right] $. For every $ \xi<\kappa $, let $U^{\times}_{\xi}$ be the normal measure used at stage $ \xi $ in the iteration. We define an ultrafilter $ U^{i}_{\xi} $: if $ U^{\times}_{\xi} $  concentrates on $d''\left(\Delta_{i}\cap \xi\right)$, set $ \lusim{U}^{i}_{\xi} $ to be the ultrafilter which concentrates on $ \Delta_{i}\cap \xi $ and is projected via $ d $ onto  $ U^{\times}_{\xi} $. Else, set $ U^{i}_{\xi} = U^{*}_{\xi} $.
	
	Let $ \lusim{\mathcal{U}}^{i}\in V $ be the sequence of names for the measures $ U^{i}_{\xi} $ defined above. Consider in $ M_{U^{i+1}} $ the $ P_{\kappa} $-name $ j_{U^{i+1}}\left( \lusim{\mathcal{U}}^{i} \right)(\kappa) $, and let--
	$$ F = \left(j_{U^{i+1}}\left( \lusim{\mathcal{U}}^{i} \right)(\kappa)\right)_G \cap M_{U^{i+1}} \in M_{U^{i+1}} $$
	$ F $ is a normal measure on $ \kappa $ which belongs to $ M_{U^{i+1}} $. Thus, it suffices to prove that $ F = U^{i} $.
	
	Pick $ X\in F $. Let $ p\in G $ be a condition such that $ p\Vdash X\in j_{U^{i+1}}\left( \lusim{\mathcal{U}}^{i} \right)(\kappa) $, namely--
	$$  \{ \xi\in \Delta \colon p\restriction_{\xi}\Vdash X\cap \xi\in U^{i}_{\xi} \}\in U^{i+1}$$
	we would like to argue that $ U^{i}_{\xi} $ in the above equation is the measure which concentrates on $ \Delta_{i} $ and is projected via $ d $ onto $ U^{\times}_{\xi} $. This requires to have-- 
	$$\{\xi\in \Delta \colon p\restriction_{\xi}\Vdash  d''\left(\Delta_i\cap \xi\right) \in U^{\times}_{\xi}\}\in U^{i+1} $$
	Let us argue that $ p $ can be extended inside $ G $ such that this holds. Work over $ N_{U^{i+1}} $, and extend $ p $ in $ G $ such that--
	$$ p\parallel \Delta_{i}\in j_{U^{i+1} }\left( \lusim{U}^{\times} \right)(\kappa) $$
	It's enough to argue that $ p $ decides the above statement in a positive way. Assume otherwise. Then--
	$$ \{ \xi< \kappa \colon p\restriction_{\xi} \Vdash \Delta_{i}\cap \xi \notin U^{\times}_{\xi} \}\in U^{i+1}\subseteq W_{i+1} $$
	For every $ \xi $ in the above set (but finitely many), $ d(\xi)\notin d''\Delta_{i} $. In particular, $ W_{i+1} $ concentrates on such $ \xi $-s, and thus in $ M\left[H\right] $, $ \kappa\notin d''\Delta_{i} $, which is a contradiction.
	
	Thus we can assume that $ p\in G $ and--
	$$  \{ \xi\in \Delta \colon p\restriction_{\xi}\Vdash d''\left(X\cap\Delta_{i}\cap \xi\right)\in U^{\times}_{\xi} \}\in U^{i+1}$$

	Therefore, 
	$$ \{ \xi \in \Delta \colon p \Vdash d(\xi)\in d''\left( X\cap \Delta_{i}\cap \xi \right) \}\in U^{i+1}$$
	and thus, in $ V\left[G\right] $,
	$$ \{ \xi \in \Delta \colon d(\xi)\in d''\left( X\cap \Delta_i \right) \}\in W^{i+1} $$
	so--
	$$ d''\{ \xi \in \Delta \colon d(\xi)\in d''\left( X\cap \Delta_i \right) \}\in W $$
	so $ d''\left( X\cap \Delta_i \right)\in W $, and in particular, $ X\in W_i $. So $ X\in U^{i} = W_{i}\cap V $.
	
	Finally, let us argue that $ U^{0}\vartriangleleft U^{1} $. Consider in $ M_{U^{1}} $ the name $ j_{U^{1}}\left( \lusim{\mathcal{U}}^{\times} \right)(\kappa) $, and let $ F\in M_{U^{1}} $ be its value with respect to the generic $ G $. It suffices to prove that $ F = U^{0} $. given $ X\in F $, there exists $ p\in G $ such that--
	$$ \{ \xi\in \Delta \colon p\restriction_{\xi}\Vdash X\cap \xi\in U^{\times}_{\xi} \}\in U^{1} $$
	In $ V\left[G\right] $,
	$$ \{ \xi\in \Delta \colon d(\xi) \in X  \}\in W^1  $$
	Recall that $ W^1\equiv_{RK} W $, and thus in $ M\left[H\right] $, 
	$$ d\left(\left[Id\right]_{W^1}\right) \in j_{W^1}(X) = j_{W}(X) = k\left( j_U(X) \right) $$
	where $ k\colon M_U\to M $ is the embedding which satisfies $ k\left( \left[f\right]_U \right) = \left[f\right]_{W^*} $. Recall that $ \mbox{crit}(k) = \kappa $, and thus $\kappa \in k\left( \kappa\cap j_U(X) \right) = k(X)$. In particular, $ X\in U^{0} $.
\end{proof}

\begin{remark} \label{Remark: U^i s derived from k}
	Denote $\langle \mu^{*1}_{0}, \ldots , \mu^{*m}_{0} \rangle = d^{-1}\{ \kappa\}  = \langle \left[Id\right]_{W^1}, \ldots, \left[Id\right]_{W^*} \rangle $. Then for every $ 1\leq i \leq m-1 $, $ U^{i} = \{ X\subseteq \kappa \colon \mu^{*i}_{0}\in k(X)  \} $. Indeed, assume that $ X\subseteq \kappa $ and $ \mu^{*i}_{0}\in k(X) = k\left( j_U(X)\cap \kappa \right) = j_{W^*}(X)\cap \left[Id\right]_{W^*} $. Since $ W^i $ and $ W^* $ are Rudin-Keisler equivalent and $ \mu^{*i}_{0} = \left[Id\right]_{W^i} $, it follows that $ X\in W^{i} $. Therefore $ X\in U^{i} $.
\end{remark}

The previous remark hints that all the measures $ U^{i} $ (for $ 0\leq i \leq m $) participate in $ j_W\restriction_{V} $, since $ U = U^{m} $ is the first step, and each of the measures $ U^{i} $ for $ 0\leq i \leq m-1 $ is derived from $ k $. This is indeed the case. In $ M_{U} $, we can derive a measure on $ \left[\kappa\right]^{m} $ using $ k $ as follows:
$$ \mathcal{E}_0 =  \{ X\subseteq \left[\kappa\right]^{m} \colon \langle \kappa, \mu^{*1}_{0}, \ldots, \mu^{*m-1}_{0} \rangle \in k(X)  \}  $$
we will prove later that $ \mathcal{E}_0 \in M_U $, and $\mbox{Ult}\left( M_U, \mathcal{E}_0 \right)$ is isomorphic to the finite iteration, with decreasing order, with $ U^{0}\vartriangleleft U^{1}\vartriangleleft \ldots \vartriangleleft U^{m-1} $.

For the rest of the subsection we provide an internal definition of $ \mathcal{E}_0 $ (namely, a definition that does not depend on $ k\colon M_U\to M $). 

We do this in a general context; thus, let us assume that $ \kappa $ is an arbitrary measurable cardinal, $ m<\omega $ is a natural number and $ U^{0}\vartriangleleft  U^{1}\vartriangleleft\ldots \vartriangleleft U^{m-1} $ are normal measures on $ \kappa $. We start by constructing $ \mathcal{E} $. 

Fix $ 0\leq i \leq m-1 $, and denote $ j = m-i $. Let us define a measure $ W_{j} $ on $ \left[\kappa\right]^{j} $. Intuitively, the ultrapower with $ W_{j} $ is the iterated ultrapower with $  U^{m-1}, \ldots, U^{i}  $ in decreasing Mitchell order. We construct $ W_{j} $ by induction on $ j $, and finally take $ \mathcal{E} = W_m $.

For $ j=1 $ (equivalently, $ i=m-1 $), set $ W_1 = U^{m-1} $. Assume  now that $ j<m $ and $ W_{j} $ was defined, such that $ M_{j} = \mbox{Ult}\left( V, W_{j} \right) $ is the model obtained by iterating $ U^{m-1}, U^{m-2}, \ldots, U^{i} $ (where $ i=m-j $). Then $ U^{i-1}\in M_{j} $. Let $ \langle \nu_{i},\ldots, \nu_{m-1} \rangle\mapsto U^{i-1}\left( \nu_{i},\ldots, \nu_{m-1} \right) $  be a function in $ V $ which represents $ U^{i} $ in the ultrapower with $ W_{j} $. Define $ W_{j+1} $ as follows:
$$ X\in W_{j+1} \iff \{  \langle \nu_{i}, \ldots, \nu_{m-1} \rangle \colon  \{ \nu_{i-1} < \nu_{i} \colon \langle \nu_{i-1}, \nu_{i},  \ldots, \nu_{m-1}  \rangle   \in X \} \in U^{i-1}\left( \nu_{i},\ldots, \nu_{m} \right)  \}\in W_{j} $$
Then $ W_{j+1} $ is as desired. 

\begin{claim}
	For every set $ X\in W_j $ there are sets $ X_{i}\in U^{i}, \ldots, X_{m-1}\in U^{m-1} $ such that $X$ contains the set of increasing sequences in  $ X_{i}\times \ldots \times X_{m-1} \in W_j $.
\end{claim}

\begin{proof}
	By induction on $ j $. For $ j=1 $ this is clear. Let us prove for $ W_{j+1} $. Let $ X\in W_{j+1} $. Then--
	$$ \{ \langle \nu_i, \ldots, \nu_{m-1} \rangle \colon \{ \nu_{i-1}<\nu_{i} \colon \langle \nu_{i-1}, \ldots, \nu_{m-1} \rangle\in X  \}\in U^{i-1}\left( \nu_{i}, \ldots, \nu_{m} \right)  \}\in W_{j}  $$
	and thus, by induction, there are sets $ X_i\in U^{i} , \ldots , X_{m-1}\in U^{m-1}$ such that every increasing sequence in $ X_i\times \ldots \times X_{m-1} $ belongs to the above set. So, for every increasing $ \langle \nu_{i}, \ldots, \nu_{m-1} \rangle\in X_{i}\times \ldots \times X_{m-1} $, There exists a set $ X\left( \nu_{i}, \ldots, \nu_{m-1} \right)\in U^{i-1}\left( \nu_{i}, \ldots, \nu_{m-1} \right) $, such that for every $ \nu_{i-1} $ in it, 
	$$ \langle \nu_{i-1}, \ldots, \nu_{m-1} \rangle\in X $$
	Let $ X_{ i-1 } = \left[  \langle \nu_{i-1}, \ldots, \nu_{m-1} \rangle\mapsto X\left( \nu_{i-1}, \ldots, \nu_{m-1} \right) \right]_{W_j} \in U^{i-1}$. Applying the induction hypothesis again, we can shrink $ X_{i},\ldots, X_{m-1} $ such that for every increasing $ \langle \nu_{i}, \ldots, \nu_{m-1} \rangle\in X_i\times \ldots \times X_{m-1} $, 
	$$ X_{i-1}\cap \nu_{i} = X\left( \nu_{i}, \ldots, \nu_{m-1} \right)\cap \nu_{i} $$
	this is true since $ \langle \nu_{i},\ldots, \nu_{m-1} \rangle\mapsto \nu_{i} $ represents $ \kappa  = \mbox{crit}\left( j_{ W_j } \right) $ modulo $ W_{j} $. Then $ X_{i-1}, \ldots, X_{m-1} $ are as desired.
\end{proof}

\begin{claim}
	The measures $ W_{j} $ are normal, in the following sense: given a sequence of sets $ \langle X_{\alpha} \colon \alpha<\kappa \rangle $, each of them in $ W_{j} $, their diagonal intersection belongs to $ W_{j} $ as well, and is defined as follows:
	$$ \underset{ \alpha<\kappa }{\triangle} X_{\alpha} = \{  \langle \nu_{i}, \ldots, \nu_{m-1} \rangle \colon \forall \alpha< \nu_{i}, \ \langle \nu_{i}, \ldots, \nu_{m-1} \rangle\in X_{\alpha} \} $$
\end{claim}

The measures $ W_j $ have also some form of ineffability, which will be useful in the next section. 

\begin{lemma} \label{Lemma: Ineffability lemma}
	Assume that $ \kappa $ is a measurable cardinal and $ U^{0}\vartriangleleft U^{1} \vartriangleleft \ldots \vartriangleleft U^n $ is a sequence of normal measures on it. For every $ \langle \nu_0, \ldots, \nu_{n} \rangle \in \left[\kappa\right]^{n+1}$, let $r\left( \nu_0, \ldots, \nu_{n} \right) $ be a function with domain $ \kappa $, such that, for every inaccessible $ \alpha<\kappa $, its restriction on $ \alpha $ has image contained $ V_{\alpha} $. Then there are:
	\begin{enumerate}
		\item Sets $ X_0\in U^{0}, \ X_1\in U^{1} , \ \ldots, \ X_n \in U^{n}$.
		\item For every $ i\leq n $ and $ \langle \nu_0, \ldots , \nu_{i} \rangle\in \left[\kappa\right]^{i+1}$,  a function $r_{i+1}\left( \nu_0, \ldots, \nu_{i} \right) $ with domain $ \kappa $.
	\end{enumerate}
	such that for every $ i<n $ and increasing sequence $ \langle \nu_0, \ldots, \nu_{n} \rangle\in X_0 \times \ldots \times X_n $, 
	$$ r\left( \nu_0, \ldots, \nu_{i}, \nu_{i+1}, \ldots, \nu_{n} \right)\restriction_{ \nu_{i+1} } = r_{i+1}\left( \nu_0, \ldots, \nu_{i} \right)\restriction_{ \nu_{i+1} } $$
\end{lemma}

\begin{remark}
	We should think of $ r\left( \nu_0, \ldots, \nu_{m-1} \right) $ as a condition in an iterated forcing $ P_{\kappa} $ of length $ \kappa $. The lemma will be useful in the next section for combining plenty of such conditions to a single condition.
\end{remark}

\begin{proof}
	Fix $ -1\leq i\leq n $ and $ \langle \nu_0, \ldots, \nu_{i} \rangle $ (where $ i=-1 $ means that the empty sequence in fixed). Let--
	$$ r_i\left( \nu_0, \ldots, \nu_{i} \right) = \left[ \langle \nu_{i+1}, \ldots, \nu_{n} \rangle\mapsto r\left(\nu_0, \ldots, \nu_{i}, \nu_{i+1}, \ldots, \nu_{n} \right)\restriction_{\nu_{i+1}}  \right]_{W_{n-i}} $$
	Then $ r_i\left( \nu_0, \ldots, \nu_{i} \right) $ is a function with domain $ \kappa $, since $\langle \nu_{i+1},  \ldots, \nu_{n} \rangle\mapsto  \nu_{i+1} $ represents $ \kappa $ in the ultrapower by $ W_{n-i} $. For every $ \alpha<\kappa $, denote--
	$$X(\nu_0, \ldots, \nu_{i})(\alpha) =\{\langle \nu_{i+1}, \ldots, \nu_{n} \rangle \colon r_i\left( \nu_0, \ldots, \nu_{i} \right)(\alpha) = r\left( \nu_0, \ldots, \nu_{i}, \nu_{i+1},\ldots, \nu_{n} \right)(\alpha) \}\in W_{n-i} $$
	where indeed $X(\nu_0, \ldots, \nu_{i})(\alpha) \in W_{n-i} $, since $r_i\left( \nu_0, \ldots, \nu_{i} \right)(\alpha) \in V_{\kappa} $. Let us apply now diagonal intersection. Then--
	$$ X\left( \nu_0, \ldots, \nu_{i} \right) = \{  \langle \nu_{i+1}, \ldots, \nu_{n} \rangle \colon r_i\left( \nu_0,\ldots, \nu_{i} \right)\restriction_{\nu_{i+1}} = r\left( \nu_0, \ldots, \nu_{n} \right)\restriction_{\nu_{i+1}} \}\in W_{n-i} $$
	
	For every $ i+1\leq k\leq n $, let $ X_{k}\left( \nu_0, \ldots, \nu_{i} \right)\in U^{k} $ be a set such that--
	$$ \prod_{i+1\leq k\leq n} X_k\left( \nu_0, \ldots, \nu_{i} \right) \subseteq X\left( \nu_0, \ldots, \nu_{i} \right)  $$
	(in the product we consider only increasing sequence. We abuse the notation and use product for simplicity).
	Such sets were constructed for every sequence $ \langle \nu_{0},\ldots, \nu_{i} \rangle $. Now just merge them. First, take--
	$$ X_0 =  X_{0}\left( \langle \rangle \right)  \in U^{0} $$
	Assuming that $ X_0, \ldots , X_{i} $ were defined, set--
	\begin{align*}
		X_{i+1} = &\underset{ \langle \nu_0, \ldots, \nu_{i} \rangle\in X_0\times \ldots \times X_i }{\triangle} X_{i+1}\left( \nu_0,\ldots, \nu_{i} \right)  =\\
		& \{ \nu_{i+1}< \kappa \colon \forall \langle \nu_{0},\ldots, \nu_{i} \rangle < \nu_{i+1},  \nu_{i+1}\in X_{i+1}\left( \nu_0,\ldots, \nu_{i} \right)  \} \in U^{i+1} 
	\end{align*}
\end{proof}

\subsection{Description of the Iteration}
Assume that $ W\in V\left[G\right] $ is a normal measure on $ \kappa $. Let $ U\in V $ be a normal measure such that $ W = U^{\times} $. Denote $\kappa^* = j_W(\kappa)$ (we will later prove that $ \kappa^* = j_{U}(\kappa) $). Let $ j_{W}\colon V\left[G\right]\to M\left[H\right] $ be the ultrapower embedding. We work by induction on $ \alpha \leq \kappa^* $ and define an iterated ultrapower $\langle M_{\alpha} \colon \alpha \leq \kappa^* \rangle$. We define as well, for every $\alpha< \kappa^* $,
\begin{enumerate}
	\item Elementary embeddings $ j_{\alpha} \colon V\to M_{\alpha} $ and $ k_{\alpha} \colon M_{\alpha}\to M $, such that $ j_W\restriction_{V} = k_{\alpha} \circ j_{\alpha} $.
	\item An ordinal $ \mu_{\alpha} = \mbox{crit}\left( k_{\alpha} \right)  $,  which is a measurable cardinal in $ M_{\alpha} $.
	\item A natural number $ 1\leq m_{\alpha} < \omega $, and a sequence of normal measures on $ \mu_{\alpha} $, $$ U^{0}_{\mu_{\alpha}} \vartriangleleft \ldots \vartriangleleft U^{m_{\alpha}-1 }_{\mu_{\alpha}} $$
	each of them belong to $ M_{\alpha} $. We also denote by $ \mathcal{E}_{\alpha} $ be the measure on $ \left[\mu_{\alpha}\right]^{m_{\alpha}-1} $ defined by taking product of the above measures, namely, a set $ X\subseteq \left[ \mu_{\alpha} \right]^{m_{\alpha}-1} $ belongs to $ \mathcal{E}_{\alpha} $ if and only if--
	$$   \{ \nu_{m_{\alpha-1}} < \mu_{\alpha} \colon \{  \ldots \{  \nu_{0}< \mu_{\alpha}  \colon \langle \nu_0, \ldots, \nu_{m_{\alpha}} \in X \rangle \}\in U^{0}_{\mu_{\alpha}  } \ldots  \}  \} \in U^{ m_{\alpha}-1 }_{\mu_{\alpha} }   $$
	Possibly $ m_{\alpha}=1 $ and then $ \mathcal{E}_{\alpha} = U^{0}_{ \mu_{\alpha} } $.
\end{enumerate}

Let us demonstrate the first two steps in $ j_W\restriction_{V} $. Recall the system $ W\cap V =U^{0} \vartriangleleft U^{1}\vartriangleleft\ldots \vartriangleleft U^{m}=U $. First, let $ M_0  = \mbox{Ult}\left( V, U \right) =\mbox{Ult}\left( V, U^{m} \right)$. Let $ k_0 \colon M_0 \to M $ be the embedding which satisfies, for every $ f\in V $,
$$ k_0\left( \left[f\right]_U \right) = \left[f\right]_{W^*} $$
$ k_0 $ is elementary since $ U\subseteq W^* $; furthermore, $ \mu_{0} =\mbox{crit}\left( k_0 \right) = \kappa $. It will also turn out that $ m_0 = m$ and $ U^{j}_{ \mu_0 } = U^{j} $ for every $ j\leq m-1 $. Thus, $ \mathcal{E}_{0 } $ is the finite iteration which corresponds to $ U^{0}\vartriangleleft \ldots \vartriangleleft U^{m-1} $ (as defined in the previous subsection). We will then define $ M_1 = \mbox{Ult}\left( M_0, \mathcal{E}_0 \right) $.

The iteration $ \langle M_{\alpha} \colon \alpha \leq \kappa^* \rangle $  is continuous, namely, for every limit $ \alpha \leq \kappa^* $, $ M_{\alpha} $ is the direct limit of $ \langle M_{\beta} \colon \beta<\alpha \rangle $. At successor steps, $ M_{\alpha+1} = \mbox{Ult}\left( M_{\alpha}, \mathcal{E}_{\mu_{\alpha} } \right)$.

For simplicity, we denote the sequence $ \left[Id\right]_{ \mathcal{E}_{ \mu_{\alpha} } } $ by $ \left[Id\right]_{\alpha} $. Arguing by induction, every element in $ M_{\alpha} $ has the form--
\begin{equation} \label{Equation: Form of elements in Malpha}
 j_{\alpha}\left( f \right)\left( j_{1,\alpha}\left( \kappa \right), j_{\alpha_0+1, \alpha} \left(\left[Id\right]_{\alpha_0}  \right), \ldots, j_{\alpha_k+1 , \alpha} \left(\left[Id\right]_{\alpha_k}\right) \right) 
\end{equation}
for some $ f\in V $ and $ \alpha_0< \ldots < \alpha_k <\alpha $.

\begin{remark}
$ M_{\alpha+1} $ can be viewed as iteration of length $ m_{\alpha} $ of $ M_{\alpha} $, in the following sense: denote--
$$ M^{m_\alpha-1  }_{\alpha} = \mbox{Ult}\left( M_{\alpha}, U^{m_{\alpha}-1}_{\mu_{\alpha} } \right) $$
$$ M^{m_{\alpha}-2}_{\alpha} = \mbox{Ult}\left(  M^{m_{\alpha}}_{\mu_{\alpha} } , U^{m_{\alpha}-2}_{\mu_{\alpha}} \right) $$
etc., up to--
$$ M^{0}_{\alpha}  = \mbox{Ult}\left(  M^{1}_{\alpha}, U^{0}_{\mu_{\alpha} } \right)$$
and take $ M_{\alpha+1} = M^{0}_{\alpha} $. We denote, for every $ 0\leq i_0 < i_1 < m_{\alpha} $,  by $ j^{i_1,i_0}_{ {\alpha} } $, the embedding from $ M^{1_1}_{\alpha} $ to $ M^{i_0}_{\alpha} $, namely, the finite iteration associated with $ U^{ i_0 }_{ \mu_{\alpha} } \vartriangleleft\ldots \vartriangleleft U^{ i_1+1 }_{ \mu_{\alpha} } $ (in decreasing order). We identify $ \left[Id\right]_{ \mathcal{E}_{ \alpha }} $ with a sequence of generators $\vec{\mu}_{\alpha} =  \langle \mu_{\alpha}, \mu^{1}_{\alpha}  , \ldots, \mu^{m_{\alpha}-1}_{\alpha} \rangle $, where, for every $ 0\leq i < m_{\alpha} $, $\mu^{i}_{\alpha} = j^{i, 0}_{ {\alpha} }\left( \mu_{\alpha} \right)$.
\end{remark}

Before we proceed, we would like to present several examples in the case where the Mitchell order is linear in $ V $.

\textbf{Example 1:} Assume that the Mitchell order on each measurable is linear in $ V $. For every $ \alpha\in \Delta $, let $ U_{\alpha, 0} $ be the unique measure on $ \alpha $ of order $ 0 $. Let $ P = P_{\kappa} $ be the Magidor iteration, where, for each $ \alpha\in \Delta $, the measure $  U^{*}_{\alpha,0} = U^{\times}_{\alpha,0} $ is taken to be $W_{\alpha}$. In $ V\left[G\right] $, consider $ W = U^{*}_{\kappa, 0} = U^{\times}_{\kappa, 0} $. In this case, $ d''\Delta\notin W $, $ m(W) =0 $ and $ j_W\restriction_{V} $ is an iterated ultrapower of $ V $, starting with $ U_{\kappa,0} $. After this step $ \kappa $ is no longer measurable. Let $ \alpha< \kappa^* = j_{U_{\kappa,0} }(\kappa) $. In $ M_{\alpha} $, $ \mu_{\alpha} $ is the least measurable  $\geq \sup\{ \mu_{\beta} \colon \beta<\alpha \} $ with cofinality above $ \kappa $ in $ V $, and $ M_{\alpha+1} = \mbox{Ult}\left( M_{\alpha}, U^{M_{\alpha}}_{ \mu_{\alpha}, 0 } \right) $ is the ultrapower with the unique measure of order $ 0 $ on $ \mu_{\alpha} $ in $ M_{\alpha} $. 

\textbf{Example 2:} Assume the same settings as in the previous example, but now $ W = U^{\times} $ for arbitrary $ U $ of order higher than $ 0 $ (below $ \kappa $ we still assume that measures of order $ 0 $ are used). We argue that now, $ m = m(W) =1 $. First, since $o(U) > 0$, $ d''\Delta\in W $, and thus $ d^{-1}\{ \kappa \} \neq \emptyset   $ in $ M\left[H\right] $. So $ m\geq 1 $. In order to prove that $ m=1 $, it suffices to prove that the following property holds in $ V\left[G\right] $: There exists a finite subset $ b\subseteq \kappa $ such that $ d $ is an injection on $ \Delta \setminus b $. Furthermore, other then finitely many, all the Prikry sequences $ G $ adds to measurables in $ \Delta $ are pairwise disjoint. Let us provide the proof. For every $ \alpha\in \Delta $, let $ C_{\alpha}\subseteq \alpha $ be the Prikry sequence added to $ \alpha $ in $ V\left[G\right] $. Then, for every $ \alpha\in \Delta $,
\begin{equation} \label{Equation for d is injective}
	\bigcup_{\beta<\alpha} C_\beta \notin U^{*}_{\alpha,0} = U^{\times}_{\alpha,0}  
\end{equation}
since otherwise there exists $ p\in G_{\alpha} $ such that $ \left(j_{U_{\alpha,0}}(p)\right)^{-\alpha} \Vdash \alpha \in \bigcup_{\beta< j_{U_{\alpha,0}} \left(\alpha\right)} \lusim{C}_{\beta}  $; but $ \left(j_{U_{\alpha,0}}(p)\right)^{-\alpha} $ forces that $ \alpha $ cannot belong to Prikry sequences of measurables above $ \alpha $, a contradiction.\\
Now we can apply equation \ref{Equation for d is injective} in a density argument: Every condition $ p$ can be direct extended to $ p^*\geq^* p $ by removing from each set $ \lusim{A}^{p}_{ \alpha } \in W_{\alpha}$ (where $ \alpha\in \Delta $) the set $ \bigcup_{ \beta<\alpha } C_\beta $. Then $ p^* $ forces that the Prikry sequences added to measurables of $ \Delta $, aside from finitely many, are pairwise disjoint. \\
Thus $ m=1 $, and the system $ U^{0} \vartriangleleft U^{1} $ consists of $ U_{\kappa,0} = U^{0} \vartriangleleft U^{1} = U $. The first step in  $ j_W\restriction_{V} $ is $ \mbox{Ult}\left( V, U \right) $, and $ U_{\kappa,0} $ is applied $ \omega $-many times to produce a Prikry sequence of critical points to $ \left[Id\right]_{W^1}  $, which is the only element in $ d^{-1}\{\kappa\} $. For every $ \alpha<\kappa^* $, $ M_{\alpha+1} = \mbox{Ult} \left( M_{\alpha},  U^{M_{\alpha}}_{ \mu_{\alpha}, 0 }\right) $ as in the previous example. The main difference is that the length of the iteration, $ \kappa^* = j_W(\kappa) = j_U(\kappa) $ is strictly higher than $ j_{U_{\kappa,0}}(\kappa) $.

\textbf{Example 3: } Assume again linearity of the Mitchell order, but now fix in advance $ m<\omega  $, and assume as well that $ o(\kappa) = m+1 $, namely, the normal measures on $ \kappa $ are $ U_{\kappa,0}\vartriangleleft U_{\kappa,1} \vartriangleleft \ldots \vartriangleleft U_{\kappa, m} $. We define the iteration $P = P_{\kappa} $ such that for every $ \alpha\in \Delta $, the measure $ W_{\alpha} $  is chosen as follows: If $ o(\alpha) = l+1 $ for some $ l\leq m $, use the measure $ W_{\alpha} =  U^{\times}_{\alpha, l} $. In $ V\left[G\right] $, let $ W = U^{\times}_{ \kappa, m } $. We argue that $ m(W) = m $. We work by induction: If $ m=0 $, then $ d''\Delta\notin W $ and thus $ m(W) =0 $. Assume that $ m\geq 1 $.  $ W $ concentrates on measurables $ \alpha \in \Delta $ such that $ W_{\alpha} = U^{\times}_{\alpha,m-1} $, and for each such $ \alpha  $, $ m\left(  W_{\alpha} \right)  = m-1 $. So $ W^* = U^{*}_{ \kappa, m } $ concentrates on measurables $ \alpha\in \Delta $ such that $ \alpha $ is the $ m $-th element in $ d^{-1}\{ d(\alpha) \} $, and thus $ m(W) = m $. It will turn out in this case that the system $ U^{0}\vartriangleleft \ldots \vartriangleleft U^{m} $ is exactly the sequence $ U_{\kappa,0}\vartriangleleft U_{\kappa,1} \vartriangleleft \ldots \vartriangleleft U_{\kappa, m} $. $ d^{-1}\{\kappa\} = \{ \left[Id\right]_{W^{1}}, \ldots , \left[Id\right]_{W^m}  \} $ contains exactly $ m $ elements, and each $ \left[Id\right]_{W_i} $ (where $ 1\leq i \leq m $) has Prikry sequence in $ M\left[H\right] $ which is generated by iterating the measure $ U_{\kappa, j-1} $ $ \omega  $-many times.

$  $

We would like to define the embedding $ k_{\alpha} \colon M_{\alpha}\to M $. We do this assuming that embeddings $ k_{\beta} \colon M_{\beta}\to M $ have been defined for every $ \beta<\alpha $. We also assume by induction that for each such $ \beta <\alpha $, a sequence $ \vec{\mu}^*_{\beta} = \langle \mu^{*0}_{\beta}, \ldots, \mu^{*m_{\beta }-1 }_{\beta} \rangle $ has been defined. We then define $ k_{\alpha}\colon M_{\alpha}\to M $ as follows: 
$$ k_{\alpha}\left( j_{\alpha}(f)\left( j_{0,\alpha}\left({{\kappa}}\right), j_{ \alpha_0+1, \alpha }\left( \left[Id\right]_{\alpha_0}\right), \ldots, j_{\alpha_k+1, \alpha}\left(\left[Id\right]_{\alpha_k}\right) \right) \right) = j_W(f)\left( \left[Id\right]_{W^*}  , \vec{\mu}^*_{\alpha_0}, \ldots, {\vec{\mu}}^*_{\alpha_k} \right) $$
for every $ f\in V $ and $ 1\leq \alpha_0< \ldots < \alpha_k $. 

We will prove by induction on $ \alpha\leq \kappa^* $ that the following properties hold:

\begin{enumerate} [label=(\Alph*)]
	\item $k_{\alpha} \colon M_{\alpha}\to M$ is elementary.
	\item Denote $ \mu_{\alpha} = \mbox{crit}\left( k_{\alpha} \right) $. Then $ \mu_{\alpha} $ is measurable in $ M_{\alpha} $. Moreover, $ \mu_{\alpha} $ is the least measurable $\mu\in  M_{\alpha}  $ which is greater or equal to $ \sup\{ \mu_{\beta} \colon \beta<\alpha \} $ and satisfies $\left( \mbox{cf}\left( \mu \right)\right)^V > \kappa  $.
	\item Let $ \mu^*_{\alpha} = k_{\alpha}\left( \mu_{\alpha} \right) $. Then $ \mu_{\alpha} $ appears as an element in the Prikry sequence of $ k_{\alpha}\left( \mu_{\alpha} \right) $ in $ H $. 
	\item Let $ \{  \mu^{*1}_{\alpha}, \ldots, \mu^{{*m}_{\alpha}-1}_{\alpha} \} $ be the increasing enumeration of $ d^{-1}\left( \mu_{\alpha} \right) $ below $ \mu^*_{\alpha} $, and denote as well  $\mu^{*0} = \mu_{\alpha} , \mu^*_{\alpha} = \mu^{m_{\alpha}}_{\alpha} $ (possibly $ m_{\alpha}=1 $ and then $ \mu_{\alpha} $ does not appear as first element in Prikry sequences of measurables below $ \mu^*_{\alpha} $). For every $ 0\leq j \leq m_{\alpha} $, there exists a measure $ U^{j}_{ \mu_{\alpha}  }\in M_{\alpha} $ on $ \mu_{\alpha} $, which satisfies--
		$$ k_{\alpha}\left( U^{j}_{\mu_{\alpha} } \right) = j_W\left(  \xi\mapsto U^{j}_{\xi} \right)\left(  k_{\alpha}\left( \mu_{\alpha} \right) \right) $$
		Moreover, 
		$$ U^{0}_{\mu_{\alpha} }\vartriangleleft  U^{1}_{\mu_{\alpha} } \vartriangleleft \ldots \vartriangleleft U^{m_{\alpha}-1 }_{\mu_{\alpha} }$$
	\item The measure $ \mathcal{E}_{\alpha} $ which corresponds to $  U^{0}_{\mu_{\alpha} }\vartriangleleft  U^{1}_{\mu_{\alpha} } \vartriangleleft \ldots \vartriangleleft U^{m_{\alpha}-1 }_{\mu_{\alpha} }$  is derived from $ k_{\alpha}\colon M_{\alpha}\to M $  in the following sense:
	$$  \mathcal{E}_{\alpha} = \{ X\subseteq \left[\mu_{\alpha} \right]^{m_{\alpha}}  \colon \langle \mu^{*0}_{\alpha} , {\mu}^{*1}_{\alpha}, \ldots, \mu^{*m_{\alpha  }-1}_{\alpha} \rangle \in k_{\alpha}\left( X \right) \}\cap M_{\alpha}$$
\end{enumerate}

The proof of the above properties goes by induction on $ \alpha $. For $ \alpha =0 $, $ k_0 \colon M_0 = M_U \to M $ is the embedding which maps each $ \left[ f \right]_U $ to $ \left[f\right]_{W^*} $; it has critical point $ \mu_0 = \kappa $. In $ M\left[H\right] $, $ \mu_0 $ appears as a first element in the Prikry sequence of $ k_0(\mu_0) = \left[Id\right]_{W^*} $, and of $ m-1 $ measurables $ {\mu_0}^{*1 } = \left[Id\right]_{W^1}, \ldots , {\mu_0}^{*m-1} = \left[Id\right]_{W^{m-1}} $. The measure $ \mathcal{E}_0\in M_0 $ derived from $ k_0 $ using $ \langle \mu_0, {\mu_0}^{*1},\ldots , {\mu_0}^{*m-1} \rangle $ is indeed the product of $ U^{0}\vartriangleleft \ldots \vartriangleleft  U^{m-1} $ by remark \ref{Remark: U^i s derived from k}.

We proceed and prove the properties for arbitrary $ 0<\alpha< \kappa^* $.

\begin{lemma}
$ k_{\alpha} \colon M_{\alpha}\to M $ is elementary, and $ j_W\restriction_{V} = k_{\alpha}\circ j_{\alpha} $.
\end{lemma}

\begin{proof}
For $ \alpha=0 $, we already argued that $ k_{0} \colon M_0\to M $ in elementary.

For simplicity, we will prove that for every $ x,y\in M_{\alpha} $, $ M_{\alpha} \vDash x\in y $ if and only if $ M\vDash k_{\alpha}(x)\in k_{\alpha}(y) $. 

Let us focus on the case where $ \alpha = \alpha'+1 $  is successor, as the limit case is simpler. There are functions $ f,g\in V $ and $ \alpha_0 < \ldots < \alpha_k <\alpha' $ such that--
$$  x = j_{\alpha}(f)\left(  j_{0,\alpha}(\kappa) , j_{\alpha_0+1, \alpha}\left( \left[Id\right]_{\alpha_0} \right), \ldots , j_{\alpha_k+1, \alpha}\left( \left[Id\right]_{\alpha_k}\right) , j_{\alpha'+1, \alpha}\left( \left[Id\right]_{\alpha'}\right) \right) $$
$$  y = j_{\alpha}(g)\left(  j_{0,\alpha}(\kappa) , j_{\alpha_0+1, \alpha}\left( \left[Id\right]_{\alpha_0} \right), \ldots , j_{\alpha_k+1, \alpha}\left( \left[Id\right]_{\alpha_k}\right) , j_{\alpha'+1, \alpha}\left( \left[Id\right]_{\alpha'}\right) \right) $$
We assumed that $ M_{\alpha} = \mbox{Ult}\left( M_{\alpha'}, \mathcal{E}_{\alpha'} \right) \vDash x\in y $, namely,
$$ \mbox{Ult}\left( M_{\alpha'}, \mathcal{E}_{\alpha'} \right) \vDash  \left[Id\right]_{\alpha'} \in j_{\alpha',\alpha}  \left( X \right) $$
where $ X $ is the set--
\begin{align*}
\left\{   \vec{\xi}    \colon \right. & \ j_{\alpha'}(f)\left(   j_{0,\alpha'}(\kappa) , j_{\alpha_0+1, \alpha'}\left( \left[Id\right]_{\alpha_0} \right), \ldots , j_{\alpha_k+1, \alpha'}\left( \left[Id\right]_{\alpha_k}\right) , \vec{\xi}  \right)   \in \\
&\left. j_{\alpha'}(g)\left(   j_{0,\alpha'}(\kappa) , j_{\alpha_0+1, \alpha'}\left( \left[Id\right]_{\alpha_0} \right), \ldots , j_{\alpha_k+1, \alpha'}\left( \left[Id\right]_{\alpha_k}\right) , \vec{\xi}  \right)   
\right\}   
\end{align*}
In particular, $ X\in \mathcal{E}_{\alpha'} $, and thus $ \vec{\mu}^{*}_{\alpha'} \in k_{\alpha'}(X) $. Since $ j_{W}\restriction_{V} = k_{\alpha'}\circ j_{\alpha'} $, it follows that--
\begin{align*}
j_{W}\left(  f  \right)\left(    \left[Id\right]_{W^*}, \vec{\mu}^{*}_{\alpha_0}, \ldots, \vec{\mu}^{*}_{ \alpha_k } , \vec{\mu}^{*}_{ \alpha' }     \right) \in  j_{W}\left(  g  \right)\left(    \left[Id\right]_{W^*}, \vec{\mu}^{*}_{\alpha_0}, \ldots, \vec{\mu}^{*}_{ \alpha_k } , \vec{\mu}^{*}_{ \alpha' }     \right)
\end{align*}
namely $ k_{\alpha}(x)\in k_{\alpha}(y) $. 
\end{proof}

We will present the proof of properties (B)-(E) is the next subsections.

\subsection{Multivariable Fusion}

Assume from now on that $ \alpha >0 $ is fixed, and we are at stage $ \alpha $ in the inductive proof of properties (A)-(E). In this subsection we develop a generalization of lemma \ref{Lemma: FS, fusion lemma} (the Fusion lemma).

Since $ \alpha>0 $, we may assume in equation \ref{Equation: Form of elements in Malpha} that $ \alpha_0 = 0 $. This will simplify some of the arguments below. We can also denote by $ \mathcal{E}_{0}^{'} $ the measure which corresponds to $ U^{0}\vartriangleleft U^{1}\vartriangleleft\ldots \vartriangleleft U^{m} $ 
(including $ U^{m} $, unlike $ \mathcal{E}_0 $) and 
$ \left[Id\right]^{'}_{0} = {\left[Id\right]_{0}}^{\frown} j_{0,1}\left(\kappa \right)$ so that every element in $ M_{\alpha} $ has the form--
\begin{equation} 
	j_{\alpha}\left( f \right)\left( j_{1,\alpha}\left( \left[Id\right]^{'}_{0} \right) , j_{\alpha_1+1, \alpha} \left(\left[Id\right]_{\alpha_1}  \right), \ldots, j_{\alpha_k+1 , \alpha} \left(\left[Id\right]_{\alpha_k}\right) \right) 
\end{equation}
for some $ f\in V $ and $ 0=\alpha_0< \alpha_{1} < \ldots < \alpha_k <\alpha $.

\begin{defn}
Let $ p\in P_{\kappa} $ be a condition. An increasing sequence $ \langle \xi, \xi^{1}, \ldots, \xi^{m} \rangle $ below $ \kappa $ is admissible for $ p $ if for every $ 1\leq i \leq m $ (if such $ i $ exists, namely $ m\geq 1 $),
$$ p\restriction_{ \xi^i } \Vdash \xi\in \lusim{A}^{\xi^{i}} \mbox{ and } \lusim{t}^{p}_{\xi^{i}} = \emptyset $$
if this holds, set--
\begin{align*} p^{\frown} \langle \vec{\xi} \rangle =& {\left({p\restriction_{\xi^{1}}}\right)^{-\xi}}^{\frown} {\langle \langle \xi \rangle, \lusim{A}^{p}_{\xi^{1}} \rangle }^{\frown} {\left({p\restriction_{ \left(\xi^{1}, \xi^{2} \right)}}\right)^{-\xi^1}}^{\frown} {\langle \langle \xi \rangle, \lusim{A}^{p}_{\xi^{2}} \rangle }^{\frown} \ldots^{\frown} \\
	&{\left({p\restriction_{ \left(\xi^{m-1}, \xi^{m} \right)}}\right)^{-\xi^{m-1}}}^{\frown} {\langle \langle \xi \rangle, \lusim{A}^{p}_{\xi^{m}} \rangle }^{\frown} \left({p\setminus \left( 
		\xi^{m}+1 \right)}\right)^{-\xi^m} 
\end{align*}
\end{defn}

In $ V\left[G\right] $, denote, for every $ \xi\in d''\Delta $ with $ \left|d^{-1}\{ \xi \}\right| = m $,  $d^{-1}\{ \xi \} = \langle   \mu^{*1}_{0}(\xi), \ldots, \mu^{*m}_{0}(\xi) \rangle  = \vec{ \mu  }^{*}_{0}(\xi) $. Then in $ M\left[H\right] $, the sequence $\left[\xi\mapsto \vec{\mu}^{*}_{0}(\xi) \right]_W$ is--
$$ \langle \kappa , \mu^{*1}_{0} , \ldots, \mu^{*m-1}_{0} , \mu^{*m}_{0}  \rangle = \langle \left[Id\right]_W , \left[ Id \right]_{W^{1}}, \ldots,  \left[Id\right]_{W^{m-1}}, \left[Id\right]_{W^m} \rangle $$

\begin{lemma} \label{Lemma: FS, Multivar Fusion for xi}
Let $ p\in P_{\kappa} $. For every increasing $\vec{\xi} = \langle \xi, \xi^{1}, \ldots, \xi^{m} \rangle $, let $ e\left( \vec{\xi} \right) $ a subset of $ P\setminus \xi  $ which is $ \leq^* $-dense open above conditions which force that $ d^{-1}\{ \xi \} = \langle \xi^{1}, \ldots, \xi^{m} \rangle $. Then there exists $ p^* \geq^* p$ and a set $ X\in \mathcal{E}^{'}_{0} $ such that for every increasing $ \langle \xi, \xi^{1}, \ldots, \xi^{m} \rangle\in X $ which is admissible for $ p^* $,
$$ p^*\restriction_{\xi} \Vdash {p^*}^{\frown} \langle \xi, \xi^{1}, \ldots, \xi^{m} \rangle\setminus \xi \in e\left( \xi, \xi^{1}, \ldots, \xi^{m} \right)$$
Furthermore, if $ p^* $ as above is chosen in $ V\left[G\right] $,
\begin{align*}
\{  \xi<\kappa \colon &  \langle \xi, \mu^{*1}_{0}(\xi), \ldots, \mu^{*m}_{0}(\xi) \rangle \mbox{ is admissible for } p^* , \  {p^*}^{\frown} \langle \xi,  \mu^{*1}_{0}(\xi), \ldots, \mu^{*m}_{0}(\xi) \rangle\in G \\
&\mbox{and }  p^*\restriction_{\xi}\Vdash {p^*}^{\frown} \langle \xi,  \mu^{*1}_{0}(\xi), \ldots, \mu^{*m}_{0}(\xi) \rangle\setminus \xi \in e\left( \xi,   \mu^{*1}_{0}(\xi), \ldots, \mu^{*m}_{0}(\xi) \right)  \}\in W  
\end{align*}
\end{lemma}

\begin{proof}
Let us assume for simplicity that $ p\geq^* 0 $. Else, just work with values of $ \xi $ above some ordinal $ \mu $ for which $ p\setminus \mu \geq^* 0 $.

For every $ \langle \xi, \xi^{1}, \ldots, \xi^{m} \rangle $, direct extend $ p(\xi, \xi^{1}, \ldots, \xi^{m}) \geq^* p $ such that for every $ 1\leq i\leq m $, $ p\left( \xi, \xi^{1}, \ldots, \xi^{m} \right) \restriction_{\xi^{i}} \parallel  \xi \in \lusim{A}^{p\left( \xi, \xi^{1}, \ldots, \xi^{m} \right)}_{\xi^{i} } $. If this is positively decided for every $ 1\leq i\leq m $, let $ r\left( \vec{\xi} \right) =  r\left( \xi, \xi^{1}, \ldots, \xi^{m} \right) \geq^* {p\left( \vec{\xi}\right)}^\frown \langle \vec{\xi} \rangle \setminus \xi $ be a condition in $e(\vec{\xi})$. 

Apply the ineffability lemma \ref{Lemma: Ineffability lemma}. There are sets $ X_0\in U^{0} $, $ X_1\in U^{1}, \ ,\ldots, \ X_{m}\in U^{m} $, such that, for every $ 0\leq j\leq m $ and an increasing sequence $ \langle \xi, \xi^{1}, \ldots, \xi^{m} \rangle\in X_0\times \ldots \times X_{m} $, there exists a condition $ r_{j+1}\left(  \xi, \xi^{1}, \ldots, \xi^{j}\right) $ such that--
$$  r\left( \xi, \xi^{1},\ldots, \xi^{m} \right)\restriction_{ \xi^{j+1} } = r_{j+1}\left( \xi, \ldots, \xi^{j} \right)\restriction_{ \xi^{j+1} } $$ 
(if $ j=m $, let $ \xi^{j+1} = \kappa $ and $ r_{j+1} = r $, and the above equality trivially holds).

The idea now will be to combine all the conditions $ r_{j+1}\left( \xi, \xi^{1}, \ldots, \xi^{j}  \right) $ into a single condition. We first define a sequence $ r^{0}\leq^* r^{1} \leq^* \ldots \leq^* r^{m-1} $ of  direct extensions of $ p $, such that, for every $ j\leq m-1 $ and $\langle  \xi, \xi^{1} , \ldots, \xi^{j}\rangle \in X_0 \times X_1\times \ldots \times X_j$ which is admissible for $ r^{j} $,
$$\left( {\left(r^{j}\right)^{\frown} \langle \xi, \xi^{1},\ldots, \xi^{j}\rangle }\right)^{ -\xi^{j} } \geq^*  r_{j+1}\left( \xi, \xi^{1},\ldots, \xi^{j}\right) $$

\begin{claim*}
	Such a sequence $ \langle r^{j} \colon j\leq m-1 \rangle $ exists.
\end{claim*}

\begin{proof}
We remark that in the proof below we omit the subindex $ j+1 $ in conditions of the form $ r_{j+1}\left( \xi, \xi^{1}, \ldots, \xi^{j} \right) $ (the index can be calculated from the number of variables).\\
\textbf{Construction for $\textbf{j=0} $:} Our goal is to define a condition $ r^{0}\geq^* p $ which has the following property: for every $ \xi\in X_0 $, $ \left(r^{0}\right)^{-\xi}\geq^* r\left( \xi \right) $. Such a condition can be constructed in a similar manner as in the Fusion lemma: we first construct a sequence $ \langle q_{ \eta } \colon \eta< \kappa \rangle $ of direct extensions of $ p $, which satisfy, for every $ \eta < \mu < \kappa$,
	$$ q_{\mu}\restriction_{\eta} = q_{\eta}\restriction_{\eta}  \ \mbox{ and } \ \left(  q_{\mu} \right)^{-\eta} \geq^* q_{\eta}$$
	and such that, for every $ \xi \in X_0 $, $ q_{\xi} \setminus \xi \geq^* r\left( \xi \right) $. Then, take-- 
	$$ r^{0}= \underset{ \xi < \kappa  }{ \bigcup } q_{\xi}\restriction_{\xi} $$
	\textbf{Construction for $ \textbf{j=1} $:} Before we proceed to the general case (for arbitrary $ j\leq m-1 $), let us construct $ r^{1}\geq^* r^{0} $. We construct it such that it has the following property:
	For every $ \langle \xi,\xi^{1} \rangle\in X_{0}\times X_{1} $ which is admissible for $ r^{1} $, 
	$$ \left(\left(r^{1}\right)^{\frown} \langle \xi ,\xi^{1}  \rangle\right)^{ -\xi^{1} } \geq^* r\left( \xi, \xi^{1} \right) $$
	First, note that for every $ \langle \xi, \xi^{1} \rangle\in X_0 \times X_1 $,
	$$ \left(  r^{0}\restriction_{ \xi^{1} } \right)^{-\xi} \geq^* r\left( \xi \right)\restriction_{\xi^{1}} = r\left( \xi, \xi^{1} \right)\restriction_{ \xi^{1} } $$
	Now, we apply Fusion. We construct a sequence $ \langle q_{\eta} \colon \eta<\kappa \rangle $ of direct extensions of $ r^{0} $. As before, we require that for every $ \eta<\mu< \kappa $,
	$$ q_{\mu}\restriction_{\eta} = q_{\eta}\restriction_{\eta}  \ \mbox{ and } \ \left(  q_{\mu} \right)^{-\eta} \geq^* q_{\eta}$$
	In addition, we require that if $ \xi^{1}\in X_{1} $,
	\begin{align*}
		q_{\xi^{1}}\restriction_{\xi^{1}} \Vdash &\forall \xi\in \lusim{A}^{q}_{\xi^{1}}\cap X_0,  \mbox{ if } r\left( \xi, \xi^{1} \right)\restriction_{\xi^{1}}\in \lusim{G}_{\xi^{1}}, \mbox{ then--} \\ 
		&{\langle \langle \xi \rangle, A^{q}_{\xi^{1}}\setminus \xi+1 \rangle}^{\frown}\left( q_{\xi^{1}}\setminus \left(\xi^{1}+1\right)\right) \geq^* r\left( \xi, \xi^{1} \right)\restriction_{g_{i+1}} 
	\end{align*}
	Indeed, fix $ \xi^{1} \in  X_1 $, and let us argue that $ q_{\xi^{1}} $ can be constructed as described above. Take $ q_{\xi^{1}}\restriction_{\xi^{1}} = \bigcup \{ q_{ \eta }\restriction_{\eta} \colon \eta<\xi^{1} \} $. Fix a generic set $ G_{\xi^{1}}\subseteq P_{\xi^{1}} $ with $ q_{\xi^{1}}\restriction_{\xi^{1}} \in G_{\xi^{1}} $, and let us describe $ q_{\xi^{1}}\setminus \xi^{1} $. Fix $ \xi\in A^{r^{0}}_{\xi^{1}} \cap X_0$. If there exists $ \mu \in \left( \xi, \xi^{1} \right) $  with $ d\left( \mu \right) \leq \xi $, do nothing. Else, it follows that $ \left( r^{0}\restriction_{\xi^{1}} \right)^{-\xi} \in G_{\xi^{1}} $. In particular, $ r\left(  \xi, \xi^{1} \right)\restriction_{ \xi^{1} }\in G_{\xi^{1}} $. Thus $ r\left( \xi, \xi^{1} \right)\setminus \xi^{1} $ is a legitimate direct extension of $ \langle \langle \xi \rangle, A^{r^{0}}_{\xi^{1}}  \rangle^{\frown} r^{0}\setminus \left(\xi^{1}+1\right) $ in $ V\left[G_{\xi^{1}}\right] $. It follows that we can choose-- 
	$$ A^{q_{\xi^{1}}}\left( \xi^{1} \right) =  A^{r^{0}}_{\xi^{1}} \cap \underset{ \xi \in A^{r^{0}}_{\xi^{1}} \cap X_0   }{\triangle} A^{ r\left( \xi, \xi^{1} \right) }_{\xi^{1}} $$
	(where, in the above equation, whenever $ \xi $ satisfies that there exists  $ \mu \in \left( \xi, \xi^{1} \right) $ with $ d\left( \mu \right)\leq \xi $, we take $A^{r\left( \xi, \xi^{1} \right)}_{\xi^{1}} = A^{r}_{\xi^{1}}$). \\
	So $ q_{\xi^{1}}(\xi^{1}) $ has been constructed, and we can proceed and define  $ q_{\xi^{1}}\setminus \left( \xi^{1}+1 \right) $. We take it to be $ r\left( \lusim{\xi}, \xi^{1} \right)\setminus \left( \xi^{1}+1 \right) $, where $ \lusim{\xi} = d\left( \xi^{1} \right) $. This concludes the construction of $ q_{\xi^{1}} $. 
	
	Applying the standard Fusion argument, we can choose $ r^{1}\geq^* r^{0} $, such that for every $ \xi^{1}< \kappa $, $ \left(r^{1}\right)^{-\xi^{1}} \geq^* q_{\xi^{1}} $. \\
	\textbf{Construction for arbitrary $ \textbf{j} $:} Assume that $ r^{j-1} $ has been constructed, and let us define $ r^{j}\geq^* r^{j-1} $, such that, for every $ \langle \xi, \xi^{1}, \ldots, \xi^{j} \rangle\in X_0 \times \ldots \times X_{j} $ which is admissible for $ r^{j} $,
	$$  \left(  \left(  r^{j} \right)^{\frown} \langle \xi, \xi_1, \ldots, \xi^{j} \rangle  \right)^{-\xi^{j}} \geq^* r\left( \xi, \xi^{1}, \ldots, \xi^{j} \right) $$
	Let us construct a sequence $ \langle q_{\eta} \colon \eta<\kappa \rangle $ of direct extensions of $ r^{j-1} $, similarly as before; however, we now require that whenever $ \xi^{j}\in X_{j} $,
	\begin{align*}
		q_{\xi^{j}}\restriction_{\xi^{j}} \Vdash &\forall \xi\in \lusim{A}^{q}_{\xi^{j}}\cap X_{j}, \mbox{ if, in } V\left[G_{\xi^{j}}\right], \  d^{-1}\{\xi\} = \{ \xi^{1}, \ldots, \xi^{j-1} \}, \ \mbox{and }\\ 
		&r\left(  \xi, \xi^{1}, \ldots, \xi^{j-1}, \xi^{j} \right)\restriction_{\xi^{j}}\in \lusim{G}_{\xi^{j}}, \mbox{ then-}\\
		&\langle \langle \xi \rangle, A^{q}_{\xi}\setminus \left(\xi+1\right) \rangle^{\frown} \left(  q_{\xi^{j}} \setminus \left( \xi^{j}+1 \right)  \right) \geq^* r\left(  \xi, \xi^{1},\ldots, \mu^{j-1}, \xi^{j} \right) \setminus \xi^{j}
	\end{align*}
	Now construct $ r^{j} = \bigcup_{\eta<\kappa} q_{\eta}\restriction_{\eta} $, such that for every $ \xi^{j}\in X_{j}$, $ \left(r^{j}\right)^{-\xi^{j}} \geq^* q_{\xi^{j}} $. The condition $ r^{j} $ defined above is as desired.
\end{proof}

Now, let $ p^* = r^{m-1} $. It has the following property: for every increasing  $ \langle \xi, \xi^{1}, \ldots, \xi^{m-1} \rangle\in X_0\times X_1 \times \ldots \times X_{m-1} $, which is admissible for $ p^* $,
$$\left(p^*\right)^{\frown}  \langle \xi, \xi^{1},\ldots, \xi^{m-1} \rangle \geq^* r\left( \xi, \xi^{1},\ldots, \xi^{m-1} \right) $$
This proves the first part of the lemma. We now proceed to the second part, and assume that $ p^* \in G $. $ p^* $ satisfies that for every increasing sequence $ \langle \xi, \xi^{1}, \ldots, \xi^{m} \rangle\in X_0 \times \ldots \times X_m $, and for every $ 1\leq i \leq m $, $ {p^*}^{\frown} \langle \vec{\xi} \rangle\restriction_{ \xi^{i} } $ decides whether $ \xi \in  \lusim{A}^{p^*}_{\xi^{i}} $. Let us argue that--
$$ \{ \xi<\kappa \colon \langle \vec{\mu}^{*}_0(\xi) \rangle  \mbox{ is admissible for } p^* \}\in W $$
Take $ X\in W $ such that $ X\subseteq X_0 \cap d''X_1 \cap \ldots \cap d'' X_m $. For every $ \xi\in X $, $ \vec{\mu}^{*}(\xi) \in X_0\times X_1\times \ldots \times X_m $. Then $ X $ can be shrinked to a set in $ W $ for which the above decisions are positive with respect to $ \vec{\xi} = \vec{\mu}^{*}_{0}(\xi) $:  Indeed, otherwise, in $ M\left[H\right] $, it would not hold that $ d^{-1}\{ \kappa\} = \langle \mu^{*1}_{0}, \ldots, \mu^{*m}_{0} \rangle $.

Finally, let us verify that $  \{ \xi<\kappa \colon  {p^*}^{\frown} \langle \vec{\mu}^*_{0}(\xi)  \rangle \in G \}\in W $. We need to verify that for a set of $ \xi $-s in $ W $ the following holds: for every $ 1\leq i \leq m $ and for every measurable $ \mu\in \left(\mu^{*i-1}(\xi), \mu^{*i}(\xi) \right)  $, $ d(\mu) > \mu^{*i-1}(\xi) $.

Recall following property from the proof of lemma \ref{Lemma: FS, jW of p minus idW* belongs to H}: If $ d''\Delta \in W $ (namely $ m>0 $; if $ m=0 $ there is nothing to prove), then there exists a finite subset $ b\subseteq \kappa $ such that for every measurable $ \mu > \sup(b) $,
$$  d(\mu) \notin \bigcup_{ \xi\in \Delta \cap \mu } \left(  d(\xi), \xi \right] $$
from now on, we consider values of $ \xi $ above $ \sup(b) $, such that $ d^{-1}\{\xi\} $ contains only measurables $ \mu $ for which the above holds (the set of such $ \xi $-s is clearly in $ W $).  Let $ \mu \in \left(\mu^{*i-1}(\xi), \mu^{*i}(\xi) \right)  $. First,  note that $ \mu< \mu^{*i}_{\xi} $, and thus--
$$ \xi = d\left( \mu^{*i}(\xi) \right)\notin \left( d(\mu), \mu \right] $$
so $ d(\mu) \geq\xi $, and thus $ d(\mu)> \xi $. This also proves the desired property for $ i=1 $. Assume now that $ i>1 $. Then $ \mu> \mu^{*i-1}(\xi) $ and thus--
$$ d(\mu)\notin \left(  d\left(  \mu^{* i-1}(\xi) \right) , \mu^{* i-1}(\xi)   \right] = \left( \xi, \mu^{* i-1}(\xi)  \right] $$ 
since we already proved that $ d(\mu)> \xi $, it follows that $ d(\mu) > \mu^{*i-1}(\xi) $, as desired.
\end{proof}

\begin{defn}
Fix $ 0<\alpha\leq\kappa^* $. An increasing sequence $ 0=\alpha_0<\alpha_1< \ldots <\alpha_k  $ below $ \alpha $ is called nice, if the are functions $ g_i, f_i, F^{j}_i $ for every $ 1\leq i\leq k $ and $ 0\leq j \leq m_{ \alpha_i } $ such that--
$$ \mu_{\alpha_1} = j_{\alpha_1}\left( g_1 \right)\left(  j_{1,\alpha_1}\left(  \left[Id\right]'_{0} \right) \right) $$
$$ t_{\alpha_1} = j_{\alpha_1}\left(  f_0 \right) \left(  j_{1,\alpha_1}\left(  \left[Id\right]'_{0} \right)\right) $$
	$$ U^{j}_{\mu_{\alpha_1} } = j_{\alpha_1}\left( F^{j}_{0} \right)\left(   j_{1,\alpha_1}\left(  \left[Id\right]'_{0} \right) \right)   \ \ \left( 0\leq j < m_1 \right)$$
and, for every $ 1\leq i < k $,
$$  \mu_{\alpha_{i+1}} = j_{\alpha_{i+1}}\left( g_{i+1} \right)\left(  j_{1,\alpha_1}\left( \left[Id\right]'_0 \right),j_{\alpha_1,\alpha_{i+1}  } \left( \left[Id\right]_{\alpha_1} \right), \ldots, j_{ \alpha_{i}, \alpha_{i+1}  }  \left( \left[Id\right]_{\alpha_i}\right) \right) $$
$$  t_{\alpha_{i+1}} = j_{\alpha_{i+1}}\left( f_{i+1} \right)\left( j_{1,\alpha_1}\left( \left[Id\right]'_0 \right),j_{\alpha_1,\alpha_{i+1}  } \left( \left[Id\right]_{\alpha_1} \right), \ldots, j_{ \alpha_{i}, \alpha_{i+1}  }  \left( \left[Id\right]_{\alpha_i}\right) \right)  $$
$$  U^{j}_{\mu_{\alpha_{i+1}}} = j_{\alpha_{i+1}}\left( F^{j}_{i+1} \right)\left(   j_{1,\alpha_1}\left( \left[Id\right]'_0 \right),j_{\alpha_1,\alpha_{i+1}  } \left( \left[Id\right]_{\alpha_1} \right), \ldots, j_{ \alpha_{i}, \alpha_{i+1}  }  \left( \left[Id\right]_{\alpha_i}\right) \right) $$
\end{defn}

Working by induction, we can define now functions, in $ V\left[G\right] $, which represent the cardinals $ \vec{\mu}^{*}_{\alpha_i} $ in $ \mbox{Ult}\left( V\left[G\right], W \right) $.

\begin{enumerate}
	\item For every $ \xi<\kappa $ with $ \left|d^{-1}(\xi)\right| = m $, recall the sequence  $ \langle {\mu}_0^{*1}(\xi), \ldots, {\mu}_0^{*m}(\xi)\rangle $ which is the increasing enumeration of $ d^{-1}( \xi) $. Denote--
	$$ \vec{\mu}^{*}_0(\xi)  = \langle \xi, {\mu}_0^{*1}(\xi), \ldots, {\mu}_0^{*m}(\xi) \rangle $$
	Then in $ M\left[H\right] $, the sequence $\left[\xi\mapsto \vec{\mu}^{*}_{0}(\xi) \right]_W$ is--
	$$ \langle \kappa = \mu_0 , \mu^{*1}_{0} , \ldots, \mu^{*m-1}_{0} , \mu^{*m}_{0}  \rangle = \langle \left[Id\right]_W , \left[ Id \right]_{W^{1}}, \ldots,  \left[Id\right]_{W^{m-1}}, \left[Id\right]_{W^m} \rangle $$
	\item Given $ \xi<\kappa $, let $ \mu_{\alpha_1}(\xi) $ be the $ n_0 $-th element in the Prikry sequence of $g_0 = g_0(\vec{\mu}^{*}_{0}(\xi) )  $. Let $ \langle \mu^{*1}_{\alpha_1}(\xi), \ldots, \mu^{*m_{ \alpha_1 }-1 }_{\alpha_1}(\xi) \rangle $ be the increasing enumeration of $ d^{-1}\left( \mu_{\alpha_1}(\xi)  \right) $ below $ g_0 $. Denote $ \mu^{*m_{\alpha_1 }}_{\alpha_1}(\xi)  = g_0$ and--
	$$ \vec{\mu}^{*}_{\alpha_1}(\xi) = \langle \mu_{\alpha_1}(\xi) , \mu^{*1}_{\alpha_1}(\xi), \ldots, \mu^{ *m_1 -1}_{\alpha_1}(\xi) \rangle $$ 
	Then in $ M\left[H\right] $,
	$$  \left[\xi\mapsto  \vec{\mu}^{*}_{ \alpha_1 }(\xi) \right]_{W} = \langle \mu_{\alpha_1} , \mu^{*1}_{\alpha_1}, \ldots, \mu^{*m_1-1 } \rangle $$
	namely, this is the sequence which includes $ \mu_{\alpha_1} $, concatenated with the increasing enumeration of $ d^{-1}\{  \mu_{\alpha_1} \} $ in $ M\left[H\right] $.
	\item Assuming that $ 0\leq j<k $ and the functions $ \vec{\mu}^{*}_{\alpha_0}(\xi), \ldots, \vec{\mu}^{*}_{\alpha_j}(\xi) $ have been defined, let $ \mu_{\alpha_{j+1}}(\xi) $ be the $ n_{j+1} $-th element in the Prikry sequence of $ g_{j+1}\left( \vec{\mu}^{*}_{0}(\xi), \vec{\mu}^{*}_{\alpha_1}(\xi), \ldots, \vec{\mu}^{*}_{\alpha_j}(\xi) \right) $ in $ M\left[H\right] $. Let $ \langle \mu^{*1}_{\alpha_{j+1}}(\xi) , \ldots, \mu^{  * m_j-1  }_{\alpha_{j+1}}(\xi)\rangle $ be the increasing enumeration of $ d^{-1}\left( \mu_{\alpha_{j+1}}(\xi) \right) $ below-- 
	$$ \mu^{ * m_{j } }_{\alpha_{j+1}}(\xi) = g_{j+1}\left( \vec{\xi}, \vec{\mu}^{*}_{\alpha_0}(\xi), \ldots, \vec{\mu}^{*}_{\alpha_j }(\xi) \right) $$ 
	Also, denote--
	$$ \vec{\mu}^{*}_{  \alpha_{ j+1 } }(\xi) = \langle \mu^{*1}_{\alpha_{j+1}}(\xi), \ldots, \mu^{ * m_{ j+1 } -1 }_{ \alpha_{j+1} }(\xi) \rangle $$
\end{enumerate}

Denote $ \mu_{\alpha} = \mbox{crit}\left( k_{\alpha} \right) $. Write $ \mu_{\alpha} =j_{\alpha}(h)\left( \kappa, \mu_{\alpha_0}, \ldots, \mu_{\alpha_k} \right) $, where $ \alpha_0 < \ldots < \alpha_{k} <\alpha $ is a nice sequence. Let $ m_0, \ldots, m_k $ be such that $ m_i = m_{\alpha_i} $. Denote $ m = m_{0} $. Let $ g_i,f_i, F_i,F^{j}_{i} $ be functions as above. 

Note that, by induction, $ \left[\xi\mapsto F^{j}_{i+1}\left(\xi,  \mu_{\alpha_0}(\xi), \ldots, \mu_{\alpha_{i}}(\xi)  \right) \right]_W = \left[\xi \mapsto U^{j}_{ \mu_{\alpha_{i+1}  }(\xi) } \right]_W$ for every $ 0\leq i \leq k $ and $ 0\leq j \leq m_{\alpha_{i+1}} $. Thus, for a set of $ \xi $-s in $ W $, 
$$  F^{j}_{i+1}\left(\xi,  \mu_{\alpha_0}(\xi), \ldots, \mu_{\alpha_{i}}(\xi)  \right)= U^{j}_{ \mu_{\alpha_{i+1}  }(\xi) } $$

\begin{defn}
Fix a nice sequence $ 0= \alpha_0 < \alpha_1 <\ldots < \alpha_k $ below $ \alpha $. Given a condition $ p\in P_{\kappa} $ and a sequence of increasing sequences--
$$ \langle \vec{\xi}, \vec{\nu}_1, \ldots, \vec{\nu}_{k} \rangle =  \langle  \langle \xi, \xi^{1}, \ldots, \xi^{m} \rangle, \langle \nu_1, \nu^{1}_{1}, \ldots, \nu^{m_0-1}_{1} \rangle, \langle \nu_{2} ,\nu^{1}_{2}, \ldots, \nu^{m_1-1}_{2} \rangle, \ldots \ldots , \langle \nu_{k}, \nu^{1}_{k},\ldots ,\nu^{m_k-1}_{k}\rangle  \rangle $$ 
we define whenever $  \langle \vec{\xi}, \vec{\nu}_1, \ldots, \vec{\nu}_{k} \rangle  $ is admissible for $ p $, and in this case, we define an extension $ p^{\frown} \langle \xi, \vec{\nu}_1, \ldots, \vec{\nu}_{k} \rangle \geq p $. 
\begin{enumerate}
	\item An increasing sequence $ \vec{\xi} = \langle \xi,  \xi^{1}, \ldots, {\xi}^{m} \rangle $ is admissible for $ p $ if for every $ 1\leq i \leq m $ (if such $ i $ exists, namely $ m\geq 1 $),
	$$ p\restriction_{ \xi^i } \Vdash \xi\in \lusim{A}^{\xi^{i}} \mbox{ and } \lusim{t}^{p}_{\xi^{i}} = \emptyset $$
	if this holds, set--
	\begin{align*} p^{\frown} \langle \vec{\xi} \rangle =& {\left({p\restriction_{\xi^{1}}}\right)^{-\xi}}^{\frown} {\langle \langle \xi \rangle, \lusim{A}^{p}_{\xi^{1}} \rangle }^{\frown} {\left({p\restriction_{ \left(\xi^{1}, \xi^{2} \right)}}\right)^{-\xi^1}}^{\frown} {\langle \langle \xi \rangle, \lusim{A}^{p}_{\xi^{2}} \rangle }^{\frown} \ldots^{\frown} \\
	&{\left({p\restriction_{ \left(\xi^{m-1}, \xi^{m} \right)}}\right)^{-\xi^{m-1}}}^{\frown} {\langle \langle \xi \rangle, \lusim{A}^{p}_{\xi^{m}} \rangle }^{\frown} \left({p\setminus \left( 
	\xi^{m}+1 \right)}\right)^{-\xi^m} 
	\end{align*}
	\item Let $  1 \leq i <k$. Assume that $ \langle \vec{\xi}, \vec{\nu}_1, \ldots, \vec{\nu}_{i} \rangle $ is admissible for $ p $ and $q = p^{\frown} \langle \vec{\xi}, \vec{\nu}_1, \ldots, \vec{\nu}_{i} \rangle $ has been defined. Denote-- 
	$$ g_{i+1} = g_{i+1}\left( \vec{\xi}, \vec{\nu}_{1}, \ldots, \vec{\nu}_{i} \right) $$
	$$ t_{i+1} = t_{i+1}\left( \vec{\xi}, \vec{\nu}_1, \ldots, \vec{\nu}_{i} \right) $$
	$$ F^{j}_{i+1} = F^{j}_{i+1}\left( \vec{\xi}, \vec{\nu}_1, \ldots, \vec{\nu}_{i} \right) \ \left( 0\leq j < m_{ \alpha_{i+1} } \right) $$
	We say that $ \langle \vec{\xi}, \vec{\nu}_1, \ldots, \vec{\nu}_{i+1} \rangle $ is admissible for $ p $ if for every $ 1\leq j< m_i $ (if such $ j $ exists),
	\begin{align*}	
		q \restriction_{ \nu^{j}_{i} } \Vdash \nu_{i}\in \lusim{A}^{q}_{\nu^{j}_{i}} \mbox{ and } t^{q}_{\nu^{j}_{i}} = \emptyset 
	\end{align*}
	and-- 
	\begin{align*}
	q\restriction_{g_{i+1}} \Vdash & \langle {t_{{i+1} } }^{\frown} \langle \nu_{i} \rangle, \lusim{A}^{q}_{g_{i+1}} \rangle \mbox{ is compatible with } q\left( g_{i+1} \right), \mbox{ and } \langle F^{j}_{i+1} \colon j\leq m_{\alpha_{i+1}} \rangle \\
	&\mbox{is the system of measures } \langle U^{j}_{ g_{i+1} } \colon j\leq m\left( U^{\times}_{g_{i+1}} \right) \rangle.
	\end{align*}
	if this holds, and $ t^{q}_{g_{i+1}} $ is an initial segment of $ {t_{i+1}}^{\frown} \langle \nu_{i+1} \rangle $, let--
	\begin{align*}
	p^{\frown} \langle \vec{\xi}, \vec{\nu}_0, \ldots, \vec{\nu}_{i+1} \rangle = & q^{\frown} \langle \vec{\nu}_{i+1} \rangle\\
	= &  	{  \left(q\restriction_{\nu^{1}_{i+1} }\right)^{-\nu_{i+1}}  }^{\frown} { \langle \langle \nu_{i+1} \rangle, \lusim{A}^{q}_{\nu^{1}_{i+1} }  \rangle  }^{\frown} { \left(q\restriction_{  \left( \nu^{1}_{i+1}, \nu^{2}_{i+1} \right) } \right)^{-\nu^{1}_{i+1} }}  ^{\frown}  \\
	& {\langle \langle \nu_{i+1} \rangle, \lusim{A}^{q}_{\nu^{2}_{i+1} } \rangle}^{\frown}\ldots ^{\frown} { \langle {t_{i+1}}^{\frown} \langle \nu_{i+1} \rangle, \lusim{A}^{q}_{  \nu^{m_{i+1} }_{i+1}  } \rangle   }^{\frown} \left(q\setminus \left( g_{i+1}+1 \right)\right) 
	\end{align*}
	else, set $ p^{\frown} \langle \vec{\xi}, \vec{\nu}_1, \ldots, \vec{\nu}_{m+1} \rangle = p^{\frown} \langle \vec{\xi}, \vec{\nu}_1, \ldots, \vec{\nu}_{m} \rangle $.
\end{enumerate}
Given $ i<\omega $ and a condition $ p $ which forces that--
$$ \langle \vec{\xi}, \vec{\nu}_1, \ldots, \vec{\nu}_{k} \rangle =  \langle   \vec{\mu}^{*}_0(\xi), \vec{\mu}^{*}_1(\xi), \ldots, \vec{\mu}^{*}_i(\xi) \rangle $$ 
We define, similarly to above, whenever a sequence 
$$ \langle  \langle \nu_{i+1}, \nu^{1}_{i+1}, \ldots, \nu^{m_{i+1}-1}_{i+1} \rangle,  \ldots \ldots , \langle \nu_{k}, \nu^{1}_{k},\ldots ,\nu^{m_k-1}_{k}\rangle  \rangle $$ 
is admissible for $ p $ above $ \langle \vec{\xi}, \vec{\nu}_1, \ldots, \vec{\nu}_{k} \rangle  $. If this is the case, we can define similarly the condition $ p^{\frown} \langle  \langle \nu_{i+1}, \nu^{1}_{i+1}, \ldots, \nu^{m_{i+1}-1}_{i+1} \rangle,  \ldots \ldots , \langle \nu_{k}, \nu^{1}_{k},\ldots ,\nu^{m_k-1}_{k}\rangle  \rangle $.
\end{defn}

\begin{thm} [Multivariable Fusion]
Let $ p\in P_{\kappa} $ be  a condition and, for every sequence $ \langle \vec{\xi}, \vec{\nu}_1, \ldots, \vec{\nu}_k \rangle $, let $ e\left( \vec{\xi}, \vec{\nu}_1, \ldots, \vec{\nu}_k \right)$ be a subset of $ P\setminus \left(\nu_k+1\right) $ which is $ \leq^* $ dense open above any condition which forces that $ \langle \vec{\mu}^{*}_0(\xi), \vec{\mu}^{*}_{\alpha_1}(\xi), \ldots, \vec{\mu}^{*}_{\alpha_k}(\xi) \rangle = \langle \vec{\xi},  \vec{\nu}_1,\ldots, \vec{\nu}_k \rangle $. Then there exists $ p^*\geq^* p $ and a set $ X\in \mathcal{E}^{'}_{0} $, such that for every sequence of increasing sequences,
$$\langle \vec{\xi}, \vec{\nu}_1,\ldots, \vec{\nu}_{k} \rangle =   \langle \langle \xi, \xi^{1}, \ldots, \xi^{m} \rangle, \langle \nu_1, \nu^{1}_{1}, \ldots, \nu^{m_1-1}_{1} \rangle, \ldots, \langle \nu_k, \nu^{1}_{k}, \ldots, \nu^{m_k-1}_{k} \rangle \rangle $$
which is admissible for $ p^* $, and such that $ \langle \xi, \xi^{1}, \ldots, \xi^{m} \rangle\in X $,
$$ {{p^*}^{\frown} \langle \vec{\xi}, \vec{\nu}_1, \ldots, \vec{\nu}_{k} \rangle}\restriction_{ \nu_{k} } \Vdash {p^*}^{\frown} \langle \vec{\xi}, \vec{\nu}_1, \ldots, \vec{\nu}_{k} \rangle\setminus \nu_k \in e\left( \vec{\xi}, \vec{\nu}_1, \ldots, \vec{\nu}_k \right) $$
Furthermore, there exists $ p^*\in G $, for which--
\begin{align*}
\{  \xi<\kappa \colon &\langle \vec{\xi}, \vec{\mu}_{\alpha_1}(\xi),\ldots, \vec{\mu}_{\alpha_k}(\xi) \rangle \mbox{ is admissible for } p^*, \\
& {p^*}^{\frown}\langle \vec{\xi}, \vec{\mu}^{*}_{\alpha_1}(\xi),\ldots, \vec{\mu}^{*}_{\alpha_k}(\xi) \rangle \in G  \mbox{ and } \\
& {{p^*}^{\frown} \langle \vec{\xi}, \vec{\mu}^{*}_{\alpha_1}, \ldots, \vec{\mu}^{*}_{\alpha_k} \rangle}\restriction_{ \mu_{\alpha_k} } \Vdash {p^*}^{\frown} \langle \vec{\xi}, \vec{\mu}^{*}_{\alpha_1}, \ldots, \vec{\mu}^{*}_{\alpha_k} \rangle\setminus \mu_{\alpha_k} \in e\left( \vec{\xi}, \vec{\mu}^{*}_{\alpha_1}, \ldots, \vec{\mu}^{*}_{\alpha_k} \right) 
 \}\in W 
\end{align*}
\end{thm}

\begin{proof}
For every $ 1\leq i\leq k $ and a sequence $ \langle \vec{\xi},\ldots, \vec{\nu}_{i} \rangle $, we define a set $ e\left( \vec{\xi}, \vec{\nu}_1, \ldots, \vec{\nu}_{i} \right) $, which is $ \leq^* $ dense open above conditions which force that-- 
\begin{align*}
&\langle d^{-1}\{ \xi \},  \langle \mu_{\alpha_1}(\xi), \mu^{*1}_{\alpha_1}(\xi),\ldots, \mu^{*m_1-1}_{\alpha_1}(\xi) \rangle, \ldots, \langle \mu_{\alpha_i}(\xi), \mu^{*1}_{\alpha_i}(\xi), \ldots, \mu^{*m_i-1}_{\alpha_i}(\xi)\rangle \rangle =\\
&\langle \langle \xi^{1}, \ldots, \xi^{m} \rangle , \langle \nu_1, \nu^{1}_1, \ldots, \nu^{m_1-1}_{1} \rangle , \ldots, \langle \nu_i, \nu^{1}_{i}, \ldots, \nu^{m_i-1}_{i} \rangle \rangle 
\end{align*}as follows:
\begin{align*}
e\left( \vec{\xi}, \vec{\nu}_1, \ldots, \vec{\nu}_{i} \right) = &\{ r\in P\setminus \nu_{i} \colon  \mbox{ for every } \langle \vec{\nu}_{i+1},\ldots, \vec{\nu}_{k} \rangle \mbox{ which is admissible for }\\
& r \mbox{ above } \langle \vec{\xi}, \vec{\nu}_1, \ldots, \vec{\nu}_{i} \rangle, r^{\frown} \langle \vec{\nu}_{i+1}, \ldots, \vec{\nu}_{k} \rangle\restriction_{\nu_{k}} \Vdash \\ & r^{\frown} \langle \vec{\nu}_{i+1}, \ldots, \vec{\nu}_{k} \rangle\setminus  \nu_{k} \in e\left( \vec{\xi}, \vec{\nu}_1, \ldots, \vec{\nu}_{k} \right)
 \}
\end{align*}

\begin{lemma} \label{Lemma: FS, main lemma for multivar Fusion}
If $ e\left( \vec{\xi}, \vec{\nu}_1, \ldots, \vec{\nu}_{i}, \vec{\nu}_{i+1} \right) $ is $ \leq^* $-dense open above conditions which force that--
\begin{align*}
	\langle \vec{\mu}^{*}_0(\xi), \vec{\mu}^{*}_{\alpha_1}(\xi), \ldots, \vec{\mu}^{*}_{\alpha_{i+1}}(\xi)  \rangle = \langle \vec{\xi},  \vec{\nu}_1,\ldots, \vec{\nu}_{i}, \vec{\nu}_{i+1} \rangle
\end{align*}
then $ e\left( \vec{\xi}, \vec{\nu}_1, \ldots, \vec{\nu}_{i} \right) $ is $ \leq^* $-dense open above conditions which force that--
\begin{align*}
	\langle \vec{\mu}^{*}_0(\xi), \vec{\mu}^{*}_{\alpha_1}(\xi), \ldots, \vec{\mu}^{*}_{\alpha_{i}}(\xi)  \rangle = \langle \vec{\xi}, \vec{\nu}_1,\ldots, \vec{\nu}_{i}\rangle
\end{align*}

\end{lemma}

\begin{proof}
Fix $ \langle \vec{\xi}, \vec{\nu}_1, \ldots, \vec{\nu}_{i}\rangle $. Let $ r\in P\setminus \nu_i $ be a condition which forces that--
\begin{align*}
	\langle \vec{\mu}^{*}_0(\xi), \vec{\mu}^{*}_{\alpha_1}(\xi), \ldots, \vec{\mu}^{*}_{\alpha_{i}}(\xi)  \rangle = \langle \vec{\xi}, \vec{\nu}_1,\ldots, \vec{\nu}_{i}\rangle
\end{align*} 
Denote for simplicity $ m = m_{i+1} $ and--
$$ g_{i+1} = g_{i+1}\left( \vec{\xi}, \vec{\nu}_1, \ldots, \vec{\nu}_i \right) \ , \  t_{i+1}  = t_{i+1}\left( \vec{\xi}, \vec{\nu}_1, \ldots, \vec{\nu}_i \right) \  , \  F^{j}_{i+1} = F^{j}_{i+1}\left( \vec{\xi}, \vec{\nu}_1, \ldots, \vec{\nu}_{i} \right)  \ \left( 0\leq j< m \right)$$
We apply lemma \ref{Lemma: FS, Multivar Fusion for xi}. For that, consider the forcing $ P\restriction_{ \left( \nu_i , g_{i+1} \right) } $ and the sequence $ F^{0}_{i+1}\vartriangleleft  F^{1}_{i+1} \vartriangleleft \ldots \vartriangleleft F^{m-1}_{i+1} $ of measures on $ g_{i+1} $. We describe a set $ d(\nu_{i+1}, \nu^{1}_{i+1}, \ldots, \nu^{m-1}_{i+1} ) \subseteq P\restriction_{ \left( \nu_i , g_{i+1} \right) } \setminus \nu_{i+1} $ which is $ \leq^* $ dense open above conditions which force that $ d^{-1}\{ \nu_{i+1}\} = \langle \nu^{1}_{i+1} , \ldots, \nu^{m-1}_{i+1} \rangle $: 
\begin{align*}
d(\nu_{i+1}, \nu^{1}_{i+1}, \ldots, \nu^{m-1}_{i+1} ) = \{& s\in P\restriction_{ \left( \nu_{i+1},g_{i+1}   \right) }  \colon \mbox{ if } s\Vdash \langle  {t_{i+1}}^{\frown} \langle \nu_{i+1} \rangle, \lusim{A}^{r}_{g_{i+1}} \rangle \geq r(g_{i+1}), \mbox{ then there}\\
&\mbox{exists a direct extension } q\geq^* \langle  {t_{i+1}}^{\frown} \langle \nu_{i+1} \rangle, \lusim{A}^{r}_{g_{i+1}} \rangle^{\frown} r\setminus \left( g_{i+1}+1 \right) \\
&\mbox{such that } s^{\frown} q\in e\left( \vec{\xi}, \ldots, \vec{\nu}_i, \vec{\nu}_{i+1} \right) \}
\end{align*}
By lemma \ref{Lemma: FS, Multivar Fusion for xi}, there exists $ r^*\restriction_{g_{i+1}} \geq^* r\restriction_{g_{i+1}} $ and a set $ X $ which belongs to the measure $ \mathcal{E}_{i+1} $ associated with $ F^{0}_{i+1}\vartriangleleft  F^{1}_{i+1} \vartriangleleft \ldots \vartriangleleft F^{m-1}_{i+1} $, such that for every increasing $ \langle \nu_{i+1}, \nu^{1}_{i+1}, \ldots, \nu^{m-1}_{i+1} \rangle \in X $,
$$ { r^* }\restriction_{ \nu_{i+1} } \Vdash {r^*}^{\frown} \langle \nu_{i+1}, \nu^{1}_{i+1}, \ldots, \nu^{m-1}_{i+1} \rangle\setminus \nu_{i+1} \in d\left( \nu_{i+1}, \nu^{1}_{i+1}, \ldots, \nu^{m-1}_{i+1} \right)  $$
Let us define $ r^*\setminus g_{i+1} $. Assume that we work in the generic extension for $ P_{g_{i+1}} $, and $ r^*\restriction_{g_{i+1}} $, which is already defined, belongs to it. For every $ \nu_{i+1} \in {A}^{r}_{ g_{i+1} } $ above $ \max t_{i+1} $, we denote $ d^{-1}\left( \nu_{i+1} \right) = \langle {\nu}^{1}_{i+1}, \ldots, {\nu}^{m-1}_{i+1} \rangle $. Let $ q\left( \nu_{i+1} \right) \geq^* \langle {t_{i+1}}^{\frown} \langle \nu_{i+1} \rangle, A^{r}_{g_{i+1}} \rangle^{\frown} r\setminus g_{i+1} $ be a condition such that--
$$ \left({r^*}^{\frown} \langle \nu_{i+1}, \ldots, \nu^{m-1}_{i+1} \rangle \setminus \nu_{i+1}\right)^{\frown} q\left( \nu_{i+1} \right)  \in e\left( \vec{\xi}, \ldots, \vec{\nu}_{i+1} \right)$$
(and take $ q\left( \nu_{i+1} \right) = r\setminus g_{i+1} $ if such $ q $ does not exist). 

We can now define $ r^*\left( g_{i+1} \right) $. Generate from the set $ X\in \mathcal{E}_{i+1} $ a corresponding set $ Y\in W_{g_{i+1}} $. Just pick $ Y $ to be a set such that every increasing sequence from $ Y\times \pi^{-1}_{1,0} Y \times \ldots \times \pi^{-1}_{m-1,0} Y $  belongs to $ X $. Let--
$$  r^*\left( g_{i+1} \right)  = \langle \langle \rangle, A^{r}_{g_{i+1}} \cap \left( Y\cup  \max t_{i+1} \right)\cap \left(\left( \triangle_{\nu_{i+1}<g_{i+1}} A^{q\left( \nu_{i+1} \right)}_{g_{i+1}} \right) \cup \max t_{i+1}\right) \rangle$$

Finally, let us define $ r^*\setminus \left(g_{i+1}+1\right) $ to be $ q\left(  \lusim{\nu}_{i+1} \right)\setminus \left(g_{i+1}+1\right) $, where here, $ \lusim{\nu}_{i+1} = d\left( g_{i+1} \right) $ can be read from the generic up to $ g_{i+1}+1 $. This concludes the definition of $ r^*\in e\left( \vec{\xi}, \ldots, \vec{\nu}_{i} \right) $.
\end{proof}

Inductively, it follows that for every increasing sequence $ \vec{\xi} = \langle \xi, \xi^{1}, \ldots, \xi^{m} \rangle $, the set $ e(\vec{\xi}) $, defined similarly as above, is $ \leq^* $-dense open above conditions which force that $d^{-1}\{ \xi \} =  \langle  \xi^{1}, \ldots, \xi^{m}\rangle $. Apply lemma \ref{Lemma: FS, Multivar Fusion for xi} one more time to obtain, from the condition $ p $ given in the formulation of the theorem, the required direct extension $ p^* $.
\end{proof}

\subsection{Proof of Properties  (B)-(E) }

\begin{lemma}
$\mu_{\alpha} = \mbox{crit}\left( k_{\alpha} \right)$  is measurable in $ M_{\alpha} $.
\end{lemma}
 
\begin{proof}
	Write $ \mu_{\alpha} = j_{\alpha}\left(h\right)\left( j_{1,\alpha}\left( \left[Id\right]^{'}_{0} \right) , j_{\alpha_1+1, \alpha} \left(\left[Id\right]_{\alpha_1}  \right), \ldots, j_{\alpha_k+1 , \alpha} \left(\left[Id\right]_{\alpha_k}\right) \right)  $, and let $ f\in V\left[G\right] $ be a function such that $ \mu_{\alpha} = \left[f\right]_W $. Let us assume, for contradiction, that for every $ \langle \vec{\xi}, \vec{\nu}_1, \ldots, \vec{\nu}_k \rangle $, $ h\left( \vec{\xi}, \vec{\nu}_1, \ldots, \vec{\nu}_k \right) $ is non-measurable. We can also assume that for every $ \xi<\kappa $, 
	$$  f(\xi) < h\left( \vec{\xi}, \vec{\mu}_{\alpha_1}(\xi), \ldots, \vec{\mu}_{\alpha_k}(\xi) \right) $$
	and let $ p\in G $ be a condition which forces this. From now, work above $ p $. We define a subset, 
	$$e\left( \vec{\xi}, \vec{\nu}_1, \ldots, \vec{\nu}_k \right) \subseteq  P\setminus \nu_{k}+1 $$
	for every $ \langle \vec{\xi}, \vec{\nu}_1, \ldots, \vec{\nu}_k \rangle $, which is $ \leq^* $ dense open above conditions which force that $ \langle \vec{\mu}_{\alpha_1}(\xi), \ldots, \vec{\mu}_{\alpha_k}(\xi) \rangle = \langle \vec{\nu}_1, \ldots, \vec{\nu}_k \rangle $:
	\begin{align*}
		e\left(\vec{\xi}, \vec{\nu}_1, \ldots, \vec{\nu}_k\right) = &\{ r\in P\setminus \nu_k+1 \colon \exists \alpha< h\left( \vec{\xi}, \vec{\nu}_1, \ldots, \vec{\nu}_k \right) , \ r\Vdash \lusim{f}(\xi) < \alpha \}
	\end{align*}
	The  $ \leq^* $-density of $e\left(\vec{\xi}, \vec{\nu}_1, \ldots, \vec{\nu}_k\right) $ essentially uses the fact that $h = h\left( \vec{\xi}, \vec{\nu}_{1}, \ldots , \nu_{k} \right) $ is not measurable; employing this, $ \lusim{f}(\xi) $ can be reduced to a $ P_{h} $-name, and thus be evaluated by less then $ h $ many possibilities by lemma \ref{Lemma: every name for an ordinael can be decided up to boudedly many values by a direct extension}.
	
	Let us apply now the Multivariable Fusion Lemma. There exists $ p^*\in G $ and sets $ X_0\in U^{0}, \ldots, X_m \in U^{m} $ such that for every sequence of sequences  $ \langle \vec{\xi}, \vec{\nu}_1, \ldots, \vec{\nu}_k \rangle $ which is admissible for $ p^* $,
	$$  \left({p^*}^{\frown} \langle \vec{\xi}, \vec{\nu}_1,\ldots ,\vec{\nu}_k \rangle \right)\restriction_{ \nu_k+1 } \Vdash \\
	\left({p^*}^{\frown} \langle \vec{\xi}, \vec{\nu}_1,\ldots , \vec{\nu}_k \rangle \right)\setminus { \left(\nu_k+1\right) } \in e\left( \xi, \nu_1 ,\ldots, \nu_k \right) $$
	whenever $ \vec{\xi}\in X_0\times \ldots X_{m} $ is increasing.
	
	For every such $ \langle \vec{\xi}, \vec{\nu}_1, \ldots, \vec{\nu}_k \rangle $, let--
	\begin{align*}
	A\left( \vec{\xi}, \vec{\nu}_1, \ldots, \vec{\nu}_k \right) = &\{  \gamma < h\left( \vec{\xi}, \vec{\nu}_1, \ldots, \vec{\nu}_k \right) \colon \exists q\geq \left({p^*}^{\frown} \langle \vec{\xi}, \vec{\nu}_1,\ldots ,\vec{\nu}_k \rangle \right)\restriction_{ \nu_k+1 } , \\
	&q\Vdash \gamma = \lusim{\alpha}\left( \vec{\xi}, \vec{\nu}_1,\ldots, \vec{\nu}_k \right) \}
	\end{align*}
	then $ A\left( \vec{\xi}, \vec{\nu}_1, \ldots, \vec{\nu}_k \right)  $ is a bounded subset of $ h\left( \vec{\xi}, \vec{\nu}_{1}, \ldots, \vec{\nu}_{k} \right) $. 
	
	For a set of $ \xi $-s in $ W $,
	$$ p^*\Vdash \lusim{f}(\xi)\in A\left( \vec{\xi}, \vec{\mu}^{*}_{\alpha_1}(\xi), \ldots, \vec{\mu}^{*}_{\alpha_k}(\xi) \right) $$
	In particular, in $ M\left[H\right] $,
	\begin{align*}
	 \left[f\right]_{W} \in &j_W\left(  \langle \vec{\xi}, \vec{\nu}_0, \ldots, \vec{\nu}_k \rangle\mapsto A\left( \vec{\xi}, \vec{\nu}_0, \ldots, \vec{\nu}_k \right) \right)\left(  \vec{\kappa}^*, \vec{\mu}^*_{\alpha_1}, \ldots, \vec{\mu}^*_{ \alpha_k } \right) = \\
	 & k_{\alpha}\left(  j_{\alpha}\left(  \langle \vec{\xi}, \vec{\nu}_1, \ldots, \vec{\nu}_k \rangle\mapsto A\left( \vec{\xi}, \vec{\nu}_1, \ldots, \vec{\nu}_k \right) \right)\left( j_{1,\alpha}\left( \left[Id\right]^{'}_{0} \right) , j_{\alpha_1+1, \alpha} \left(\left[Id\right]_{\alpha_1}  \right), \ldots, j_{\alpha_k+1 , \alpha} \left(\left[Id\right]_{\alpha_k}\right) \right)  \right)
	\end{align*}
	But--
\begin{align*}
	&\left|   j_{\alpha}\left(  \langle \vec{\xi}, \vec{\nu}_1, \ldots, \vec{\nu}_k \rangle\mapsto A\left( \vec{\xi}, \vec{\nu}_1, \ldots, \vec{\nu}_k \right) \right)\left( j_{1,\alpha}\left( \left[Id\right]^{'}_{0} \right) , j_{\alpha_1+1, \alpha} \left(\left[Id\right]_{\alpha_1}  \right), \ldots, j_{\alpha_k+1 , \alpha} \left(\left[Id\right]_{\alpha_k}\right) \right) \right| < \\
	&j_{\alpha}\left(h\right)\left(  j_{1,\alpha}\left( \left[Id\right]^{'}_{0} \right) , j_{\alpha_1+1, \alpha} \left(\left[Id\right]_{\alpha_1}  \right), \ldots, j_{\alpha_k+1 , \alpha} \left(\left[Id\right]_{\alpha_k}\right)\right) = \mu_{\alpha}
\end{align*}
and thus $ \mu_{\alpha} = \left[f\right]_W \in \mbox{Im}\left( k_{\alpha} \right) $, a contradiction.
\end{proof}

\begin{corollary}
$ \mu_{\alpha} $ is the least measurable $ \mu $ in $ M_{\alpha} $ such that $ \mu \geq \sup\{ \mu_{\alpha'} \colon \alpha'<\alpha \} $ and $ \left(\mbox{cf}\left(  \mu \right)\right)^V > \kappa $.
\end{corollary}

\begin{proof}
Assume for contradiction that there exists a measurable $ \lambda $ in $ M_{\alpha} $, such that $ \sup\{ \mu_{\alpha'} \colon \alpha'<\alpha \} \leq \lambda < \mu_{\alpha} $. Let us argue that $  \left(\mbox{cf}\left( \lambda \right)\right)^{V} \leq\kappa $.

First, $ \lambda = k_{\alpha}\left( \lambda \right) $, since $ \mbox{crit}\left( k_{\alpha} \right) = \mu_{\alpha} $, and thus $ \lambda $ is measurable in $ M $. Note that $ \lambda < \mu_{\alpha} < \kappa^* = j_{\alpha}\left( \kappa \right) $, so $ \lambda < j_W(\kappa) $  and thus $ \lambda  $ has cofinality $ \omega $ in $ M\left[H\right] $. Thus, $ \left(\mbox{cf}\left( \lambda \right)\right)^{V\left[G\right]} = \omega $, and, in particular, $  \left(\mbox{cf}\left( \lambda \right)\right)^{V} <\kappa $. 
\end{proof}

\begin{lemma} \label{Lemma:  mu alpha appears in the sequence of its k alpha }
$ \mu_{\alpha} $ appears as an element in the Prikry sequence of $ k_{\alpha}\left( \mu_{\alpha} \right) $. 
\end{lemma}

\begin{proof}
In $ M\left[H\right] $, denote by $ t^* $ the initial segment of the Prikry sequence of $ k_{\alpha}\left( \mu_{\alpha} \right) $ which consists of all the ordinals below $ \mu_{\alpha} $. Denote by $ n^* $ the length of $ t^* $. Let $ \langle \vec{\xi}, \vec{\nu}_1, \ldots, \vec{\nu}_{k} \rangle\mapsto t^*\left( \vec{\xi}, \vec{\nu}_1, \ldots, \vec{\nu}_{k} \right) $ be a function in $ V $ such that--
$$t^* = j_{\alpha}\left(   \langle \vec{\xi}, \vec{\nu}_1, \ldots, \vec{\nu}_{k} \rangle\mapsto t^*\left( \vec{\xi}, \vec{\nu}_1, \ldots, \vec{\nu}_k \right) \right)\left( j_{1,\alpha}\left(  \left[Id\right]^{'}_{\alpha} \right)   ,  j_{\alpha_1+1, \alpha}\left(  \left[Id\right]_{\alpha_1} \right) , \ldots, j_{\alpha_k+1, \alpha}\left(  \left[Id\right]_{\alpha_k}  \right) \right)  $$
(this can be done by modifying the nice sequence $ \langle \vec{\alpha}_1, \ldots, \vec{\alpha}_k \rangle $, if necessary, so that $ t^* $ can be represented by it). We can assume that for every $ \langle \vec{\xi}, \vec{\nu}_1, \ldots, \vec{\nu}_k \rangle $, $ t^*\left( \vec{\xi}, \vec{\nu}_1, \ldots, \vec{\nu}_k \right) $ is a sequence of length $ n^* $.  Since $ k_{\alpha}\left( t^* \right) = t^* $, 
$$ \left[ \xi \mapsto t^*\left( \vec{\xi}, \vec{\mu}^{*}_{\alpha_1}(\xi), \ldots, \vec{\mu}^{*}_{\alpha_k}(\xi) \right) \right]_W = t^* $$
In $ V\left[G\right] $, denote, for every $ \xi<\kappa $,
$$ \mu_{\alpha}(\xi) = \mbox{ the } \left(n^*+1\right)\mbox{-th element in the Prikry sequence of } h\left( \vec{\xi}, \vec{\mu}^{*}_{\alpha_1}(\xi), \ldots, \vec{\mu}^{*}_{\alpha_k}(\xi) \right) $$
Clearly $ \left[\xi\mapsto \mu_{\alpha}(\xi)\right]_W \geq \mu_{\alpha} $. We argue that equality holds. We will prove that for every $ \eta< \left[  \xi \mapsto \mu_{\alpha}(\xi)\right]_W $, $ \eta < \mu_{\alpha} $. Assume that such $ \eta $ is given, and let $ f\in V\left[G\right] $ be a function such that $ \left[f\right]_W = \eta $. Then we can assume that for every $ \xi<\kappa $,
$$ f(\xi) < \mu_{\alpha}(\xi) $$
and let $ p\in G $ be a condition which forces this. 

Let us apply now the Multivariable Fusion Lemma. For every $ \langle \vec{\xi}, \vec{\nu}_1, \ldots, \vec{\nu}_{k} \rangle $, consider the set--
\begin{align*}
e\left( \vec{\xi}, \vec{\nu}_1, \ldots, \vec{\nu}_{k} \right) = \{ &r\in P\setminus \nu_k \colon \exists \alpha<h\left( \vec{\xi}, \vec{\nu}_1,\ldots, \vec{\nu}_{k} \right), \ r\Vdash \mbox{ if } t^*\left( \vec{\xi}, \vec{\nu}_1,\ldots, \vec{\nu}_k \right)\\ 
&\mbox{is an initial segment of the Prikry sequence of } h\left( \vec{\xi}, \vec{\nu}_1, \ldots, \vec{\nu}_k \right),\\
&\mbox{then } \lusim{f}(\xi) < \alpha \} 
\end{align*}
then $ e\left( \vec{\xi}, \vec{\nu}_1, \ldots, \vec{\nu}_k \right) $ is $ \leq^* $ dense open above conditions which force that $ d^{-1}\{ \xi \} = \langle \xi^{1}, \ldots, \xi^{m} \rangle $ and $  \langle \vec{\nu}_0, \ldots, \vec{\nu}_k \rangle = \langle \vec{\mu}_{\alpha_0}(\xi),\ldots, \vec{\mu}_{\alpha_k}(\xi) \rangle  $. This follows in several steps: First, use the $ \leq^* $-closure to reduce $ \lusim{f}(\xi) $ to a $ P_{h+1} $-name, where $ h = h\left( \vec{\xi}, \vec{\nu}_1, \ldots, \vec{\nu}_k \right) $. This can be done by taking a direct extension of a given condition in the forcing $ P\setminus \left(h+1\right) $. Second, reduce $ \lusim{f}(\xi) $ to a $ P_{h} $ name, by applying on the Prikry forcing at coordinate $ h $ the following fact: If $ t^*  =  t^*\left( \vec{\xi}, \vec{\nu}_1, \ldots, \vec{\nu}_k \right) $ is an initial segment of the Prikry sequence of $ h $, and a given an ordinal is forced to be below the successor of $ t^*$ in this sequence, then its value can be decided by taking a direct extension. Finally, apply lemma \ref{Lemma: FS, evaluating new functions by old ones} and direct extend in the forcing $ P_{\left(\nu_k,h\right)} $, to bound the value of $ \lusim{f}(\xi) $ by an ordinal below $ h $. 

 Thus, there exists $ p^*\in G $, such that for every $ \langle \vec{\xi}, \vec{\nu}_1, \ldots, \vec{\nu}_k \rangle $ which is admissible for $ p^* $,
$$ {p^*}^{\frown} \langle \vec{\xi}, \vec{\nu}_1, \ldots, \vec{\nu}_k \rangle\restriction_{ \nu_k } \Vdash {p^*}^{\frown} \langle \vec{\xi}, \vec{\nu}_1, \ldots, \vec{\nu}_k \rangle\setminus \nu_k \in e\left( \vec{\xi}, \vec{\nu}_1, \ldots, \vec{\nu}_k \right) $$
in particular, $ {p^*}^{\frown} \langle \vec{\xi}, \vec{\nu}_1, \ldots, \vec{\nu}_k \rangle\restriction_{\nu_k} $ forces that there exists $ \alpha< h\left( \vec{\xi}, \vec{\nu}_1, \ldots, \vec{\nu}_k \right) $ such that--
$$ {p^*}^{\frown} \langle \vec{\xi}, \vec{\nu}_{1}, \ldots, \vec{\nu}_{k} \rangle\setminus  \nu_k \Vdash \lusim{f}(\xi)<\alpha $$ 
Let--
$$ A\left( \vec{\xi}, \vec{\nu}_1, \ldots, \vec{\nu}_{k} \right) = \{  \gamma \colon \exists q\geq {p^*}^{\frown}\langle \vec{\xi}, \vec{\nu}_1, \ldots, \vec{\nu}_k \rangle\restriction_{ \nu_{k} }, \ q\Vdash \lusim{\alpha} = \gamma \}$$
then $ p^*\Vdash \lusim{f}(\xi)\in A\left( \vec{\xi}, \vec{\mu}^{*}_{\alpha_1}(\xi), \ldots, \vec{\mu}^{*}_{\alpha_k}(\xi) \right) $. Arguing as in the previous lemma, it follows that $ \eta = \left[f\right]_W \in \mbox{Im}\left( k_{\alpha} \right)\cap k_{\alpha}\left( \mu_{\alpha} \right) = \mu_{\alpha} $, as desired.
\end{proof}

\begin{lemma}
Assume that for every measurable $ \xi<\kappa $, $\lusim{U}_{\xi}$ is a $ P_{\xi} $-name for a measure on $ \xi $ which belongs to $ V $. Let $ \lusim{U} = j_{\alpha}\left( \xi \mapsto \lusim{U}_{ \xi } \right)  \left( \mu_{\alpha} \right) $. Then there exists a set $ \mathcal{F}\in M_{\alpha} $ of measures on $ \mu_{\alpha} $ in $ M_{\alpha} $, with $ \left| \mathcal{F} \right|<\mu_{\alpha} $, such that, for some $ p\in G $,
	$$ \left(j_W(p)\right)^{\frown} \langle \vec{\mu}^{*}_{\alpha_0}, \ldots, \vec{\mu}^{*}_{\alpha_k} \rangle \Vdash k_{\alpha}\left(\lusim{U} \right)\in k_{\alpha}'' \mathcal{F} $$
In particular, there exists a measure $F\in \mathcal{F}$ such that $ k_{\alpha}\left( F \right) = \left( k_{\alpha}\left( \lusim{U} \right) \right)_H $.
\end{lemma}

\begin{proof}
For every $ \xi<\kappa $, fix an enumeration $ s_{\xi} $ of all the normal measures on $ \xi $ in $ V $. Apply Multivariable Fusion. Define for every $ \langle \vec{\xi}, \vec{\nu}_1, \ldots, \vec{\nu}_{k} \rangle $ the set--
	\begin{align*}
		e\left( \vec{\xi}, \vec{\nu}_0, \ldots, \vec{\nu}_{k} \right) = &\{ r\in P\setminus \nu_k \colon \mbox{there exists a set of ordinals } A \mbox{ of cardinality}\\
		&\mbox{strictly smaller than } h\left( \vec{\xi}, \vec{\nu}_1, \ldots, \vec{\nu}_{k} \right), \mbox{ such that }\\
		& r\restriction_{  h\left( \vec{\xi}, \vec{\nu}_1, \ldots, \vec{\nu}_{k} \right) } \Vdash \lusim{U}_{ h\left( \vec{\xi}, \vec{\nu}_1, \ldots, \vec{\nu}_{k} \right) } \in s_{ h\left( \vec{\xi}, \vec{\nu}_1, \ldots, \vec{\nu}_{k} \right) } '' A \}
	\end{align*}
	As before, there exists $ p^* \geq^* p  $ in $ G $, such that for every $ \langle \vec{\xi}, \vec{\nu}_1, \ldots, \vec{\nu}_{k} \rangle $ which is admissible for $ p^* $,
	\begin{align*}
		{p^*}^{\frown} \langle \vec{\xi}, \vec{\nu}_1, \ldots, \vec{\nu}_{k} \rangle\restriction_{\nu_{k}} \Vdash {p^*}^{\frown} \langle \vec{\xi}, \vec{\nu}_1, \ldots, \vec{\nu}_{k} \rangle\setminus{\nu_{k}}\in e\left( \vec{\xi}, \vec{\nu}_1, \ldots, \vec{\nu}_{k} \right)
	\end{align*}
	Let $ A\left( \vec{\xi}, \vec{\nu}_1, \ldots, \vec{\nu}_{k} \right) $ be the set of ordinals, forced by some extension of $ {p^*}^{\frown} \langle \vec{\xi}, \vec{\nu}_1, \ldots, \vec{\nu}_{k} \rangle\restriction_{\nu_{k}} $, to be an element of the set $ \lusim{A} $ above. Then $   A\left( \vec{\xi}, \vec{\nu}_1, \ldots, \vec{\nu}_{k} \right)  $ is a set of ordinals of cardinality strictly below $ h\left( \vec{\xi}, \vec{\nu}_1, \ldots, \vec{\nu}_{k} \right) $. 
	
	In $ V\left[G\right] $,
	$$ \{ \xi<\kappa \colon {p^*}^{\frown} \langle \vec{\xi}, \vec{\mu}^{*}_{\alpha_1}(\xi), \ldots, \vec{\mu}^{*}_{\alpha_k }(\xi) \rangle \Vdash   U_{ h\left( \vec{\xi}, \vec{\nu}_1, \ldots, \vec{\nu}_{k} \right) } \in  s_{ h\left( \vec{\xi}, \vec{\nu}_1, \ldots, \vec{\nu}_{k} \right) }'' A^*\left( \vec{\xi}, \vec{\mu}^{*}_{\alpha_1}(\xi), \ldots, \mu^{*}_{\alpha_k}(\xi) \right) \}\in W $$
	Now, in $ M_{\alpha} $, denote--
	$$ \mathcal{F}  =  \left(j_{ \alpha }(s)_{ \mu_{\alpha} } \right)''  j_{\alpha} \left(  \langle \vec{\xi}, \vec{\nu}_1, \ldots, \vec{\nu}_{k} \rangle\mapsto A^*\left( \vec{\xi}, \vec{\nu}_1, \ldots, \vec{\nu}_{k} \right) \right)\left( j_{1,\alpha}\left( \left[Id\right]^{'}_{0} \right), j_{\alpha_1+1, \alpha}\left(  \left[Id\right]_{1} \right) , \ldots,  j_{\alpha_k+1, \alpha }\left( \left[Id\right]_{k} \right) \right) $$
	Then $ \left| \mathcal{F} \right| < \mu_{\alpha} $, and, in $ M\left[H\right] $, 
	$$ \left(j_W(p^*)\right)^{\frown} \langle \vec{\mu}^{*}_{\alpha_0}, \ldots, \vec{\mu}^{*}_{\alpha_k} \rangle \Vdash k_{\alpha}\left( \lusim{\mathcal{U}} \right) \in k_{\alpha}'' \mathcal{F} $$

\end{proof}
We now apply the above lemma on specific names for measures $ \lusim{U}_{\xi} $. For every measurable $ \xi<\kappa $, let $ W_\xi = U^{\times}_{\xi} $ be the measure used to singularize $ \xi $ at stage $ \xi $ in the iteration $ P = P_{\kappa} $. Each such $ W^{\times}_{\xi} $ can be assigned to a sequence of Rudin-Keisler equivalent measures on $ \xi $, 
$$ W^{0}_{\xi}, \ldots, W^{m_{\xi}}_{\xi} = W_{\xi} $$
as defined in section $ 2 $. Denote $ U^{j}_{\xi} = W^{j}_{\xi} \cap V \in V $ for every $ 0\leq j \leq m_{\xi} $. Then--
$$U^{0}_{\xi}  \vartriangleleft  U^{1}_{\xi} \vartriangleleft \ldots  \vartriangleleft U^{m_{\xi}}_{\xi} $$

\begin{corollary} \label{Corollary: FS, how U^j_mualpha depends on U^j_xi}
For every $ 1\leq j <  m_{\alpha} $, there exists a measure $ U^{j}_{\mu_{\alpha} }\in M_{\alpha} $ on $ \mu_{\alpha} $, such that--
$$ k_{\alpha}\left(  U^{j}_{\mu_{\alpha} } \right) = j_{W} \left(  \xi \mapsto U^{j}_{\xi}  \right)\left( k_{\alpha}\left( \mu_{\alpha} \right)  \right)  = \left[ \xi\mapsto U^{j}_{ h\left(  \vec{\mu}^{*}_{\alpha_0}(\xi), \ldots, \vec{\mu}^{*}_{\alpha_{k}}(\xi) \right) } \right]_W $$
\end{corollary}

We consider the above corollary as the definition of the measures $ U^{j}_{ \mu_{\alpha} } $ for every $ 1\leq j<  m_{\alpha} $. Note that, by elementarity, $ U^{0}_{ \mu_{\alpha} }\vartriangleleft U^{1}_{\mu_{\alpha} }\vartriangleleft\ldots \vartriangleleft  U^{m_{\alpha}-1 }_{ \mu_{\alpha} } $. Let $ \mathcal{E}_{\alpha} $ be the measure on $ \left[\kappa\right]^{m_{\alpha}} $ which corresponds to the iterated ultrapower with $ U^{m_{\alpha}-1}_{\mu_{\alpha}} \vartriangleright \ldots \vartriangleright U^{0}_{ \mu_{\alpha} } $ in decreasing order. We argue that $ \mathcal{E}_{\alpha} $ is derived from $ k_{\alpha} \colon M_{\alpha}\to M $. 

\begin{lemma}
$ U^{0}_{ \mu_{\alpha} } = \{ X\subseteq \mu_{\alpha} \colon \mu_{\alpha}\in k_{\alpha}(X) \}\cap M_{\alpha} $. Furthermore, if $ m_{\alpha}>1 $, then for every $ 1\leq j \leq m_{\alpha}-1 $, 
$$ U^{j}_{ \mu_{\alpha} } = \{ X\subseteq \mu_{\alpha} \colon \mu^{*j}_{\alpha}\in k_{\alpha}(X) \}\cap M_{\alpha}  $$
\end{lemma}

\begin{remark}
	$  $
\begin{enumerate}
\item Let us note that for every $ j $ as above, $ \mu_{\alpha}< \mu^{j}_{\alpha}< k_{\alpha}\left( \mu_{\alpha} \right) $, so $ U^{*j}_{\mu_{\alpha} } $ is an  ultrafilter which concentrates on $ \mu_{\alpha} $.
\item  We deliberately did not define, in corollary \ref{Corollary: FS, how U^j_mualpha depends on U^j_xi}, the measure $ U^{m_{\alpha}}_{ \mu_{\alpha} } $ - it is not derived from $ k_{\alpha} $ and does not participate in $ j_W\restriction_{V} $. The exception is $ \alpha =0 $ where $ U^{m} = U $ is the first step in the iteration.
\end{enumerate}
\end{remark}

\begin{proof}
We first provide the proof for $ U^{0}_{\mu_{\alpha} } $. Assume that $ X\in M_{\alpha} $ and $ \mu_{\alpha}\in k_{\alpha}(X) $. Write--
$$ X = j_{\alpha}\left( \langle \vec{\xi}, \vec{\nu}_0, \ldots, \vec{\nu}_{k} \rangle\mapsto X\left( \vec{\xi}, \vec{\nu}_0, \ldots, \vec{\nu}_{k} \right) \right)\left( j_{1,\alpha}\left(\vec{\kappa}\right), j_{\alpha_0, \alpha}\left(\vec{\mu}_{\alpha_0}\right), \ldots, j_{\alpha_k, \alpha}\left(\vec{\mu}_{\alpha_k}\right) \right) $$
where, without loss of generality, the nice sequence $ \langle {\alpha_0}, \ldots, {\alpha_k} \rangle $ can be used to represent $ \mu_{\alpha} $ in $ M_{\alpha} $, in the usual sense that for a function $ h\in V $,
$$ \mu_{\alpha} = j_{\alpha}(h)\left( j_{1,\alpha}\left(\vec{\kappa}\right), j_{\alpha_0, \alpha}\left(\vec{\mu}_{\alpha_0}\right), \ldots, j_{\alpha_k,\alpha}\left(\vec{\mu}_{\alpha_k}\right) \right) $$
Apply Multivariable Fusion. For every $\langle \vec{\xi}, \vec{\nu}_1, \ldots, \vec{\nu}_{k}\rangle $, let--
\begin{align*}
e\left( \vec{\xi}, \vec{\nu}_1, \ldots, \vec{\nu}_{k} \right) = \{ & r\in P\setminus \nu_k \colon r\restriction_{ h\left( \vec{\xi}, \vec{\nu}_1, \ldots, \vec{\nu}_{k} \right) }\parallel X\left( \vec{\xi}, \vec{\nu}_1, \ldots, \vec{\nu}_{k} \right)\in U^{\times }_{ h\left( \vec{\xi}, \vec{\nu}_1, \ldots, \vec{\nu}_{k} \right) } , \\
& \mbox{if it decides positively, then } r\restriction_{ h\left( \vec{\xi}, \vec{\nu}_1, \ldots, \vec{\nu}_{k} \right) } \Vdash \lusim{A}^{r}_{ h\left( \vec{\xi}, \vec{\nu}_1, \ldots, \vec{\nu}_{k} \right) } \subseteq \\
& X\left( \vec{\xi}, \vec{\nu}_1, \ldots, \vec{\nu}_{k} \right) ; \mbox{ else, } r\restriction_{ h\left( \vec{\xi}, \vec{\nu}_1, \ldots, \vec{\nu}_{k} \right) } \Vdash \lusim{A}^{r}_{ h\left( \vec{\xi}, \vec{\nu}_1, \ldots, \vec{\nu}_{k} \right) } \mbox{ is disjoint} \\
&\mbox{from } X\left( \vec{\xi}, \vec{\nu}_1, \ldots, \vec{\nu}_{k} \right). \mbox{ Moreover, } r\restriction_{ h\left( \vec{\xi}, \vec{\nu}_1, \ldots, \vec{\nu}_{k} \right) } \parallel \mbox{lh}\left( t^{r}_{ h\left( \vec{\xi}, \vec{\nu}_1, \ldots, \vec{\nu}_{k} \right) } \right) > n^*, \\
&\mbox{and if it decides positively, then there exists a bounded subset } \\
& A\left( \vec{\xi}, \vec{\nu}_1, \ldots, \vec{\nu}_{k} \right)\subseteq h\left( \vec{\xi}, \vec{\nu}_1, \ldots, \vec{\nu}_{k} \right) \mbox{ for which } r\restriction_{ h\left( \vec{\xi}, \vec{\nu}_1, \ldots ,\vec{\nu}_{k} \right) } \Vdash \mbox{ the } n^*\mbox{-th}\\
&\mbox{element of }t^{r}_{ h\left( \vec{\xi}, \vec{\nu}_1, \ldots, \vec{\nu}_{k} \right) } \mbox{ belongs to } A\left( \vec{\xi}, \vec{\nu}_1, \ldots, \vec{\nu}_{k} \right)  \}
\end{align*}
Applying the same tools above, there exists a condition $ p^*\in G $ and a bounded subset $ A^*\left( \vec{\xi}, \vec{\nu}_1, \ldots, \vec{\nu}_{k} \right) $, such that--
\begin{align*}
\{ \xi<\kappa \colon &p^*\restriction_{ h\left( \vec{\mu}^*_{ 0 }(\xi), \vec{\mu}^{*}_{\alpha_1}(\xi), \ldots, \vec{\mu}^{*}_{\alpha_k}(\xi) \right) }  \   \Vdash \mbox{if lh}\left( t^{r}_{ h\left( \vec{\mu}^*_{ 0 }(\xi), \vec{\mu}^{*}_{\alpha_1}(\xi), \ldots, \vec{\mu}^{*}_{\alpha_k}(\xi) \right) } \right)\geq n^* \mbox{ then the }\\
&n^*\mbox{-th element in the Prikry sequence of }\\
&h\left( \vec{\mu}^*_{ 0 }(\xi), \vec{\mu}^{*}_{\alpha_1}(\xi), \ldots, \vec{\mu}^{*}_{\alpha_k}(\xi) \right) \mbox{ belongs to } A^*\left( \vec{\mu}^*_{ 0 }(\xi), \vec{\mu}^{*}_{\alpha_1}(\xi), \ldots, \vec{\mu}^{*}_{\alpha_k}(\xi) \right)  \}\in W 
\end{align*}
Since $ A^*\left( \vec{\mu}^*_{ 0 }(\xi), \vec{\mu}^{*}_{\alpha_1}(\xi), \ldots, \vec{\mu}^{*}_{\alpha_k}(\xi) \right) $ is bounded in $ h\left( \vec{\mu}^*_{ 0 }(\xi), \vec{\mu}^{*}_{\alpha_1}(\xi), \ldots, \vec{\mu}^{*}_{\alpha_k}(\xi) \right) $, it follows that--
$$ \{ \xi <\kappa  \colon p^*\restriction_{ h\left( \vec{\mu}^*_{ 0 }(\xi), \vec{\mu}^{*}_{\alpha_1}(\xi), \ldots, \vec{\mu}^{*}_{\alpha_k}(\xi) \right) }\Vdash \mbox{lh}\left(  t^{r}_{ h\left( \vec{\mu}^*_{ 0 }(\xi), \vec{\mu}^{*}_{\alpha_1}(\xi), \ldots, \vec{\mu}^{*}_{\alpha_k}(\xi) \right) } \right) <n^*  \}\in W $$
By the choice of $ p^* $, it follows that for a set of $ \xi $-s in $ W $, 
$$ p^*\restriction_{  h\left( \vec{\mu}^*_{ 0 }(\xi), \vec{\mu}^{*}_{\alpha_1}(\xi), \ldots, \vec{\mu}^{*}_{\alpha_k}(\xi) \right) } \parallel X\left( \xi, \mu_{\alpha_0}(\xi), \ldots, \mu_{\alpha_k}(\xi)  \right)\in U^{\times}\left( \vec{\mu}^*_{ 0 }(\xi), \vec{\mu}^{*}_{\alpha_1}(\xi), \ldots, \vec{\mu}^{*}_{\alpha_k}(\xi) \right) $$
we argue that for a set of $ \xi $-s in $ W $, the decision is positive. Indeed, otherwise, it holds in $ M\left[H\right] $ that--
$$ \mu_{\alpha} = \left[  \xi \mapsto \mu_{\alpha}(\xi) \right]_W \in \left[ \xi \mapsto h\left( \vec{\mu}^*_{ 0 }(\xi), \vec{\mu}^{*}_{\alpha_1}(\xi), \ldots, \vec{\mu}^{*}_{\alpha_k}(\xi) \right)\setminus X\left( \vec{\mu}^*_{ 0 }(\xi), \vec{\mu}^{*}_{\alpha_1}(\xi), \ldots, \vec{\mu}^{*}_{\alpha_k}(\xi) \right)  \right]_W = k_{\alpha}(\mu_{\alpha}\setminus X) $$
contradicting the choice of $ X $. Thus, for a set of $ \xi $-s in $ W $, 
$$ p^*\restriction_{  h\left( \vec{\mu}^*_{ 0 }(\xi), \vec{\mu}^{*}_{\alpha_1}(\xi), \ldots, \vec{\mu}^{*}_{\alpha_k}(\xi) \right) } \Vdash  X\left( \vec{\mu}^*_{ 0 }(\xi), \vec{\mu}^{*}_{\alpha_1}(\xi), \ldots, \vec{\mu}^{*}_{\alpha_k}(\xi) \right)\in U^{\times}\left( \vec{\mu}^*_{ 0 }(\xi), \vec{\mu}^{*}_{\alpha_1}(\xi), \ldots, \vec{\mu}^{*}_{\alpha_k}(\xi) \right) $$
recall that $ U^{0} = U^{\times} \cap V $; hence--
$$ p^*\restriction_{  h\left( \vec{\mu}^*_{ 0 }(\xi), \vec{\mu}^{*}_{\alpha_1}(\xi), \ldots, \vec{\mu}^{*}_{\alpha_k}(\xi) \right) } \Vdash  X\left( \vec{\mu}^*_{ 0 }(\xi), \vec{\mu}^{*}_{\alpha_1}(\xi), \ldots, \vec{\mu}^{*}_{\alpha_k}(\xi) \right)\in U^{0}\left( \vec{\mu}^*_{ 0 }(\xi), \vec{\mu}^{*}_{\alpha_1}(\xi), \ldots, \vec{\mu}^{*}_{\alpha_k}(\xi) \right) $$
and thus, in $ M\left[H\right] $,  $ k_{\alpha}(X) \in k_{\alpha}\left( U^{0}_{\mu_{\alpha}} \right) $. In particular, in $ M_{\alpha} $, $ X\in U^{0}_{\mu_{\alpha}} $.

We now proceed to the proof for $ U^{j}_{\mu_{\alpha}} $ for every $ 1\leq j \leq m_{\alpha}-1 $. Assume that $ \mu^{*j}_{\alpha}\in k_{\alpha}(X) $, and recall that $ \mu^{*j}_{\alpha} $ is the $ j $-th element in $ d^{-1}\left( \mu_{\alpha} \right) $. We repeat the same argument above. First, define--
\begin{align*}
e\left( \vec{\xi}, \vec{\nu}_1, \ldots, \vec{\nu}_{k} \right) = \{ & r\in P\setminus \nu_k \colon r\restriction_{ h\left( \vec{\xi}, \vec{\nu}_1, \ldots, \vec{\nu}_{k} \right) }\parallel X\left( \vec{\xi}, \vec{\nu}_1, \ldots, \vec{\nu}_{k} \right)\in U^{j }_{ h\left( \vec{\xi}, \vec{\nu}_1, \ldots, \vec{\nu}_{k} \right) } , \\
	& \mbox{if it decides positively, then } r\restriction_{ h\left( \vec{\xi}, \vec{\nu}_1, \ldots, \vec{\nu}_{k} \right) } \Vdash \lusim{A}^{r}_{ h\left( \vec{\xi}, \vec{\nu}_1, \ldots, \vec{\nu}_{k} \right) } \subseteq \\
	& d'' X\left( \vec{\xi}, \vec{\nu}_1, \ldots, \vec{\nu}_{k} \right) ; \mbox{ else, } r\restriction_{ h\left( \vec{\xi}, \vec{\nu}_1, \ldots, \vec{\nu}_{k} \right) } \Vdash \pi^{-1}_{j,0} {''} \lusim{A}^{r}_{ h\left( \vec{\xi}, \vec{\nu}_1, \ldots, \vec{\nu}_{k} \right) } \mbox{ is disjoint} \\
	&\mbox{from }  X\left( \vec{\xi}, \vec{\nu}_1, \ldots, \vec{\nu}_{k} \right). \mbox{ Moreover, } r\restriction_{ h\left( \vec{\xi}, \vec{\nu}_1, \ldots, \vec{\nu}_{k} \right) } \parallel \mbox{lh}\left( t^{r}_{ h\left( \vec{\xi}, \vec{\nu}_1, \ldots, \vec{\nu}_{k} \right) } \right) > n^*, \\
	&\mbox{and if it decides positively, then there exists a bounded subset } \\
	& A\left( \vec{\xi}, \vec{\nu}_1, \ldots, \vec{\nu}_{k} \right)\subseteq h\left( \vec{\xi}, \vec{\nu}_1, \ldots, \vec{\nu}_{k} \right) \mbox{ for which } r\restriction_{ h\left( \vec{\xi}, \vec{\nu}_1, \ldots ,\vec{\nu}_{k} \right) } \Vdash \mbox{ the } n^*\mbox{-th}\\
	&\mbox{element of }t^{r}_{ h\left( \vec{\xi}, \vec{\nu}_1, \ldots, \vec{\nu}_{k} \right) } \mbox{ belongs to } A\left( \vec{\xi}, \vec{\nu}_1, \ldots, \vec{\nu}_{k} \right)  \}
\end{align*}
Now we argue as before, and claim that--
$$ p^*\restriction_{  h\left( \vec{\mu}^*_{ 0 }(\xi), \vec{\mu}^{*}_{\alpha_1}(\xi), \ldots, \vec{\mu}^{*}_{\alpha_k}(\xi) \right) } \Vdash  X\left( \vec{\mu}^*_{ 0 }(\xi), \vec{\mu}^{*}_{\alpha_1}(\xi), \ldots, \vec{\mu}^{*}_{\alpha_k}(\xi) \right)\in U^{j}\left( \vec{\mu}^*_{ 0 }(\xi), \vec{\mu}^{*}_{\alpha_1}(\xi), \ldots, \vec{\mu}^{*}_{\alpha_k}(\xi) \right) $$
indeed, otherwise, there exists a set of $ \xi $-s in $ W $ for which--
$$ \mu^{j}_{\alpha}(\xi) = \pi_{j,0}^{-1}\left( \mu_{\alpha}(\xi) \right) \notin X\left( \vec{\mu}^*_{ 0 }(\xi), \vec{\mu}^{*}_{\alpha_1}(\xi), \ldots, \vec{\mu}^{*}_{\alpha_k}(\xi) \right) $$
contradicting the fact that $ \mu^{j}_{\alpha}\in k_{\alpha}\left( X \right) $. It follows that, in $ M\left[H\right] $, $  k_{\alpha}(X)\in k_{\alpha}\left(   U^{j}_{\mu_{\alpha}}  \right) $, and so $ X\in U^{j}_{\mu_{\alpha} } $.
\end{proof}

\begin{corollary}
	$\mathcal{E}_{\alpha} = \{  X\subseteq \left[ \mu_{\alpha}  \right]^{m_{\alpha}}  \colon \langle \mu_{\alpha}, \mu^{*1}_{\alpha}, \ldots, \mu^{*m_{\alpha}-1}_{\alpha} \rangle \in k_{\alpha}\left(X\right) \}\cap M_{\alpha}$. 
\end{corollary}

\begin{proof}
It suffices to prove that--
$$\mathcal{E}_{\alpha} \subseteq \{  X\subseteq \left[ \mu_{\alpha}  \right]^{m_{\alpha}}  \colon \langle \mu_{\alpha}, \mu^{*1}_{\alpha}, \ldots, \mu^{*m_{\alpha}-1}_{\alpha} \rangle \in k_{\alpha}\left(X\right) \}\cap M_{\alpha}$$
Start from $ X\in \mathcal{E}_{\alpha} $. Then there are sets $ X_0\in U^{0}_{\mu_{\alpha}} , \ldots, X_{m_{\alpha}-1}\in U^{m_{\alpha}-1}_{\mu_{\alpha} }$ such that the set of increasing sequences in $ X_0 \times \ldots \times X_{m-1}$ is contained in $ X $. Thus every increasing sequence in $ k_{\alpha}\left(X_0\right)\times \ldots \times k_{\alpha}\left( X_{m-1} \right)$ belongs to $k_{\alpha}(X) $, and by the previous lemma, $ \langle \mu_{\alpha}, \ldots, \mu^{*m_{ \alpha }-1}_{\alpha} \rangle\in k_{\alpha}(X) $, as desired.
\end{proof}

This concludes the proof of properties $ (A)-(F) $ from the beginning of the section.

Recall that $ \kappa^* = j_{U}(\kappa) $. Note that $ \kappa^* = j_{\kappa^*}\left( \kappa \right) $, since $ \kappa $ is mapped to $ \kappa^* $ after the first step in the iteration, and every step after it is taken on a measurable $ \mu^{j}_{\alpha} $ below $ \kappa^* $. Moreover, $ \mbox{sup}\left( \mu_{\alpha} \colon \alpha<\kappa^* \right) = \kappa^* $ and thus $ \mbox{crit}\left( k_{\kappa^*} \right) \geq \kappa^* $. Let us use the above properties and argue that the induction halts after $ \kappa^* $-many steps. 

\begin{lemma}
$ k_{\kappa^*} \colon M_{\kappa^*}\to M $ is the identity function. In particular, $j_W\restriction_{ V } = j_{\kappa^*}$.
\end{lemma}

\begin{proof}
It suffices to prove that for every ordinal $ \eta $, $ \eta\in \mbox{Im}\left( k_{\kappa^*} \right) $. Fix such $ \eta $ and let $ g\in V\left[G\right] $ be a function such that $ \left[g\right]_{W^*} = \eta $. Let $ \lusim{g} $ be a $ P = P_{\kappa} $-name for it. For every $ \xi<\kappa $, let--
$$ e(\xi) = \{ r\in P\setminus \xi \colon \mbox{for some } A\subseteq \kappa \mbox{ with } \left|A\right|<\kappa, \ r\Vdash \lusim{g}(\xi)\in A \} $$
$ e(\xi) $ is $ \leq^* $-dense open by lemma \ref{Lemma: every name for an ordinael can be decided up to boudedly many values by a direct extension}. By Fusion, there exists $ p\in G $ such that for every $ \xi<\kappa $,
$$  p\restriction_{\xi}\Vdash  \mbox{for some } A\subseteq \kappa \mbox{ with } \left|A\right|<\kappa, \  \left(  p\setminus \xi \right)^{-\xi}  \Vdash \lusim{g}(\xi)\in A  $$
Let $ A(\xi) = \{ \gamma<\kappa \colon \exists q\geq p\restriction_{\xi}, \ q\Vdash \gamma \in \lusim{A}  \} $. Then for every $ \xi<\kappa $,  $ \left| A(\xi) \right| <\kappa $, and $ p^{-\xi} \Vdash \lusim{g}(\xi)\in A(\xi) $. Recall that for every $ p\in G $, $ \left(j_{W}(p)\right)^{-\left[Id\right]_{W^*}  }\in H $. Thus, in $ M\left[H\right] $,
$$ \left[g\right]_{W^*} \in j_{W^*}\left( \xi \mapsto A(\xi) \right)( \left[Id\right]_{W^*} ) = k_{\kappa^*}\left(  \   j_{\kappa^*} \left( \xi\mapsto A(\xi) \right) \left( j_{1,\kappa^*}\left( \kappa \right) \right)  \   \right)   $$
but $ \left| j_{\kappa^*} \left( \xi\mapsto A(\xi) \right) \left( j_{1,\kappa^*}\left( \kappa \right) \right) \right| < j_{\kappa^*}\left( \kappa \right)  = \kappa^*  \leq \mbox{crit}\left( k_{\kappa^*} \right)  $, and thus $ \left[g\right]_{W^*}\in \mbox{Im}\left( k_{\kappa^*} \right) $.
\end{proof}

This completes the proof of theorem \ref{Thm: Structure of j_Wrestriction V}.

\begin{lemma}
Fix $ \alpha<\kappa^* $ and denote $ m = m_{\alpha} $. Let $ 0< j \leq m $. Then $ \mu^{*j}_{\alpha} $ is measurable in $ M $, and its Prikry sequence in $ M\left[H\right] $ is the sequence of critical points obtained by iterating the measure $ U^{j-1}_{ \mu_{\alpha} } $ over $ M_{\alpha} $.
\end{lemma}

\begin{proof}
First, by lemma \ref{Lemma:  mu alpha appears in the sequence of its k alpha },  $ \mu_{\alpha} $ appears in the Prikry sequence of $ \mu^{*j}_{\alpha} $. Let $ \lambda $ be the element after it in this Prikry sequence, and let us argue that $ \lambda = j_{  U^{j-1}_{\mu_{\alpha} }  }\left( \mu_{\alpha} \right) = \mu^{j-1}_{\alpha}  $. Since $ j_{  U^{j-1}_{\mu_{\alpha} }  }\left( \mu_{\alpha} \right) $ is measurable in $ M_{\alpha+1} $ and has cofinality above $ \kappa $ in $ V $, there exists $ \beta>\alpha $ such that $\mu_{\beta} = \mu^{j}_{\alpha} $; Also, $ k_{\beta}\left( \mu_{\beta} \right) = k_{\beta}\left( j_{\alpha+1,\beta  } \left( \mu^{j}_{\alpha} \right) \right) = k_{\alpha+1}\left( \mu^{j}_{\alpha} \right) = \mu^{*j}_{\alpha} $, and thus $ \mu_{\beta} = \mu^{j}_{\alpha} $ appears as an element in the Prikry sequence of $ \mu^{*j}_{\alpha} $. Thus, $ \lambda \leq \mu^{j}_{\alpha} $, and it suffices to prove that $ \lambda = \mu^{j}_{\alpha} $. Assume for contradiction that $ \lambda < \mu^{j-1}_{\alpha} = j_{  U^{j-1}_{\mu_{\alpha} }  }\left( \mu_{\alpha} \right)$.

In $M_{\alpha}$ write-- $$\mu_{\alpha} = j_{\alpha}(h)\left( j_{1,\alpha}\left(  \left[Id\right]^{'}_{0} \right), j_{ \alpha_1+1, \alpha }\left( \left[Id\right]_{\alpha_1}  \right) , \ldots, j_{  \alpha_k +1, \alpha }\left( \left[Id\right]_{\alpha_k} \right)  \right)$$ 
and-- 
$$\lambda = j_{\alpha+1}(g)\left( j_{1,\alpha}\left(  \left[Id\right]^{'}_{0} \right), j_{ \alpha_1+1, \alpha }\left( \left[Id\right]_{\alpha_1}  \right) , \ldots, j_{  \alpha_k +1, \alpha }\left( \left[Id\right]_{\alpha_k} \right), \left[Id\right]_{\alpha}  \right)$$
 for some functions $ f,g $ in $ V $. Recall that $ \left[Id\right]_{\alpha} = \langle \mu_{\alpha}, \ldots, \mu^{j-1}_{\alpha}, \ldots, \mu^{m_\alpha-1}_{\alpha} \rangle $, so we can assume that for every $ \vec{\xi}, \vec{\nu}_1 , \ldots, \vec{\nu}_k, \vec{\nu} = \langle \nu_0, \ldots, \nu_{m-1} \rangle$,
$$  g\left( \vec{\xi}, \vec{\nu}_1, \ldots, \vec{\nu}_k, \langle \nu^0, \ldots, \nu^{j-1}, \ldots , \nu^{m-1}\rangle \right)  <\min \{ h\left( \vec{\xi}, \vec{\nu}_1, \ldots, \vec{\nu}_k\right),  \nu^{j-1}\} $$

Let $ t^* $ be the initial segment of the Prikry sequence of $ k_{\alpha}\left( \mu_{\alpha} \right) $ which consists of all the ordinals below $ \mu_{\alpha} $. Fix a function $ \langle \vec{\xi}, \vec{\nu}_1, \ldots, \vec{\nu}_k \rangle\mapsto t^*\left( \vec{\xi}, \vec{\nu}_1, \ldots, \vec{\nu}_k \right) $ which represents $ t^* $ in $ M_{\alpha} $ (as in lemma \ref{Lemma:  mu alpha appears in the sequence of its k alpha }). 

For simplicity, we adopt the following notation below: whenever $ \langle \vec{\xi}, \vec{\nu}_1, \ldots, \vec{\nu}_k \rangle $ are fixed, let $ h = h\left( \vec{\xi}, \vec{\nu}_1, \ldots, \vec{\nu}_k \right) $. Also, for every $ \nu < h $, we denote $ d^{-1}(\nu) = \langle  \nu^{1}, \ldots, \nu^{m-1} \rangle $ (whenever $ m\neq 1 $). We also denote $ \nu^0 = \nu $ and $ \vec{\nu} = \langle \nu^{0}, \ldots, \nu^{m-1} \rangle $. Let $ C = C\left( \vec{\xi},\vec{\nu}_1, \ldots, \vec{\nu}_k \right) $ be the club of closure points of $ \nu_0\mapsto g\left( \vec{\xi}, \vec{\nu}_1, \ldots, \vec{\nu}_k, \vec{\nu} \right)  $ (this is a club in $ h $. We remark that it is necessary in the proof below only in the case where $ j=1 $).

We now apply the Multivariable Fusion lemma. Fix $ \langle \vec{\xi}, \vec{\nu}_1, \ldots, \vec{\nu}_k \rangle $, and let--
\begin{align*}
e\left( \vec{\xi}, \vec{\nu}_1, \ldots, \vec{\nu}_k \right) = \{ & r\in P\setminus \nu_k \colon  \mbox{for every } \nu\in \lusim{A}^{r}_{ h }, \\  &A^{r}_{h} \setminus \nu \subseteq  \left(h\setminus \nu^{j-1}\right) \cap C\left( \vec{\xi}, \vec{\nu}_1, \ldots, \vec{\nu}_k \right) \}
\end{align*}

First let us consider the case where $ j>1 $. There exists $ p\in G $ such that for a set of  $ \xi<\kappa $ in $ W $, the condition $ p^{\frown} \langle \vec{\xi}, \vec{\mu}^{*}_{\alpha_1}(\xi), \ldots, \vec{\mu}^{*}_{\alpha_k}(\xi) \rangle$ forces that the element which appears after $\mu_{\alpha}(\xi)$ in the Prikry sequence of $ h\left( \vec{\xi}, \vec{\mu}^{*}_{\alpha_0}(\xi), \ldots, \vec{\mu}^{*}_{\alpha_k }(\xi) \right) $ is strictly greater then $ \mu^{*j-1}(\xi) $. Thus, in $ M\left[H\right] $,
$$ \lambda > \mu^{*j-1} > j_W\left(g\right)\left( \vec{\kappa}, \vec{\mu}^{*}_{\alpha_0}, \ldots , \vec{\mu}^{*}_{\alpha_k}, \vec{\mu}^{*}_{\alpha} \right) \geq \lambda $$
which is a contradiction.

If $ j=1 $, we use the club $ C $ defined above:  Since  $ \mu_{\alpha}< \lambda $, it follows that $ \lambda > j_W(g)\left( \vec{\kappa}, \vec{\mu}^{*}_{\alpha_1}, \ldots, \vec{\mu}^{*}_{\alpha_k}, \vec{\mu}^{*}_{\alpha} \right) $ which is again a contradiction as above.
\end{proof}

\begin{corollary}
In $ M\left[H\right] $, recall the sequence--
$$ d^{-1}\{\kappa\} = \langle \mu^{*1}_{0}, \ldots, \mu^{*m}_{0} \rangle  = \langle \left[Id\right]_{W^{1}}, \ldots, \left[Id\right]_{W^{m}} \rangle$$
For every $ 1\leq j\leq m $, the cardinal $ \mu^{*j} = \left[Id\right]_{W^j} $ is measurable in $ M $, and its Prikry sequence in $ M\left[H\right] $ is the sequence of critical points in the iterated ultrapower, $ \omega $-many times, taken with the measure $ U^{j-1}  = W^{j-1}\cap V\in V$.
\end{corollary}

Finally, let us provide a sufficient condition for definability of $ j_W\restriction_{V} $. Denote $ \mathcal{\vec{U}} = \langle \langle U^{0}_{\xi}, \ldots , U^{m_{\xi}-1}_{\xi} \rangle  \colon  \xi\in \Delta \rangle $.

\begin{lemma} \label{Lemma: FS, sufficient condition for definability}
If $ \mathcal{\vec{U}}\in V $, then $ j_W\restriction_{V} $  is definable in $ V $.
\end{lemma}

\begin{proof}
We prove by induction on $ \alpha\leq \kappa^* $ that $ j_{\alpha} \colon V\to M_{\alpha} $ is definable in $ V $. Fix $ \alpha<\kappa^* $ and assume that $ j_{\alpha} \colon V\to M_{\alpha} $ has been defined in $ V $. Let us define the measure $ \mathcal{E}_{\mu_{\alpha}} $. 
	
We use below the usual notations: for some $ \alpha_1<\ldots < \alpha_k <\alpha$  and $ h\in V $,
$$ \mu_{\alpha} = j_{\alpha}(h)\left(   j_{1,\alpha}\left(\left[Id\right]^{'}_{0}\right),   j_{\alpha_1+1, \alpha}\left(   \left[Id\right]_1 \right) , \ldots, j_{\alpha_k+1, \alpha}\left(  \left[Id\right]_k \right) \right) $$
(for sake of definability, we can use the least $ \langle \alpha_0, \ldots, \alpha_k \rangle $ and $ h $, taken with respect to a prescribed well order of $ V_\lambda $ for $ \lambda $ high enough).
For every $ \xi\in \Delta $, let $ \mathcal{E}(\xi) $ be the measure on $ \left[\xi\right]^{ m_{\xi}-1 } $ which corresponds to the sequence $$U^{0}_{\xi}  \vartriangleleft  U^{1}_{\xi} \vartriangleleft \ldots  \vartriangleleft U^{m_{\xi}-1}_{\xi} $$
Since $ \mathcal{U}  $ belongs to $ V $, the mapping  $\xi\mapsto  \mathcal{E}(\xi) $ belongs to $ V $ as well. By corollary \ref{Corollary: FS, how U^j_mualpha depends on U^j_xi}, for every $ \alpha<\kappa^* $,
$$  \mathcal{E}_{\mu_{\alpha}} =    j_\alpha \left( \langle \vec{\xi}, \vec{\nu}_1, \ldots , \vec{\nu}_k  \rangle \mapsto \mathcal{E}( h\left( \vec{\xi}, \vec{\nu}_1, \ldots , \vec{\nu}_k \right)  )   \right)\left(   j_{1,\alpha}\left(\left[Id\right]^{'}_{0}\right),   j_{\alpha_1+1, \alpha}\left(   \left[Id\right]_1 \right) , \ldots, j_{\alpha_k+1, \alpha}\left(  \left[Id\right]_k \right)   \right) $$
and this definition can be carried inside $ V $.
\end{proof}

\bibliography{FullSuppRestElm.bib}
\bibliographystyle{plain} 

\end{document}